\documentclass[a4paper,10pt]{amsart}
\usepackage{amsmath,amsthm,amssymb,enumerate,mathtools,stmaryrd}
\usepackage{epsfig}
\usepackage[utf8]{inputenc}
\usepackage{bm}
\usepackage{bbm}
\usepackage{esint}
\usepackage[T1]{fontenc}
\usepackage{tikz-cd}
\usetikzlibrary{decorations.pathreplacing}
\usepackage{amsfonts}
\usepackage{amsxtra}
\usepackage{euscript,mathrsfs}
\usepackage{color}
\usepackage[left=3.5cm,right=3.5cm,top=2.75cm,bottom=2.25cm]{geometry}
\usepackage[colorlinks=true, linktocpage=true, linkcolor=red!70!black, citecolor=green!50!black, urlcolor=black]{hyperref}
\allowdisplaybreaks
\usepackage{tikz}
\usepackage{diagbox,multirow}

\usepackage{cancel}
\usepackage[normalem]{ulem}
\usepackage[backgroundcolor=yellow, colorinlistoftodos,prependcaption,textsize=small,textwidth=25mm]{todonotes}

\usepackage{enumitem}
\setenumerate{label={\rm (\alph{*})}}

\numberwithin{equation}{section}

\makeatletter
\def\namedlabel#1#2{\begingroup
    #2%
    \def\@currentlabel{#2}%
    \phantomsection\label{#1}\endgroup
}
\makeatother

\newcommand{\dashint}{\fint}

\newcommand{\R}{\mathbb R}

\newcommand{\bd}{\mathrm{BD}}
 \newcommand{\dif}{\operatorname{d}\!}
 \newcommand{\locc}{\mathrm{loc}}
\newcommand{\lebe}{\operatorname{L}}
\newcommand{\sobo}{\operatorname{W}}
 \newcommand{\hold}{\operatorname{C}}
\newcommand{\ball}
{\operatorname{B}}
\newcommand{\A}{\mathbb{A}}

\newcommand{\rsym}{\mathbb{R}_{\mathrm{sym}}}
\newcommand{\dista}{\mathrm{dist}}
 \newcommand{\spt}{\operatorname{supp}}
\newcommand{\sym}{\operatorname{sym}}

\newcommand{\RR}{\mathbb{R}_{\mathrm{sym}}^{2\times 2}}

\newcommand{\E}{\mathbf{E}}

\newcommand{\di}{\mathrm{div}}
\newcommand{\sg}{\varepsilon}
\renewcommand{\E}{\mathrm{E}}
\newcommand{\uu}{\mathbf{u}}
\newcommand{\mres}{\mathbin{\vrule height 1.6ex depth 0pt width
0.13ex\vrule height 0.13ex depth 0pt width 1.3ex}}
\newcommand{\lebeweak}{\mathrm{L}_{\mathrm{w}^{*}}^{1}}
\newcommand{\tocean}{\bm{\tau}_{\mathrm{ocean}}}

\newcommand{\bv}{\mathrm{BV}}
\newcommand{\vv}{\mathbf{v}}
\newcommand{\ld}{\operatorname{LD}}
\newcommand{\bsigma}{\bm{\sigma}}
\newcommand{\ww}{\mathbf{w}}

\newcounter{alphasect}
\def\alphainsection{0}

\let\oldsection=\section
\def\section{%
  \ifnum\alphainsection=1%
    \addtocounter{alphasect}{1}
  \fi%
\oldsection}%
\renewcommand\thesection{%
  \ifnum\alphainsection=1%
    \Alph{alphasect}%
  \else%
    \arabic{section}%
  \fi%
}%

\allowdisplaybreaks

\theoremstyle{plain}

\newtheorem*{theorem*}{Theorem}
\newtheorem{theorem}{Theorem}[section]
\newtheorem{lemma}[theorem]{Lemma}
\newtheorem{proposition}[theorem]{Proposition}

\newtheorem{corollary}[theorem]{Corollary}

\newtheorem{definition}[theorem]{Definition}

\theoremstyle{remark}

\newtheorem{remark}[theorem]{Remark}

\definecolor{fg}{RGB}{34,139,34}  

\begin{document}


\title[On Hibler's model for sea ice in the singular limit]{On the singular limit in  Hibler's sea ice model}
\author[R. Denk]{Robert Denk}
\author[F. Gmeineder]{Franz Gmeineder}
\author[M. Hieber]{Matthias Hieber}
\address{Robert Denk: University of Konstanz, Department of Mathematics and Statistics, Universit\"{a}tsstra\ss e 10, 78464 Konstanz, Germany}
\email{robert.denk@uni-konstanz.de}
\address{Franz Gmeineder: University of Konstanz, Department of Mathematics and Statistics, Universit\"{a}tsstra\ss e 10, 78464 Konstanz, Germany}
\email{franz.gmeineder@uni-konstanz.de}
\address{Matthias Hieber: Technical University of Darmstadt, Department of  Mathematics, Schlossgartenstrasse 7, 64289 Darmstadt, Germany}
\email{hieber@mathematik.tu-darmstadt.de}
\begin{abstract}
We establish the existence of energy-driven solutions to the momentum balance equation in Hibler's sea ice model. As a main novelty and different from previous results, we deal with the singular limit and therefore cover the true unregularized Hibler stress. To this end, we introduce an energy-based notion of solution that is able to capture plasticity effects of sea ice. This requires certain relaxations of the Hibler energies and, by the different function space set-up, comes with novel challenges. In particular, we establish a bulk approximation result of the boundary terms in the evolutionary relaxed Hibler energies. This is achieved by developing a novel reduction scheme for nonlinear trace expressions which should be of independent interest. Finally, based on our main results, we classify our findings within a broader concept of solutions that is applicable to the non-constant mass case too. 

\end{abstract}


\date{\today}


\maketitle

\setcounter{tocdepth}{1}
\tableofcontents

\section{Introduction}
\subsection{Hibler's model}\label{sec:hiblerintro}
The mathematical description of sea ice is a fundamental problem that connects fluid and continuum mechanics. On the one hand, sea ice has features of plastic materials and therefore obeys limiting stress-strain relations. On the other hand, by its viscous behaviour, it equally shares aspects of fluids, and the modelling of its visco-plastic properties is an active area of research; see, e.g., Golden et al. 
\cite{Golden,GoldenEtAl}, Eisenman \& Wettlaufer \cite{EisemanWettlaeufer} and the references therein. 

One of the key models widely used in applications is due to Hibler \cite{Hibler}. This system and variations thereof have attracted large attention over the past decades; see, e.g., \cite{GubaLorenzSulsky,Brandt_2022,DingelDisser,Kimmrich,LiuTiti,Meh,MehlmannKorn21,MehlmannRichter2024} for an incomplete list. In view of rigorous existence theories, it is often customary to approximate the stress-strain relations which are typical for plastic materials by (at least locally) linearizable ones. This gives access to powerful machineries such as maximal regularity (see  \cite{Denk} and \cite{Brandt_2022}  and \cite{BBH25} for applications to sea ice) and is implemented in various models, e.g., the ICON modelling framework \cite{Icon,Icon2021}. Yet, it leads to formulations in spatial Sobolev spaces, whereby such techniques are well-suited to grasp plasticity effects of sea ice only to a limited extent. 

Here, instead, we study the momentum balance equation in the \emph{singular limit}, meaning that we do not work with regularized stresses and therefore account for plastic behaviour. In this regard, the main objective of the paper is to come up with a variant of Hibler's original momentum balance equation that allows to rigorously establish the existence of solutions in a suitable sense. Indeed, concerning the singular limit, even a concept of solutions is not known at present. In this  paper, we aim to make a significant first step in this direction and introduce both a notion of solutions  \emph{and} rigorously prove their existence. 

In view of the outline of our results in Section \ref{sec:outline}, we briefly summarise Hibler's original system from \cite{Hibler}. Let $\Omega\subset\R^{2}$ be open and bounded with Lipschitz boundary $\partial\Omega$, and let $T>0$. The motion of sea ice on $\Omega$ is primarily modelled via the key equation of \emph{momentum balance} for the horizontal ice velocity $\uu\colon\Omega_{T}\coloneqq (0,T)\times\Omega\to\R^{2}$. Denoting by $m$ the ice mass per unit area, the latter reads as 
\begin{align}\label{eq:hibler1}
m(\dot{\uu}+\uu\cdot\nabla \uu)=\mathrm{div}(\bm{\sigma})-mc_{\mathrm{Cor}}n\times \uu - mg\nabla H + \bm{\tau}_{\mathrm{atm}}+\bm{\tau}_{\mathrm{ocean}}(\uu)\qquad\text{in}\;\Omega_{T}. 
\end{align}
In \eqref{eq:hibler1}, $\dot{\uu}+\uu\cdot\nabla \uu $ represents the material derivative, whereas the expression involving $c_{\mathrm{Cor}}$ models the Coriolis force. For future reference, we note that by the typically small velocities of sea ice, the term  $\uu\cdot\nabla\uu$ of convective type, describing turbulence, as well as the Coriolis forces can be neglected. The momentum balance equation \eqref{eq:hibler1} then becomes 
\begin{align}\label{eq:hiblerMAIN1}
m\dot{\uu}=\mathrm{div}(\bm{\sigma}) - mg\nabla H + \bm{\tau}_{\mathrm{atm}}+\bm{\tau}_{\mathrm{ocean}}(\uu)\qquad\text{in}\;\Omega_{T}. 
\end{align}
The term $mg\nabla H$, in turn, incorporates the gravity $g$ and the sea surface dynamic height $H\colon\Omega\to\R$. Finally,  $\bm{\tau}_{\mathrm{atm}}$ represents air or atmospheric forces such as winds, whereas $\bm{\tau}_{\mathrm{ocean}}$ models water or oceanic forces in interaction with sea ice. More precisely, for suitable constants and rotations $R_{\mathrm{atm}},R_{\mathrm{ocean}}\in\mathrm{SO}(2)$, these force terms are given by
\begin{align}
&\bm{\tau}_{\mathrm{atm}} = \rho_{\mathrm{atm}}C_{\mathrm{atm}}|\mathbf{U}_{\mathrm{atm}}|R_{\mathrm{atm}}\mathbf{U}_{\mathrm{atm}},\label{eq:hiblo1} \\ 
&\bm{\tau}_{\mathrm{ocean}}(\uu) = \rho_{\mathrm{ocean}}C_{\mathrm{ocean}}|\mathbf{U}_{\mathrm{ocean}}-\uu|R_{\mathrm{ocean}}(\mathbf{U}_{\mathrm{ocean}}-\mathbf{u}). \label{eq:hiblo2}
\end{align}
This particular structure is typical for wind and oceanic force terms. Since, in real world scenarios,  the horizontal sea ice velocities are bounded, it is moreover customary to  consider modifications of such terms regarding their growth behaviour. 

In \eqref{eq:hibler1}, however, the key object is the viscoplastic \emph{stress} $\bsigma$. For sea ice, $\bsigma$ is a highly degenerate nonlinear expression involving the Hibler deformation tensor $\mathbb{T}\uu$, the pressure $P$, masses as well as the \emph{mean ice thickness} $h$ and the \emph{ice compactness} $a$. On the one hand, it is $\bsigma$ which captures plasticity effects. On the other hand, this is the key term that causes complications  when trying to come up with a notion of solution to \eqref{eq:hibler1}. Subject to natural hypotheses, the stresses $\bsigma$ shall be addressed in detail in Section \ref{sec:driving} below.

In \cite{Hibler}, the momentum balance equation \eqref{eq:hibler1} is  augmented with transport equations for the mean ice thickness and ice compactness, reading as 
\begin{align}\label{eq:hibler2}
\begin{split}
\dot{h} + \mathrm{div}(h\uu)=S_{h},\\ 
\dot{a} + \mathrm{div}(a\uu)=S_{a}, 
\end{split}
 \end{align}
 where $S_{h},S_{a}$ are certain non-linear expressions in $h$ and $a$. Lastly, subject to initial data $\uu_{0}$, sea ice is supposed to rest at the boundary $\partial\Omega$. Specifically, the coupled system (\eqref{eq:hibler1},\eqref{eq:hibler2}) or ($\eqref{eq:hiblerMAIN1},\eqref{eq:hibler2}$), respectively, is augmented with zero Dirichlet boundary conditions for $\uu$ and Neumann boundary conditions for $h$ and $a$. 

 In the present paper, we exclusively focus on the momentum balance equation \eqref{eq:hiblerMAIN1}. Most notably, however, we deal with the \emph{true} stress terms \emph{without regularizations}. In particular, we overcome the key difficulties concerning the nonlinear expressions, and establish both notions and existence results that shall be extended towards \eqref{eq:hibler2} in future work. 
\subsection{Main results, notions of solutions and outline of the paper}\label{sec:outline}
We now come to a description of our main results and the outline of the paper. Focusing on \eqref{eq:hiblerMAIN1} subject to a constant mass hypothesis, one may view \eqref{eq:hiblerMAIN1} as a perturbed gradient flow for  Hibler's viscous plastic energy. Incorporating plasticity phenomena, in turn, requires suitable relaxations of the underlying energies to the space $\bd(\Omega)$ of functions of \emph{bounded deformation}. Postponing the precise assumptions and notions of solutions to Section \ref{sec:main}, we formulate the main result of the present paper as a metatheorem:  
\begin{itemize}
\item \textbf{Main result} (Theorems \ref{thm:main} and \ref{thm:relaxedevoleq}). \emph{For the motion of sea ice and subject to constant mass, there is a \emph{natural energy-driven notion of solutions} to the momentum balance equation in the \emph{singular limit}. Such solutions exist on the entire time interval and for all natural initial data. Lastly, they can be obtained as limits of solutions to regularized problems, and especially satisfy the momentum balance equation in a \emph{weak sense}.} 
\end{itemize}
In view of this result, we emphasize three key points:\\ 

(i) \textbf{Plasticity effects.} As a main novelty and different from previous contributions (see, e.g.,  \cite{Brandt_2022,Binz_2023}), our existence result does \emph{not require} a stress regularization. In view of natural plasticity effects such as hardening or the formation of cracks in sea ice, the latter comes with different obstructions towards an existence theory. This shall be discussed in detail in Section \ref{sec:modified} below, and  we highlight two chief points:  First, the incorporation of plasticity effects necessitates a function space set-up which differs from basically all previous existence results. In particular, our main result (Theorem \ref{thm:main}) is based on the space $\bd(\Omega)$ of functions of bounded deformation rather than Sobolev spaces; see Section \ref{sec:prelims} for its precise definition. This is natural, since spatial Sobolev regularity (and thus absolute continuity of the gradients for $\mathscr{L}^{2})$ excludes the description of cracks. The latter might suggest to study our main equation on $\bv(\Omega;\R^{2})$, the $\R^{2}$-valued fields of bounded variation, see, e.g., \cite{AFP,EvansGariepy}. However, in this case, the special form of the Hibler viscous-plastic stress (see Section \ref{sec:driving}) does not admit non-trivial $\lebe^{1}$-estimates. This means that the natural energies driving the sea ice evolution \emph{do not suffice} to obtain control of the full velocity gradients in $\lebe^{1}(\Omega;\R^{2\times 2})$ or the $\R^{2\times 2}$-valued Radon measures, respectively.
This is due to a basic obstruction from harmonic analysis known as Ornstein's Non-Inequality \cite{Ornstein}, see Remark \ref{rem:OrnsteinBV}. In particular, $\bv$-regularity of the horizontal sea ice velocity cannot be expected. By the structure of the Hibler stress, many of the statements below can be formulated more abstractly in the recently established framework of functions of bounded $\mathbb{A}$-variation \cite{BDG}, even though we view the latter primarily as an auxiliary tool. 

(ii) \textbf{Relaxations and lower order energy concentrations.} The space $\bd(\Omega)$ allows to admit velocity fields whose symmetric gradients are $\RR$-valued finite Radon measures. Since we aim to allow for potential cracks of sea ice, this particularly forces us to work with  singular measures with respect to $\mathscr{L}^{2}$. In consequence, this must be reflected in our notion of solution for the modified momentum balance equation. Here we take advantage of the gradient flow formulation or energy-driven evolution. In our situation, the energies are defined in terms of suitable relaxations to the space $\bd(\Omega)$, to be introduced and studied in Section \ref{sec:relaxthebulk}. Such \emph{relaxations} are extensions of the energy functionals which allow for singularity with respect to $\mathscr{L}^{2}$ of the Hibler differential expressions. The latter are conceptually similar to the more classical total variation functionals on $\bv$, and the underlying gradient flows are strongly related to the total variation or mean curvature flow problems on $\bv$ (see, e.g., \cite{ABCM01a,ABCM01b,ABCM02,HZ,LichnewskyTemam}). Yet, the vectorial structure of the problems considered here, but also the different function space set-up comes with novel challenges. One such instance, which is at the heart of the existence proof \emph{under a constant mass hypothesis}, is the limiting behaviour of bulk integrals close to the boundary; see Theorem \ref{thm:bdryapprox} and Remarks \ref{rem:bulk1}--\ref{rem:bulk2} and \ref{rem:bulk3}.

This result, which deals with energy concentrations close to the boundary, should be of independent interest. More precisely, it establishes the principle that \emph{lower order energy concentrations} along Lipschitz hypersurfaces are not affected by Ornstein's Non-Inequality. In particular, it represents the first \emph{non-linear variant} of the metaprinciple from \cite{DieningGmeineder} that Ornstein's Non-Inequality becomes invisible in the context of trace estimates. 

(iii) \textbf{Notion of solutions.} As mentioned above, in this paper we largely work subject to a constant mass hypothesis. At least for short times, this hypothesis  is well-justified from a physical viewpoint, and admits to set up a variational formulation in the spirit of gradient flows for the relaxed viscous-plastic Hibler energies; see Definition \ref{def:varsol1} for the precise notion. Our main  existence result, Theorem \ref{thm:main}, also implies the existence of the canonical \emph{weak solutions} to the momentum balance equation. By this, we understand a relaxed version of \eqref{eq:hiblerMAIN1}, and seems to be novel in the context of sea ice too. 

The existence proof relies on several limit passages for approximate problems, and it is here where the bulk approximation enters crucially. By its very structure, it is particularly accessible for future numerical studies; see, e.g., \cite{Meh,MehlmannKorn21} for available numerical methods. 

With this being the primary motivation of the present work, we conclude the paper in Section \ref{sec:weakvar} by embedding the concept of variational solutions into a broader class that allows for non-constant mass too. To this end, we introduce the notion of \emph{weak-variational solutions} (WVS), which overcomes the deficit of  variational structure in the variable mass case. The precise analysis of the latter shall be the content of future work, in turn strongly relying on the foundational case considered here. To keep the paper at a reasonable length, we discuss their connections to the variational solutions from Definition \ref{def:varsol1} in a simplified scenario. This shall form the basis for future works.

{\small\subsection*{Acknowledgments} 
Franz Gmeineder acknowledges  financial support through the Hector foundation (Project Number FP 626/21). The research of Matthias Hieber is supported by DFG Research Unit FOR 5528 on Geophysical Flows.}

\section{On the momentum balance equation and modelling of stresses}\label{sec:modified}
In this section, we discuss the underlying assumptions that motivate the study of the true Hibler stress terms $\bsigma$ in \eqref{eq:hiblerMAIN1} and natural variants thereof. Based on their specific form, we moreover discuss some of the aforementioned obstructions on a more technical level. 
\subsection{Hibler deformations, stresses and energies}\label{sec:driving}
In order to motivate the specific formulation of the sea ice model in Section \ref{sec:varsol} below, we first identify the driving energy term. In the following, let $T>0$ and suppose that $\Omega\subset\R^{2}$ is an open and bounded Lipschitz domain. For a sufficiently regular map $\mathbf{v}\colon \Omega\to\R^{2}$, we define its symmetric or trace-free symmetric gradient, respectively, by
\begin{align}\label{eq:keydiffops}
\varepsilon\mathbf{v}\coloneqq\frac{1}{2}(\nabla\mathbf{v}+\nabla\mathbf{v}^{\top}),\;\;\;\varepsilon^{D}\mathbf{v}\coloneqq\varepsilon\mathbf{v} - \frac{1}{2}\mathrm{tr}(\nabla\mathbf{v})\mathbbm{1}_{2\times 2}
\end{align}
with the $(2\times 2)$-unit matrix $\mathbbm{1}_{2\times 2}\in\R^{2\times 2}$. Whenever we write $\varepsilon$ in the sequel, we tacitly understand $\varepsilon=(\varepsilon_{ij})_{1\leq i,j\leq 2}=\varepsilon\mathbf{v}$ and analogously for $\varepsilon^{D}$; in longer expressions, we also write $\sg(\vv)$ and $\sg^{D}(\vv)$. For a given ice pressure function $P$ the corresponding nonlinear  \emph{bulk} and \emph{shear viscosities} via 
\begin{align}\label{eq:identity1}
\zeta(\varepsilon,P):=\frac{P}{2\triangle(\varepsilon)}\;\;\;\text{and}\;\;\;\eta(\varepsilon,P):=\frac{1}{e^{2}}\zeta(\varepsilon,P). 
\end{align}
The emerging terms $\triangle(\varepsilon)$ and $e$ stem from the modelling of the viscous-plastic stress $\bsigma$ through a constitutive law involving the deformation tensor $\sg$. More precisely, as in \cite{Hibler} and as is typical in the description of plasticity phenomena, we suppose that the principal stresses belong to an elliptical yield curve. In this situation, the quantity $e$ is the ratio between the major and minor principal axes of the yield curve, and $\triangle(\varepsilon)$ is defined by
\begin{align}\label{eq:triangledef}
\triangle(\varepsilon):=\Big((\varepsilon_{11}^{2}+\varepsilon_{22}^{2})\Big(1+\frac{1}{e^{2}}\Big) + \frac{4}{e^{2}}\varepsilon_{12}^{2} + 2\varepsilon_{11}\varepsilon_{22}\Big(1-\frac{1}{e^{2}}\Big)\Big)^{\frac{1}{2}}. 
\end{align}
In order to model sea ice as an idealised visco-plastic fluid, we consider stresses of the form
\begin{align}\label{eq:stress}
\begin{split}
\bsigma &=\frac{2}{e^{2}}\zeta(\varepsilon,P)\varepsilon + [\zeta(\varepsilon,P)-\eta(\varepsilon,P)]\mathrm{tr}(\varepsilon)\mathbbm{1}_{2\times 2} - \frac{P}{2}\mathbbm{1}_{2\times 2} \\
& \!\!\!\!\stackrel{\eqref{eq:identity1}_{2}}{=}\frac{2}{e^{2}}\zeta(\varepsilon,P)\varepsilon + \Big(\Big(1-\frac{1}{e^{2}}\Big)\zeta(\varepsilon,P)\mathrm{tr}(\varepsilon) - \frac{P}{2}\Big)\mathbbm{1}_{2\times 2} \\
& = \frac{2}{e^{2}}\zeta(\varepsilon,P)\varepsilon^{D} + \Big(\zeta(\varepsilon,P)\mathrm{tr}(\varepsilon) - \frac{P}{2}\Big)\mathbbm{1}_{2\times 2} \\
& \!\!\!\!\stackrel{\eqref{eq:identity1}_{1}}{=}\frac{2}{e^{2}}\zeta(\varepsilon,P)\varepsilon^{D} +  \zeta(\varepsilon,P)\Big(\mathrm{tr}(\varepsilon) -\triangle(\varepsilon)\Big)\mathbbm{1}_{2\times 2}.
\end{split}
\end{align}
This constitutive law  particularly reflects the fact that, for small values of  $\triangle(\varepsilon)$, singularities in the viscosities might occur and plasticity effects enter. 

Subject to the additional {constant pressure assumption} $P=\mathrm{const.}$, we now let $\bm{\varphi}\in\hold_{c}^{\infty}(\Omega;\R^{2})$ be arbitrary but fixed. In the following, we crucially employ that, with respect to the usual Hilbert-Schmidt inner product on $\R^{2\times 2}$, the trace-free symmetric matrices are orthogonal to $\R\mathbbm{1}_{2\times 2}$. We then have, with $\langle\cdot,\cdot\rangle$ denoting the usual integral pairing, 
\begin{align}\label{eq:sigmacomp}
\begin{split}
 -\langle\di(\bsigma),\bm{\varphi}\rangle & =  \langle \bsigma,\varepsilon\bm{\varphi}\rangle \\
& = \left\langle \frac{2}{e^{2}}\zeta(\varepsilon,P)\Big(\varepsilon^{D}\mathbf{v} +  \frac{e^{2}}{2}\mathrm{tr}(\varepsilon\mathbf{v}) \mathbbm{1}_{2\times 2}\Big),\varepsilon\bm{\varphi}\right\rangle\\
& = \left\langle \frac{2}{e^{2}}\zeta(\varepsilon,P)\Big(\varepsilon^{D}\mathbf{v} +  \frac{e^{2}}{2}\mathrm{tr}(\varepsilon\mathbf{v}) \mathbbm{1}_{2\times 2}\Big),\Big(\varepsilon^{D}\bm{\varphi}+\frac{1}{2}\mathrm{tr}(\varepsilon\bm{\varphi})\mathbbm{1}_{2\times 2}\Big)\right\rangle \\
& = \left\langle\frac{2}{e^{2}}\zeta(\varepsilon,P)\varepsilon^{D}\mathbf{v},\varepsilon^{D}\bm{\varphi}\right\rangle +  \zeta(\varepsilon,P)\mathrm{tr}(\varepsilon\mathbf{v})\mathrm{tr}(\varepsilon\bm{\varphi}).
\end{split}
\end{align}
Let us define, for $\alpha>0$, 
\begin{align}\label{eq:Tdefine}
\mathbb{T}\mathbf{v} := \frac{\sqrt{2}}{e}\varepsilon^{D}\mathbf{v}+\alpha\mathrm{tr}(\varepsilon\mathbf{v})\mathbbm{1}_{2\times 2}.
\end{align}
We compute
\begin{align}\label{eq:Tcomp}
\begin{split}
\langle -\mathbb{T}^{*}(\zeta(\varepsilon,P)\mathbb{T}\mathbf{v}),\bm{\varphi}\rangle & = \langle \zeta(\varepsilon,P)\mathbb{T}\mathbf{v},\mathbb{T}\bm{\varphi}\rangle \\
& = \left\langle \frac{2}{e^{2}}\zeta(\varepsilon,P)\varepsilon^{D}\mathbf{v},\varepsilon^{D}\bm{\varphi}\right\rangle  + 2\alpha^{2}\langle\zeta(\varepsilon,P)\mathrm{tr}(\varepsilon\mathbf{v}),\mathrm{tr}(\varepsilon\bm{\varphi})\rangle,
\end{split}
\end{align}
whereby setting $\alpha=\frac{1}{\sqrt{2}}$ in \eqref{eq:Tdefine} yields
\begin{align}\label{eq:centralstressidentity}
\langle -\di(\bsigma),\bm{\varphi}\rangle = \langle -\mathbb{T}^{*}(\zeta(\varepsilon,P)\mathbb{T}\mathbf{v}),\bm{\varphi}\rangle
\end{align}
in view of \eqref{eq:sigmacomp} and \eqref{eq:Tcomp}. The canonical choice of $\mathbb{T}$ in this situation therefore is
\begin{align}\label{eq:Tdefine1}
\mathbb{T}\mathbf{v} \coloneqq \frac{\sqrt{2}}{e}\varepsilon^{D}\mathbf{v}+\frac{1}{\sqrt{2}}\mathrm{tr}(\varepsilon\mathbf{v})\mathbbm{1}_{2\times 2}, 
\end{align}
and we call $\mathbb{T}\vv$ the \emph{Hibler deformation tensor} or simply \emph{Hibler deformation}. The nonlinear operator $\mathbb{T}^{*}(\zeta(\varepsilon,P)\mathbb{T}\mathbf{v})$, in turn, is a generalised divergence-form operator, and shall be referred to as the \emph{Hibler operator}. This operator agrees with variants considered previously in \cite{Hibler} or \cite{Brandt_2022}.  

In order to derive the specific form of the underlying Hibler energy, we confirm that $|\mathbb{T}\vv|=\triangle(\varepsilon)$ with $\triangle(\varepsilon)$ as in \eqref{eq:triangledef}. We write 
\begin{align}\label{eq:hiblerwritedown}
\mathbb{T}\vv = \left(\begin{matrix} \Big(\frac{1}{\sqrt{2}e}+\frac{1}{\sqrt{2}}\Big)\varepsilon_{11} + \Big(\frac{1}{\sqrt{2}}-\frac{1}{\sqrt{2}e}\Big)\varepsilon_{22} & \frac{\sqrt{2}}{e}\varepsilon_{12} \\ \frac{\sqrt{2}}{e}\varepsilon_{12} & \Big(\frac{1}{\sqrt{2}}-\frac{1}{\sqrt{2}e} \Big)\varepsilon_{11}+\Big(\frac{1}{\sqrt{2}}+\frac{1}{\sqrt{2}e} \Big)\varepsilon_{22}
\end{matrix} \right), 
\end{align}
whereby we arrive at 
\begin{align*}
|\mathbb{T}\vv|^{2} & = \Big(\frac{1}{\sqrt{2}e}+\frac{1}{\sqrt{2}}\Big)^{2}\varepsilon_{11}^{2} + 2\Big(\frac{1}{\sqrt{2}e}+\frac{1}{\sqrt{2}}\Big)\Big(\frac{1}{\sqrt{2}}-\frac{1}{\sqrt{2}e}\Big)\varepsilon_{11}\varepsilon_{22} + \Big(\frac{1}{\sqrt{2}}-\frac{1}{\sqrt{2}e}\Big)^{2}\varepsilon_{22}^{2} \\ 
& + \frac{4}{e^{2}}\varepsilon_{12}^{2} \\ 
& + \Big(\frac{1}{\sqrt{2}}-\frac{1}{\sqrt{2}e}\Big)^{2}\varepsilon_{11}^{2} + 2\Big(\frac{1}{\sqrt{2}}-\frac{1}{\sqrt{2}e}\Big)\Big(\frac{1}{\sqrt{2}}+\frac{1}{\sqrt{2}e}\Big)\varepsilon_{11}\varepsilon_{22} + \Big(\frac{1}{\sqrt{2}}+\frac{1}{\sqrt{2}e}\Big)^{2}\varepsilon_{22}^{2} \\ 
& = \Big(1+\frac{1}{e^{2}}\Big)(\varepsilon_{11}^{2}+\varepsilon_{22}^{2}) + 2\Big(1-\frac{1}{e^{2}}\Big)\varepsilon_{11}\varepsilon_{22} + \frac{4}{e^{2}}\varepsilon_{12}^{2} = \triangle(\varepsilon)^{2}. 
\end{align*}
In consequence, the generalized divergence form operator from \eqref{eq:centralstressidentity} becomes 
\begin{align}\label{eq:concreteHibler}
\mathbb{T}^{*}(\zeta(\varepsilon,P)\mathbb{T}\vv) & = \mathbb{T}^{*}\Big(\frac{P}{2\triangle(\varepsilon)}\mathbb{T}\vv) = \mathbb{T}^{*}\Big(\frac{P}{2}\frac{\mathbb{T}\vv}{|\mathbb{T}\vv|} \Big). 
\end{align}
It is then clear that the canonical energy leading to \eqref{eq:concreteHibler} is given by 
\begin{align}\label{eq:hiblerenergy1}
\mathscr{F}[\vv;\Omega] \coloneqq \int_{\Omega}\frac{P}{2}|\mathbb{T}\vv|\dif x,  
\end{align}
and \eqref{eq:hiblerenergy1} shall be referred to as \emph{Hibler-type energy}. For future reference, we note that, a priori, \eqref{eq:hiblerenergy1} only makes sense if $\mathbb{T}\vv$ is an integrable $\RR$-valued map. 

The Hibler deformation tensor $\mathbb{T}$ from \eqref{eq:concreteHibler} has a more difficult structure than the usual symmetric gradient. For the following, it is thus useful to record that there exists $1<c<\infty$ such that
\begin{align}\label{eq:pointwisecompa}
\frac{1}{c}|z^{\mathrm{sym}}| \leq |T[z^{\sym}]| \leq c|z^{\mathrm{sym}}|
\end{align}
holds for all $z\in\R^{2\times 2}$, where we have put
\begin{align}\label{eq:Tdef}
T[z] = \frac{\sqrt{2}}{e}z^{D} + \frac{1}{\sqrt{2}}\mathrm{tr}(z)\mathbbm{1}_{2\times 2}
\end{align}
with the deviatoric part $z^{D}\coloneqq z-\frac{1}{2}\mathrm{tr}(z)\mathbbm{1}_{2\times 2}$. Even though elementary, \eqref{eq:pointwisecompa} is essential for the sequel, and so we briefly pause to give the argument: Put 
\begin{align*}
\bm{\xi}^{(1)}\coloneqq \left(\begin{matrix} 0 & 1 \\ 1 & 0 \end{matrix}\right),\;\;\;\bm{\xi}^{(2)}\coloneqq \left(\begin{matrix} -1 & 0 \\ 0 & 1 \end{matrix}\right)\;\;\;\text{and}\;\;\;\bm{\xi}^{(3)}\coloneqq \mathbbm{1}_{2\times 2} =\left(\begin{matrix} 1 & 0 \\ 0 & 1 \end{matrix} \right). 
\end{align*}
An arbitrary element $\mathbf{z}\in\RR$ can be written, for $\alpha,\beta,\gamma\in\R$, as 
\begin{align*}
\mathbf{z} \coloneqq \left(\begin{matrix} \beta & \alpha \\ \alpha & \gamma \end{matrix} \right) \;\;\;\text{and so}\;\;\; T[\mathbf{z}] = \frac{\sqrt{2}}{e}\alpha \bm{\xi}^{(1)} + \frac{1}{\sqrt{2}e}(\beta-\gamma)\bm{\xi}^{(2)} + \frac{1}{\sqrt{2}}(\beta+\gamma)\bm{\xi}^{(3)}.
\end{align*}
In particular, because the matrices $\frac{\sqrt{2}}{e}\bm{\xi}^{(1)},\frac{1}{\sqrt{2}e}\bm{\xi}^{(2)}$ and $\frac{1}{\sqrt{2}}\bm{\xi}^{(3)}$ are linearly independent, $\mathbf{z}\mapsto |T[\mathbf{z}]|$ is a norm on $\RR$. Since all norms on $\RR$ are equivalent, we arrive at the lower bound of \eqref{eq:pointwisecompa}. The upper bound is trivial, and so \eqref{eq:pointwisecompa} follows.

\subsection{The momentum balance equation} 
Based on the specific Hibler stresses $\bsigma$ and the Hibler deformation tensors $\mathbb{T}\vv$, our primary focus is the momentum balance equation \eqref{eq:hiblerMAIN1}. The latter can be written in the form 
\begin{align}\label{eq:hibler1a}
\begin{cases}
\partial_{t}{\uu} = \mathbb{T}^{*}(\bm{\sigma}) + \mathbf{f} + \bm{\tau}_{\mathrm{ocean}}(\uu)&\;\text{in}\;(0,T)\times\Omega,\\ 
\uu=0&\;\text{on}\;(0,T)\times\partial\Omega,\\ 
\uu(0,\cdot) = \uu_{0}&\;\text{on}\;\{0\}\times\Omega, 
\end{cases}
\end{align}
where we write $\mathbf{f}$ for the sum of gravitational and atmospheric forces. 

As mentioned above, we shall largely work under a \emph{constant mass} and \emph{constant pressure} assumption. In principle and as long they are given, they can be incorporated as coefficients in the Hibler energy, see Remark \ref{rem:xdependence}. Following the modelling from Section \ref{sec:modelling},  we allow $\bm{\sigma}$ to be \emph{singular}, see also \eqref{eq:truestress} below. In essence, this means that the problem \eqref{eq:hibler1a} cannot be treated by linearization techniques. Since we shall approximate the true singular Hibler stress by stabilized  stresses and study their limiting behaviour, we speak of the \emph{singular limit}. To this end, \eqref{eq:hibler1a} shall be treated by energy-based methods, giving us access to variational techniques.

To keep the paper at a reasonable length, we moreover consider the ocean forces with cut-off, meaning that 
\begin{align}\label{eq:cutoff} 
\bm{\tau}_{\mathrm{ocean}}(\uu) \coloneqq \rho_{\mathrm{ocean}}C_{\mathrm{ocean}}\eta(|\mathbf{U}_{\mathrm{ocean}}-\uu|)R_{\mathrm{ocean}}(\mathbf{U}_{\mathrm{ocean}}-\uu), 
\end{align}
where $\eta\colon\R_{\geq 0}\to \R_{>0}$ is a function which satisfies for some fixed $\gamma\in(0,1)$
\begin{align}\label{eq:simpleeta}
\eta(0)=0,\;\;\;\eta\;\text{is linear on some interval $[0,N]$}\;\text{and}\;\eta(s)=s^{-\gamma}\;\text{for all large}\;s.
\end{align}
The reason for this modification  is as follows. We recall that our primary focus is on the singular limit. As shall be discussed in Section \ref{sec:modelling} below, this a priori leads to velocities of spatial $\bd$-regularity. By the Sobolev-Strauss embedding on domains \cite{Strauss} (see Lemma \ref{lem:poincaresobolev} and also \cite{GmRa}), our solutions are a priori of spatial $\lebe^{2}$-integrability in $n=2$ dimensions. When deriving useful energy estimates, this suggests that \eqref{eq:cutoff} should be at most of linear growth at infinity. To strike a balance, we do not fully neglect this term, but employ a cut-off; note that \eqref{eq:cutoff} still carries the typical structure of oceanic force terms. 

We believe that this can be improved, but requires different tools on the higher integrability of solutions. This is a point where our theory heavily departs from related full gradient scenarios. For instance,  tools such as maximum principles or parabolic Moser-type iteration strategies are neither available nor can be expected; see Remark \ref{rem:criticalcutoff} for more detail. Yet, since the cut-off incorporates some of the underlying ocean-driven effects, this might still be sufficient to reflect such effects in future numerical simulations.

\subsection{The true Hibler stress and plasticity effects}\label{sec:modelling}
We now address the stress terms $\bsigma$ from \eqref{eq:hibler1a} from a mathematical perspective. By our discussion in Section \ref{sec:driving}, the stress terms are typically not linearizable. Thus, they allow for the incorporation of plasticity phenomena which are largely absent in previously studied \emph{regularised} models. Most notably, {Brandt} et al. \cite{Brandt_2022} and {Liu} et al. \cite{LiuTiti} consider stress-strain relations  roughly of the form 
	\begin{align}\label{eq:stressstrain}
		\bm{\sigma} \approx  \frac{\mathbb{T}\mathbf{u}}{\sqrt{\varepsilon +|\mathbb{T}\mathbf{u}|^{2}}}, 
	\end{align}
	where $\mathbb{T}\uu$ is Hibler's deformation tensor from \eqref{eq:Tdefine1}. From an energy perspective, stress terms such as \eqref{eq:stressstrain} arise from energy functionals with elliptic $\hold^{2}$-integrands. For instance, \eqref{eq:stressstrain} stems from the energy density  $F_{\varepsilon}(\mathbb{T}\uu)\coloneqq\sqrt{\varepsilon+|\mathbb{T}\uu|^{2}}$. The smoothness and strict convexity of $F_{\varepsilon}$ allows to linearize the accordingly modified Hibler model and to employ the powerful machinery of maximal regularity, see, e.g., \cite{Denk}. By its elliptic nature, the regularisation \eqref{eq:stressstrain} comes with smoothing effects and spatial $\sobo^{1,p}$-regularity, $p>1$. 
	
As one the of the key points of the present paper, when modelling sea ice, this is \emph{not} what is to be expected. Indeed, sea ice is a typical instance  where fractures can occur. Such plasticity effects are fully excluded in the presence of \emph{any} result that entails the spatial $\sobo^{1,1}$-regularity of solutions; in other words, posing any model of sea ice on Sobolev-type spaces \emph{cannot account for} plasticity effects. This is a routine insight dating back to, e.g., \cite{AnzGia,FuchsSeregin,Suq,TemamStrang}, and is reflected by \emph{true} stress terms roughly of the form 
	\begin{align}\label{eq:truestress}
	\bm{\sigma} \approx \frac{\mathbb{T}\uu}{|\mathbb{T}\mathbf{u}|}, 
	\end{align}
see \eqref{eq:stress}ff.. This stress term corresponds to the Lipschitz, yet non-differentiable energy density $F(\mathbb{T}\uu)\coloneqq|\mathbb{T}\uu|$. 

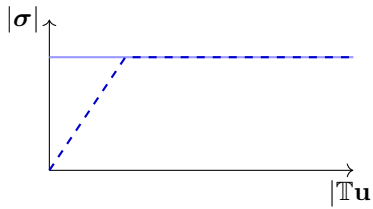
\begin{figure}
\begin{tikzpicture}
\draw[->] (0,0) -- (4,0);
\draw[->] (0,0) -- (0,2); 
\node[left] at (0,2) {$|\bm{\sigma}|$};
\draw[blue!40!white,thick] (0,1.5) -- (4,1.5);
\draw[blue!80!black,dashed,thick] (0,0) -- (1,1.5) -- (4,1.5);
\node[below] at (4,0){$|\mathbb{T}\uu|$};
\end{tikzpicture}
\caption{Two variants of stress-strain relations based on the Hibler deformation tensor $\mathbb{T}\uu$, see \eqref{eq:hiblerwritedown}. The light blue line corresponds to the standard Hibler stress, whereas the dashed blue line corresponds to a modified Hibler stress \`{a} la Mohr-Coulomb.}\label{fig:stressstrain}
    \end{figure}

From a modelling perspective, modifications of \eqref{eq:truestress} are equally reasonable. For instance, we might consider constitutive relations which reflect elastic behaviour for small Hibler deformations and plastic behaviour for large Hibler deformations. Conceptually, this is similar to the Drucker-Prager or Mohr-Coloumb models for rocks or soils, see  \cite{DruckerPrager,Lubliner}. In particular, the latter leads to stresses $\bm{\sigma}$ which are linear in $|\mathbb{T}\uu|$ close to zero, whereas they coincide with \eqref{eq:truestress} for large values of $|\mathbb{T}\uu|$. This conceptual difference is visible from the underlying stress-strain diagrams displayed in Figure \ref{fig:stressstrain}. To provide a unifying perspective on such models, it is thus reasonable to assume stress-strain relations of the form $\bm{\sigma}=F(\mathbb{T}\uu)$, where the potential $F\colon\RR\to\R$ has \emph{linear growth}. By this we understand that there exist constants $c_{1},c_{3}>0$ and $c_{2}\in\R$ such that 
    \begin{align}\label{eq:lingrowthmodelling}
    c_{1}|z|-c_{2} \leq F(z) \leq c_{3}(1+|z|)\qquad\text{for all}\;z\in\RR. 
    \end{align}
    Based on the above stress-strain relation, this leads to energies of the form 
    \begin{align}\label{eq:Hiblerenergymodel}
    \mathscr{F}[\uu;\Omega] \coloneqq \int_{\Omega}F(\mathbb{T}\uu)\dif x,\qquad \uu\colon\Omega\to\R^{2}. 
    \end{align}
    In order to incorporate  plasticity effects of sea ice, it is moreover clear that the distributional expression $\mathbb{T}\uu$ should be allowed to be a finite $\RR$-valued measure which may be singular with respect to $\mathscr{L}^{2}$, in formulas $\mathbb{T}\uu\in\mathrm{RM}_{\mathrm{fin}}(\Omega;\RR)$. Whereas this describes the physical perspective, measures are required from a compactness viewpoint to deal with energy concentrations on small sets.
    
   In turn, if the distributional expression $\mathbb{T}\uu$ is a measure, then it is a priori unclear how to interpret the stress (e.g., \eqref{eq:truestress}) or the associated Hibler energies \eqref{eq:Hiblerenergymodel}. The latter is, in some sense, easier to handle. Still, it  requires an extension of the original Hibler energy by lower semicontinuity, letting us pass to the limits in the critical nonlinear expressions. 

    This is achieved by \emph{relaxations}. By the latter, we mean extensions of the Hibler energies to spaces where compactness is available while maintaining  lower semicontinuity properties. As we discuss as a consequence of Reshetnyak's lower semicontinuity theorem  \cite{Reshetnyak} (see also Goffman \& Serrin \cite{GoffmanSerrin}) in Section \ref{sec:relaxthebulk}, this leads to relaxed energies of the form \begin{align}\label{eq:relaxintro}
		\begin{split}
	\mathscr{F}_{0}^{*}[\uu;\Omega]  \coloneqq 	\int_{\Omega}F\Big(\frac{\dif\mathbb{T}\uu}{\dif\mathscr{L}^{2}}\Big)\dif x & + \int_{\Omega}F^{\infty}\Big(\frac{\dif\mathbb{T}\uu}{\dif|\mathbb{T}^{s}\uu|}\Big)\dif|\mathbb{T}^{s}\uu| \\ & + \int_{\partial\Omega}F^{\infty}(-\mathrm{tr}_{\partial\Omega}(\uu)\otimes_{\mathbb{T}}\nu_{\partial\Omega})\dif\mathscr{H}^{1}.
	\end{split}
	\end{align}
	Here, $\mathbb{T}\uu = \frac{\dif\mathbb{T}\uu}{\dif\mathscr{L}^{2}}\mathscr{L}^{2}+\frac{\dif\mathbb{T}\uu}{\dif|\mathbb{T}^{s}\uu|}|\mathbb{T}^{s}\uu|$ is the Lebesgue-Radon-Nikod\'{y}m decomposition of $\mathbb{T}\uu\in\mathrm{RM}_{\mathrm{fin}}(\Omega;\RR)$. Moreover, $F^{\infty}(z)\coloneqq \lim_{t\searrow 0}tF(\frac{z}{t})$ is the recession function of the integrand $F$, and is precisely designed to capture the potential singularity of $\mathbb{T}\uu$ with respect to $\mathscr{L}^{2}$. Hence, when constructing approximate solutions, the $F^{\infty}$-terms deal with concentration effects. This also concerns the boundary integral in \eqref{eq:relaxintro}, see Remark \ref{rem:bdryintegral} below. At this stage, we point out that controlling $\mathbb{T}\uu$ is essentially equivalent to controlling\footnote{If $\sg\uu$ is a Radon measure and not necessarily a function, we write $\E\uu$ to highlight this difference.} $\mathrm{E}\uu$, see \eqref{eq:pointwisecompa}. In conclusion, the coercivity of $\mathscr{F}_{0}^{*}$ can only ensured on the space 
    \begin{align}
    \bd(\Omega)\coloneqq \{\uu\in\lebe^{1}(\Omega;\R^{2})\colon\;\mathrm{E}\uu\in\mathrm{RM}_{\mathrm{fin}}(\Omega;\RR)\}
    \end{align}
    of \emph{functions of bounded deformation}. As a consequence of a basic obstruction from harmonic analysis, there are no Korn-type inequalities in the $\lebe^{1}$-framework as considered here. This principle is known as Ornstein's \emph{Non-Inequality}  \cite{Ornstein,Conti}, and implies that $\bv(\Omega;\R^{2})$ (the functions of bounded variation, see \cite{AFP,EvansGariepy}) is a proper subspace of $\bd(\Omega)$. In particular, for a generic $\bd$-map, its full distributional gradient does not need to belong to $\mathrm{RM}_{\mathrm{fin}}(\Omega;\R^{2\times 2})$. Therefore, $\sobo^{1,p}(\Omega;\R^{2})\subsetneq \bv(\Omega;\R^{2})\subsetneq\bd(\Omega)$ for all $1\leq p<\infty$, and so plasticity models as considered here cannot be embedded into the framework of Sobolev maps.  In particular, this allows to model scenarios where solutions \emph{do not regularize} instantaneously and hardening or cracks may occur indeed. 
    
    From a conceptual perspective, the diffusion described by \eqref{eq:hibler1a} is similar to variational solutions of the $1$-Laplacian equation as studied in, e.g., \cite{LichnewskyTemam,HZ}. As a key point, we do not employ nonlinear semigroup methods, but rather an energy-driven evolution that keeps track of the plastic behaviour in the relaxed energies, see Definition \ref{def:varsol1}. 

\subsection{Strategy and selected obstructions} 
We conclude the present section by discussing the underlying overall strategy of proof and obstructions. In the context considered here, it is natural to impose homogeneous Dirichlet boundary conditions along $\partial\Omega$. This is also the generic situation in Hibler's original model \cite{Hibler} and reflects the fact that sea ice rests at the boundary. In establishing the existence of energy-driven solutions, we employ a three-layer approximation scheme. This corresponds to regularizing the initial values, a viscosity stabilization and a stress regularization. Finally, we pass to the limit in these approximations, where the limit in the stress regularization is the singular limit. 

By the $\bd$-set-up, the relaxed energies \eqref{eq:relaxintro} incorporate the boundary values as penalization terms. In particular, it is essential to admit competitors with non-zero boundary values to obtain a proper energy-driven evolution. Now, the weak approximate solutions to the stabilized or regularized problems automatically come with zero boundary values in the trace sense. The passage from zero to general traces as required for the energy-driven formulation then requires a careful approximation of the boundary penalization terms by bulk integrals. Since, in the present situation, the full gradients of the horizontal sea ice velocities are not available, this obstruction cannot be approached as in the full gradient case; see Section \ref{sec:relaxthebulk} for the resolution of this matter. This, in turn, comes with a variety of other complications and is fully solved in view of the underlying evolution equation in Section \ref{sec:main}. 

\section{Function spaces and auxiliary results}\label{sec:prelims}
\subsection{General notation}
 Unless stated otherwise, $\Omega\subset\R^{n}$ is an open and bounded set with Lipschitz boundary $\partial\Omega$; typically, $n=2$. Once the boundary is locally flattened and written as the finite union of rotated and translated graphs of Lipschitz functions, the maximum of the underlying Bi-Lipschitz constants shall be referred to as the \emph{Lipschitz character} of $\partial\Omega$.

 For $T>0$, we denote the space-time cylinder by $\Omega_{T}:=\Omega\times(0,T)$. We then denote, for a finite dimensional inner product space $X$, by $\mathrm{RM}_{\mathrm{fin}}(\Omega;X)$ the finite, $X$-valued Radon measures on $\Omega$. Here, finiteness is understood in the sense that $|\mu|(\Omega)<\infty$ with the total variation measure $|\mu|$. If $X=\R$, we write $\mathrm{RM}_{\mathrm{fin}}(\Omega):=\mathrm{RM}_{\mathrm{fin}}(\Omega;\R)$ for simplicity. The $n$-dimensional Lebesgue and the $(n-1)$-dimensional Hausdorff measures are denoted by $\mathscr{L}^{n}$ or $\mathscr{H}^{n-1}$, respectively. Lastly, for a Borel subset $B\subset\Omega$, the restriction $\mu\mres B$ of $\mu\in\mathrm{RM}_{\mathrm{fin}}(\Omega;X)$ is given by $(\mu\mres B)(A):=\mu(A\cap B)$. 
 
 For a Banach space $(X,\|\cdot\|_{X})$, we say that $\vv\colon [0,T]\to X$ is \emph{strongly measurable} if there exists a sequence $(\vv_{j})$ of $X$-valued simple functions such that $\|\vv_{j}(t,\cdot)-\vv(t,\cdot)\|_{X}\to 0$ as $j\to\infty$ for $\mathscr{L}^{1}$-a.e. $0<t<T$. If $X$ is separable, we call $\vv\colon [0,T]\to X'$ \emph{weakly*-measurable} if, for any $z\in X$, the map $[0,T]\ni t \mapsto \langle \vv(t,\cdot),z\rangle$ is measurable; the space $\lebeweak(0,T;X')$ is defined as the collection of all weakly*-measurable maps $\vv\colon [0,T]\to X'$ such that 
\begin{align}\label{eq:weak*-bochner}
\int_{0}^{T}\|\vv(t,\cdot)\|_{X'}\dif t<\infty. 
\end{align}
We refer the reader to the Appendix, Section \ref{sec:appendix}, for more detail. For two vectors $a,b\in\R^{n}$, we define their symmetric tensor product via
\begin{align}\label{eq:symtensor}
a\odot b :=\frac{1}{2}(a\otimes b + b\otimes a)
\end{align}
with the usual tensor product $a\otimes b:=ab^{\top}$. For $z\in\R^{n\times n}$, we write $|z|$ for the usual Hilbert-Schmidt norm. Finally, $c>0$ denotes constants which may change from one line to the other; we shall only specify them if their exact value is required.

\subsection{Functions of bounded deformation}\label{sec:BD}
Let $\Omega\subset\R^{n}$ be open. For a vector field $\uu\in\lebe_{\locc}^{1}(\Omega;\R^{n})$, we define its \emph{total deformation} by
\begin{align}\label{eq:totdef}
|\E\uu|(\Omega):=\sup\left\{ \int_{\Omega}\uu\cdot\mathrm{div}(\bm{\varphi})\dif x \colon\;\bm{\varphi}\in\hold_{c}^{\infty}(\Omega;\R_{\mathrm{sym}}^{n\times n}),\;|\bm{\varphi}|\leq 1\right\}.
\end{align}
Here, $\mathrm{div}(\bm{\varphi})$ denotes the row-wise divergence. We say that $u$ is \emph{of bounded deformation} provided $|\mathrm{E}\uu|(\Omega)<\infty$, and define the \emph{space of functions of bounded deformation} via
\begin{align}\label{eq:BDdef}
\bd(\Omega):=\{\uu\in\lebe^{1}(\Omega;\R^{n})\colon\;|\E\uu|(\Omega)<\infty\}.
\end{align}
The symmetric distributional gradient of a vector field $\uu\in\bd(\Omega)$ thus is a finite, $\R_{\mathrm{sym}}^{n\times n}$-valued Radon measure. The relation of $\bd$ and $\bv$ is clarified in the next remark:
\begin{remark}\label{rem:OrnsteinBV}
Ornstein's Non-Inequality \cite{Ornstein,Conti} asserts that there exists a sequence $(\uu_{j})\subset\hold_{c}^{\infty}(\Omega;\R^{n})$ such that $\sup_{j\in\mathbb{N}}\|\sg\uu_{j}\|_{\lebe^{1}(\Omega)}<\infty$, whereas $\limsup_{j\to\infty}\|\nabla\uu_{j}\|_{\lebe^{1}(\Omega)}=+\infty$. In particular, Korn's inequality fails for $p=1$. In consequence, $\bv(\Omega;\R^{n})\subsetneq\bd(\Omega)$. 
\end{remark}
Now let $\uu\in\bd(\Omega)$. The corresponding Lebesgue-Radon-Nikod\'{y}m decomposition of $\mathrm{E}\uu$ into its absolutely continuous and singular parts with respect to $\mathscr{L}^{n}$ reads as
\begin{align}\label{eq:symgraddecomp}
\begin{split}
\E\uu  = \E^{a}\uu + \E^{s}\uu & = \frac{\dif\E^{a}\uu}{\dif\mathscr{L}^{n}}\mathscr{L}^{n} + \frac{\dif \E^{s}\uu}{\dif |\E^{s}\uu|}|\E^{s}\uu| = \mathscr{E}\uu\,\mathscr{L}^{n}  +  \frac{\dif \E^{s}\uu}{\dif |\E^{s}\uu|}|\E^{s}\uu|
\end{split}
\end{align}
where $\mathscr{E}\uu$ is the symmetric part of the approximate gradient $\nabla\uu$ of $\uu$; see, e.g. \cite{ACD,GmRa2018}. In view of the Hibler deformation $\mathbb{T}\uu$ in $n=2$ dimensions, see \eqref{eq:hiblerwritedown}--\eqref{eq:Tdef}, this gives rise to the decomposition 
\begin{align}\label{eq:hiblerdecomp}
\begin{split}
\mathbb{T}\uu  = \mathbb{T}^{a}\uu + \mathbb{T}^{s}\uu & = \frac{\dif\mathbb{T}^{a}\uu}{\dif\mathscr{L}^{2}}\mathscr{L}^{2} + \frac{\dif \mathbb{T}^{s}\uu}{\dif |\mathbb{T}^{s}\uu|}|\mathbb{T}^{s}\uu| = \mathscr{T}\uu\,\mathscr{L}^{2}  +  \frac{\dif \mathbb{T}^{s}\uu}{\dif |\mathbb{T}^{s}\uu|}|\mathbb{T}^{s}\uu|, 
\end{split}
\end{align}
where $\mathscr{T}\uu = T[\mathscr{E}\uu]$. For future reference, we remark that endowing $\bd(\Omega)$ with the \emph{$\bd$-norm}
\begin{align}\label{eq:BDnormdef}
\|\uu\|_{\bd(\Omega)}:=\|\uu\|_{\lebe^{1}(\Omega)}+|\E\uu|(\Omega),\qquad\uu\in\bd(\Omega),
\end{align}
makes $\bd(\Omega)$ into a non-separable Banach space. The norm topology defined by \eqref{eq:BDnormdef} is too strong for most applications: For instance, smooth approximation fails with respect to \eqref{eq:BDnormdef}.  Hence, it is natural to pass to weaker notions of convergence as follows: Given $\uu,\uu_{1},\uu_{2},...\in\bd(\Omega)$, we say that $(\uu_{j})$ converges to $\uu$ in the 
\begin{itemize}
\item \emph{weak*-sense} to $\uu$ (and write $\uu_{j}\stackrel{*}{\rightharpoonup}\uu$ in $\bd(\Omega)$) if and only if $\uu_{j}\to \uu$ strongly in $\lebe^{1}(\Omega;\R^{n})$ and $\E \uu_{j}\stackrel{*}{\rightharpoonup} \E\uu$ in $\mathrm{RM}_{\mathrm{fin}}(\Omega;\R_{\mathrm{sym}}^{n\times n})\cong\hold_{0}(\Omega;\R_{\mathrm{sym}}^{n\times n})'$ as $j\to\infty$.
\item \emph{strict sense} or \emph{strictly} to $\uu$ (and write $\uu_{j}\stackrel{s}{\rightharpoonup}\uu$) if and only if $d_{s}(\uu_{j},\uu)\to 0$ as $j\to\infty$ with the \emph{strict metric} $d_{s}$ defined by
\begin{align*}
d_{s}(\uu,\vv):=\|\uu-\vv\|_{\lebe^{1}(\Omega)}+||\mathrm{E}\uu|(\Omega)-|\mathrm{E}\vv|(\Omega)|,\qquad \uu,\vv\in\bd(\Omega).
\end{align*}
\end{itemize}
We then have that norm convergence implies strict convergence and strict convergence implies weak*-convergence, but in general none of the reverse implications hold true. The space $\mathrm{LD}(\Omega)$ is the Sobolev analogue of $\bd(\Omega)$, meaning that 
\begin{align*}
\mathrm{LD}(\Omega)\coloneqq \{\uu\in\bd(\Omega)\colon\;\mathrm{E}\uu\ll\mathscr{L}^{n}\}, 
\end{align*}
and we define a variant with zero boundary values via 
\begin{align*}
\mathrm{LD}_{0}(\Omega) \coloneqq \overline{\hold_{c}^{\infty}(\Omega;\R^{n})}^{\|\cdot\|_{\bd(\Omega)}}. 
\end{align*}
The following lemma collects some useful properties of weak*- and strict convergences which will crucially enter the proofs of our main results. 
\begin{lemma}[Weak*-compactness and smooth approximation] \label{lem:auxBD}
Let $\Omega\subset\R^{n}$ be open and bounded with Lipschitz boundary $\partial\Omega$. Then the following hold:
\begin{enumerate}
\item\label{item:aux0} There exists a separable Banach space $(X,\|\cdot\|_{X})$ such that $X'\cong\bd(\Omega)$.
\item\label{item:aux0A} If $\uu,\uu_{1},\uu_{2},...\in\bd(\Omega)$ are such that $\uu_{j}\to\uu$ in $\lebe_{\locc}^{1}(\Omega;\R^{2})$, then 
\begin{align*}
|\mathbb{T}\uu|(\Omega)\leq \liminf_{j\to\infty}|\mathbb{T}\uu_{j}|(\Omega). 
\end{align*}
\item\label{item:aux1} \emph{Weak*-compactness:} If $(\uu_{j})\subset\bd(\Omega)$ satisfies $\sup_{j\in\mathbb{N}}\|\uu_{j}\|_{\bd(\Omega)}<\infty$, then there exists $\uu\in\bd(\Omega)$ and a subsequence $(\uu_{j(k)})\subset (\uu_{j})$ such that we have $\uu_{j(k)}\stackrel{*}{\rightharpoonup}\uu$ in $\bd(\Omega)$ as $k\to\infty$.
\item\label{item:aux2} \emph{Smooth approximation:} For any $\uu\in\bd(\Omega)$, there exists $(\uu_{j})\subset \bd(\Omega)\cap\hold^{\infty}(\Omega;\R^{n})$ such that $d_{s}(\uu,\uu_{j})\to 0$ as $j\to\infty$.
\end{enumerate}
\end{lemma}
Property \ref{item:aux0} shall be established in the Appendix, Section \ref{sec:appendix}; properties \ref{item:aux0A}--\ref{item:aux2} are collected from \cite[\S 2]{BDG}.

In order to formulate a suitable version of Poincar\'{e}'s inequality, we first recall that on open and connected subsets $\Omega\subset\R^{n}$, the nullspace of $\mathrm{E}$ is given by the \emph{rigid deformations} 
\begin{align*}
\mathscr{R}:=\{\pi\colon x \mapsto Ax+b\colon\;A\in\R_{\mathrm{skew}}^{n\times n},\;b\in\R^{n}\}. 
\end{align*}
\begin{lemma}[Poincar\'{e}-Sobolev]\label{lem:poincaresobolev} Let $\Omega\subset\R^{n}$ be open and bounded with Lipschitz boundary $\partial\Omega$. Then the following hold: 
\begin{enumerate}
\item\label{item:PoincareSobolev1} $\bd(\Omega)\hookrightarrow \lebe^{\frac{n}{n-1}}(\Omega;\R^{n})$ and $\bd(\Omega)\hookrightarrow\hookrightarrow\lebe^{q}(\Omega;\R^{2})$ for all $1\leq q<\frac{n}{n-1}$. Here, '$\hookrightarrow\hookrightarrow$' denotes compactness of the embedding.  
\item Suppose moreover that $\Omega$ is connected. Then there exists a constant $c>0$ independent of $\uu$ and an element $\pi\in\mathscr{R}$ such that 
\begin{align}
\|\uu-\pi\|_{\lebe^{\frac{n}{n-1}}(\Omega)}\leq c\, |\E\uu|(\Omega). 
\end{align}
\end{enumerate}
\end{lemma}
The following trace theorem is due to Temam \& Strang \cite{TemamStrang}, see also \cite{BDG}. 
\begin{lemma}[Trace operator and extensions]\label{lem:traceoperator}
Let $\Omega\subset\R^{n}$ be open and bounded with Lipschitz boundary oriented by the outer unit normal $\nu_{\partial\Omega}\colon\partial\Omega\to\mathbb{S}^{1}$. Then the following hold: 
\begin{enumerate}
    \item\label{item:tracex1} There exists a surjective,  linear \emph{trace operator} $\mathrm{tr}_{\partial\Omega}\colon\bd(\Omega)\to\lebe^{1}(\partial\Omega;\R^{n})$ which is bounded with respect to the $\bd$-norm and continuous with respect to strict convergence on $\bd(\Omega)$. In particular, we have $\mathrm{tr}_{\partial\Omega}(\uu)=\uu|_{\partial\Omega}$ $\mathscr{H}^{n-1}$-a.e. on $\partial\Omega$ provided that $\uu\in\bd(\Omega)\cap\hold(\overline{\Omega};\R^{n})$. 
    \item\label{item:tracex2} Let $\widetilde{\Omega}\subset\R^{n}$ be open and bounded with Lipschitz boundary such that  $\Omega\Subset\widetilde{\Omega}$. If $\uu\in\bd(\Omega)$ and $\vv\in\bd(\widetilde{\Omega}\setminus\overline{\Omega})$, then the \emph{glued} function $\mathbf{w}\coloneqq \mathbbm{1}_{\Omega}\uu + \mathbbm{1}_{\widetilde{\Omega}\setminus\overline{\Omega}}\mathbf{v}$ belongs to $\bd(\widetilde{\Omega})$ together with 
    \begin{align*}
    \E\mathbf{w} = \E\uu\mres\Omega + \E\mathbf{v}\mres(\widetilde{\Omega}\setminus\overline{\Omega}) + (\mathrm{tr}_{\partial\Omega}(\vv-\uu)\odot\nu_{\partial\Omega})\mathscr{H}^{n-1}\mres\partial\Omega.
    \end{align*}
\end{enumerate}
\end{lemma}
The trace operator from Lemma \ref{lem:traceoperator} is \emph{not} continuous with respect to weak*-convergence on $\bd(\Omega)$. Indeed, the $\hold_{c}^{\infty}(\Omega;\R^{n})$-maps are dense in $\bd(\Omega)$ with respect to weak*-convergence. For future reference, we moreover single out the following remark.
\begin{remark}[On the right-inverse of the trace operator]\label{rem:peetre}
In the situation of Lemma \ref{lem:traceoperator}, there exists a bounded operator $\mathrm{ext}\colon\lebe^{1}(\partial\Omega;\R^{n})\to\sobo^{1,1}(\Omega;\R^{n})$ such that $\mathrm{tr}_{\partial\Omega}(\mathrm{ext}(\mathbf{v}))=\mathbf{v}$ for all $\mathbf{v}\in\lebe^{1}(\partial\Omega;\R^{n})$. In particular, $\|\mathrm{ext}(\mathbf{v})\|_{\sobo^{1,1}(\Omega)}\leq c\|\mathbf{v}\|_{\lebe^{1}(\partial\Omega)}$ holds for all $\mathbf{v}\in\lebe^{1}(\partial\Omega;\R^{n})$ with a constant $c>0$ independent of $\vv$. However, in general, it is not possible to choose $\mathrm{ext}$ to be bounded and linear; see Peetre \cite{Peetre}. 
\end{remark}
The bulk approximation of boundary terms in Section \ref{sec:relaxthebulk} moreover requires a Poincar\'{e}-type inequality as established  recently by the second named author, S\"{u}li and Tscherpel \cite{GmeSulTsc}: 
\begin{lemma}[of Poincar\'{e}-type]\label{lem:GSTLemma}
Let $U\subset\R^{n}$ be open and bounded with Lipschitz boundary $\partial U$. Moreover, let $\Gamma\subset\partial U$ be $\mathscr{H}^{n-1}$-measurable such that $\mathscr{H}^{n-1}(\partial U_{k}\cap\Gamma)>0$ holds for every connected component $U_{k}$ of $U$. Then there exists a constant $c>0$ such that 
\begin{align}\label{eq:tabeabound}
\|\uu\|_{\lebe^{1}(U)}\leq c\Big(\|\mathrm{tr}_{\partial U}(\uu)\|_{\lebe^{1}(\Gamma)}+|\mathrm{E}\uu|(U)\Big)\qquad\text{for all}\;\uu\in\bd(U). 
\end{align}
In particular, the constant $c>0$ is monotone in the ratio $\mathscr{H}^{n-1}(\partial U)/\mathscr{H}^{n-1}(\Gamma)$ and is monotone in the Lipschitz character of $\partial U$.
\end{lemma}
Note that \cite[Prop. 3.9]{GmeSulTsc} establishes inequality \eqref{eq:tabeabound} even with the stronger $\lebe^{\frac{n}{n-1}}$-norm of $\uu$ on its left-hand side. For us, however, \eqref{eq:tabeabound} is sufficient. Moreover, we remark that it is important indeed to require $\mathscr{H}^{n-1}(\partial U_{k}\cap\Gamma)>0$ for every connected component $U_{k}$ of $U$: If, e.g., $U_{1}$ is a connected component with $\mathscr{H}^{n-1}(\partial U_{1}\cap\Gamma)=0$, then the inequality \eqref{eq:tabeabound} is violated for any element of the form $\uu=\mathbbm{1}_{U_{1}}\pi$, where $\pi\in\mathscr{R}\setminus\{0\}$ is arbitrary. \\

In the following, it is convenient to adopt a slightly broader viewpoint \`{a} la \cite{BDG} and to consider integral formulas that apply to the Hibler operator $\mathbb{T}$ as a special case. To this end, let $V,W$ be two finite dimensional inner product spaces and let $\mathbb{A}$ be an $W$-valued differential operator of the form 
\begin{align}\label{eq:diffopform}
\mathbb{A}=\sum_{k=1}^{n}\mathbb{A}_{k}\partial_{k}\qquad\text{acting on}\;u\colon\R^{n}\to V, 
\end{align}
where $\mathbb{A}_{k}\colon V\to W$ is linear for each $k\in\{1,...,n\}$; typically, $V=\R^{n}$ and $W=\R_{\mathrm{sym}}^{n\times n}$. We define its formal $\lebe^{2}$-adjoint by 
\begin{align}\label{eq:formaladjoint}
\A^{*} \coloneqq \sum_{k=1}^{n}\A_{k}^{\top}\partial_{k}\qquad\text{acting on $u\colon W\cong W'\to V'\cong V$.}
\end{align}
For $\mathbf{a}\in V$ and $\mathbf{b}=(b_{1},...,b_{n})\in\R^{n}$, its Fourier symbol then gives rise to the generalized tensor product 
\begin{align}\label{eq:pairingA}
\mathbf{a}\otimes_{\mathbb{A}}\mathbf{b} \coloneqq \mathbb{A}[\mathbf{b}]\mathbf{a} := \sum_{k=1}^{n}b_{k}\mathbb{A}_{k}\mathbf{a}.
\end{align}
In the special case where $\mathbb{A}$ is the gradient or the symmetric gradient, we have $\otimes_{\nabla}=\otimes$ and $\otimes_{\varepsilon}=\odot$, see \eqref{eq:symtensor}. For the Hibler operator $\mathbb{A}=\mathbb{T}$, where $V=\R^{2}$ and $W=\RR$, the associated generalized tensor product is consequently denoted by $\otimes_{\mathbb{T}}$. Following \cite{BDG}, we put 
\begin{align*}
\bv^{\mathbb{A}}(\Omega) \coloneqq \{\uu\in\lebe^{1}(\Omega;V)\colon\;\mathbb{A}\uu\in\mathrm{RM}_{\mathrm{fin}}(\Omega;W)\}. 
\end{align*}
Based on \eqref{eq:pairingA}, we have the following weak product rule for $\uu\in\bv^{\mathbb{A}}(\Omega)$ and $\eta\in\sobo^{1,\infty}(\Omega)$: 
\begin{align}\label{eq:prodruleweak}
\mathbb{A}(\eta\uu)=\eta\,\mathbb{A}\uu + \uu\otimes_{\mathbb{A}}\nabla\eta\mathscr{L}^{n}\qquad\text{as vectorial measures on $\Omega$}. 
\end{align}
Now assume that $\mathbb{A}$ is $\mathbb{C}$-elliptic, meaning that the complexified symbol $\mathbb{A}[\xi]\colon V+\mathrm{i}V\to W + \mathrm{i}W$ is injective for all $\xi\in\mathbb{C}^{n}\setminus\{0\}$. Moreover, we suppose that $\Omega\subset\R^{n}$ is open and bounded with Lipschitz boundary oriented by the outer unit normal $\nu_{\partial\Omega}\colon\partial\Omega\to\mathbb{S}^{n-1}$. By the generalization \cite[Theorem 1.1]{BDG} of Lemma \ref{lem:traceoperator}, there exists a surjective trace operator $\mathrm{tr}_{\partial\Omega}\colon\bv^{\mathbb{A}}(\Omega)\to\lebe^{1}(\partial\Omega;V)$ with the accordingly modified continuity properties. Based on the pairing \eqref{eq:pairingA}, one then has the \emph{Gauss-Green formula}
\begin{align}\label{eq:gaussgreen}
\int_{\partial\Omega}\varphi (\mathrm{tr}_{\partial\Omega}(\uu)\otimes_{\mathbb{A}}\nu_{\partial\Omega})\dif\mathscr{H}^{n-1} = \int_{\Omega}\varphi\dif\mathbb{A}\uu + \int_{\Omega}\uu\otimes_{\mathbb{A}}\nabla\varphi\dif x
\end{align}
whenever $\varphi\in\hold^{0,1}(\Omega)$ and $\uu\in\bv^{\mathbb{A}}(\Omega)$, see \cite[\S 4]{BDG}. The equation \eqref{eq:gaussgreen} is understood in $W$. For an equation in $\R$ and hereafter $\bm{\varphi}\in\hold^{0,1}(\Omega;W)$, the scalar variant is given by 
\begin{align}\label{eq:IBPscalar}
\int_{\partial\Omega}\bm{\varphi}\cdot(\mathrm{tr}_{\partial\Omega}(\uu)\otimes_{\mathbb{A}}\nu_{\partial\Omega})\dif\mathscr{H}^{n-1} = \int_{\Omega}\bm{\varphi}\cdot \dif\mathbb{A}\uu + \int_{\Omega}\uu\cdot\mathbb{A}^{*}\bm{\varphi}\dif x.
\end{align}
\begin{remark}\label{rem:GGapplication}
For the main purposes of this paper, one could directly work with the $\mathbb{C}$-elliptic Hibler operator $\mathbb{A}=\mathbb{T}$. However, due to \eqref{eq:pointwisecompa}, the associated $\bv^{\mathbb{A}}$-space is just $\bd(\Omega)$; in this sense, \eqref{eq:gaussgreen} primarily serves to yield the requisite Gauss-Green formula which shall enter Section \ref{sec:weakvar} below. In particular, \eqref{eq:gaussgreen} holds with $\mathbb{A}=\mathbb{T}$ and $\uu\in\bd(\Omega)$. We remark that, due to \cite{BDG,GmRa}, all crucial properties available for $\bd$ carry over to $\bv^{\mathbb{A}}$ with the natural modifications.  
\end{remark}

\subsection{Integrands of linear growth}
We now collect some background definitions and facts on convex functionals of linear growth. To this end, let $F\colon\rsym^{2\times 2}\to\R$ be a convex integrand which satisfies 
\begin{align}\label{eq:lingrowthprelims}
c_{1}|z|-c_{2}\leq F(z)\leq c_{3}(1+|z|)\qquad\text{for all}\;z\in\rsym^{2\times 2}, 
\end{align}
for certain constants $c_{1},c_{2},c_{3}>0$. In the context considered here, $F$ is not required to be differentiable. For $z_{0}\in\mathbb{R}_{\mathrm{sym}}^{2\times 2}$, we call $\bm{\xi}\in\mathbb{R}_{\mathrm{sym}}^{2\times 2}$ a \emph{subgradient} at $z_{0}$ provided that 
\begin{align}\label{eq:subgradientdef}
F(z)\geq F(z_{0})+\bm{\xi}\cdot(z-z_{0})\qquad\text{for all}\;z\in\R_{\mathrm{sym}}^{2\times 2}. 
\end{align}
The \emph{subdifferential} $\partial F(z_{0})$ is the collection of all subgradients at $z_{0}$. If $F$ is differentiable at $z_{0}$, we have that $\partial F(z_{0})=\{F'(z_{0})\}$.

Subject to the convexity and linear growth assumption \eqref{eq:lingrowthprelims}, the associated recession function 
\begin{align}\label{eq:recessionfunction}
F^{\infty}(z):=\lim_{t\searrow 0}tF\Big(\frac{z}{t}\Big)
\end{align}
exists for every $z\in\rsym^{2\times 2}$. The function $F^{\infty}$ is designed to capture the behaviour of the integrand for large values of the argument. In particular, in the relaxation of convex linear growth functionals to $\bd(\Omega)$, it serves to describe the singular part $\mathrm{E}^{s}\uu$. For future reference, we record the following properties of $F$ and $F^{\infty}$:
\begin{lemma}\label{lem:subdif}
Let $F\colon\rsym^{2\times 2}\to\R$ be convex and satisfy \eqref{eq:lingrowthprelims}. Then the following hold: 
\begin{enumerate}
\item\label{item:RobertDenk} $F$ and $F^{\infty}$ are  Lipschitz on $\rsym^{2\times 2}$.
\item\label{item:RobertDenk1} The recession dominates the subgradients in the sense that 
\begin{align*}
\bm{\xi}\cdot \bm{\eta} \leq F^{\infty}(\bm{\eta})\qquad\text{for all}\;z,\bm{\eta}\in\rsym^{2\times 2}\;\text{and all}\;\bm{\xi}\in\partial F(z).
\end{align*}
\end{enumerate}
\end{lemma}
\begin{proof} Assertion \ref{item:RobertDenk} directly follows from the linear growth and convexity of $F$. For \ref{item:RobertDenk1}, we conclude by \eqref{eq:lingrowthprelims} that $t\bm{\xi}\cdot\bm{\eta}\leq F(z+t\bm{\eta})-F(z)$ for $t>0$. We split
\begin{align*}
\frac{F(z+t\bm{\eta})-F(z)}{t} = \frac{F(z+t\bm{\eta})-F(t\bm{\eta})}{t} + \frac{F(t\bm{\eta})}{t} - \frac{F(z)}{t} \eqqcolon \mathrm{I}(t) + \mathrm{II}(t) + \mathrm{III}(t). 
\end{align*}
By convexity and the linear growth assumption on $F$, $F$ is globally Lipschitz. Hence,  $\mathrm{I}(t)\to 0$ as $t\to\infty$. Since $\mathrm{III}(t)\to 0$ as $t\to \infty$, the claim follows by the very definition of $F^{\infty}$. 
\end{proof}
We conclude this subsection by gathering some elementary facts on convex functionals of measures; see \cite{GoffmanSerrin,Reshetnyak} for more detail. To this end, it is convenient to adopt a slightly more general viewpoint and thereby let $\Omega\subset\R^{n}$ be open and bounded. For a finite dimensional inner product space $W$ and a convex  integrand $F\colon W\to\R$ satisfying  \eqref{eq:lingrowthprelims} with the natural modifications, we define the \emph{linear perspective integrand}
\begin{align}\label{eq:linearperspective}
F^{\#}(t,z)\coloneqq \begin{cases}
tF(\frac{z}{t})&\;\text{if}\;t>0,\;z\in W,\\ 
F^{\infty}(z) &\;\text{if}\;t=0,\;z\in W. 
\end{cases}
\end{align}
Then $F^{\#}\colon [0,\infty)\times W\to\R$ is continuous, of linear growth and positively $1$-homogeneous. For $\bm{\mu}\in\mathrm{RM}_{\mathrm{fin}}(\Omega;W)$ with a Lebesgue-Radon-Nikod\'{y}m decomposition $\bm{\mu}=\bm{\mu}^{a}+\bm{\mu}^{s}$ with respect to $\mathscr{L}^{n}$, we define a new measure $F(\bm{\mu})$ on $\mathscr{B}(\Omega)$ by 
\begin{align*}
F(\bm{\mu})(A)  \coloneqq \int_{A}\dif F(\bm{\mu}) & \coloneqq \int_{A}F^{\#}\Big(\frac{\dif\,(\mathscr{L}^{n},\bm{\mu})}{\dif|(\mathscr{L}^{n},\bm{\mu})|}\Big)\dif\,|(\mathscr{L}^{n},\bm{\mu})| \\ & = \int_{A}F\Big(\frac{\dif\bm{\mu}}{\dif\mathscr{L}^{n}}\Big)\dif\mathscr{L}^{n} + \int_{A}F^{\infty}\Big(\frac{\dif\bm{\mu}}{\dif|\bm{\mu}^{s}|}\Big)\dif|\bm{\mu}^{s}|
\end{align*}
for $A\in\mathscr{B}(\Omega)$. The central lower semicontinuity result for functionals of measures now is
\begin{lemma}[Reshetnyak's  lower semicontinuity theorem, {\cite{Reshetnyak}}]\label{lem:resh}
Let $\Omega\subset\R^{n}$ be open and bounded, and let $F\colon W\to \R$ be a convex integrand with \eqref{eq:lingrowthprelims}. If $\bm{\mu},\bm{\mu}_{1},\bm{\mu}_{2},...\in\mathrm{RM}_{\mathrm{fin}}(\Omega;W)$ are such that $\bm{\mu}_{j}\stackrel{*}{\rightharpoonup}\bm{\mu}$ in the weak*-sense in $\mathrm{RM}_{\mathrm{fin}}(\Omega;W)\cong\hold_{0}(\Omega;W)'$, then 
\begin{align}\label{eq:LSCResh}
F(\bm{\mu})(\Omega) \leq \liminf_{j\to\infty} F(\bm{\mu}_{j})(\Omega). 
\end{align}
\end{lemma}
The following inequality of Jensen type might be clear to the experts; to keep the paper self-contained, a proof is offered in the Appendix, Section \ref{sec:jensen}. 
\begin{lemma}[of Jensen-type]\label{lem:jensen}
Let $\bm{\mu}\in\mathrm{RM}_{\mathrm{fin}}(\Omega;W)$ and denote, for $\varepsilon>0$, by $\rho_{\varepsilon}$ the $\varepsilon$-rescaled variant of a standard mollifier $\rho\in\hold_{c}^{\infty}(\ball_{1}(0);\R_{\geq 0})$. Then, for any convex function $F\colon W\to\R$ with \eqref{eq:lingrowthprelims} and any open set $U\subset\Omega$ with $\mathrm{dist}(U,\partial\Omega)>\varepsilon$, we have \emph{Jensen's inequality}
\begin{align}\label{eq:Jensen}
\int_{U}F(\rho_{\varepsilon}*\bm{\mu})\dif x \leq \int_{U_{\varepsilon}} F(\bm{\mu}), 
\end{align}
where $U_{\varepsilon}\coloneqq\{x\in\Omega\colon\;\dista(x,U)<\varepsilon\}$. 
\end{lemma}

\subsection{The coarea formula}
We now record a version of Federer's coarea formula: 
\begin{lemma}[Coarea formula, {\cite[Chpt. 3.4.3, Thm. 2]{EvansGariepy}}]\label{lem:coarea}
Let $f\colon\R^{n}\to\R^{m}$ be Lipschitz, where $n\geq m$. Then, for each $\mathscr{L}^{n}$-integrable function $g\colon\R^{n}\to\R^{m}$, $g|_{f^{-1}(\{y\})}$ is $\mathscr{H}^{n-m}$-integrable for $\mathscr{L}^{m}$-a.e. $y$, and 
\begin{align*}
\int_{\R^{n}}g(x)\mathbf{J}_{f}(x)\dif x = \int_{\R^{m}}\int_{f^{-1}(\{y\})}g\dif\mathscr{H}^{n-m}\dif y. 
\end{align*}
Here, $\mathbf{J}_{f}$ is the Jacobian determinant of $f$. In particular, 
if $m=1$, then $\mathbf{J}_{f}=|\nabla f|$. 
\end{lemma}

\subsection{Monotone operators and compactness}
In this ultimate subsection, we collect background facts that will turn out useful when studying approximative problems.

Let $H$ be a real Hilbert space, and suppose that $\Phi\colon H\to(-\infty,\infty]$ satisfies the following:
\begin{enumerate}[label={(C\arabic*)}]
\item\label{item:hikaru1} $\Phi\colon H\to (-\infty,\infty]$ is convex and proper, meaning that  $\Phi\not\equiv\infty$, 
\item\label{item:hikaru2} $E_{\lambda}\coloneqq \{x\in \mathrm{dom}(\Phi)\colon\;\Phi(x)\leq \lambda\}$ is compact in $H$ for any $\lambda\in\R$, where the effective domain of $\Phi$ is given by 
\begin{align*}
\mathrm{dom}(\Phi)\coloneqq \{x\in H\colon\;\Phi(x)\neq \infty\}. 
\end{align*}
\item\label{item:hikaru3} $g\colon \lebe^{2}(0,T;H)\to \lebe^{2}(0,T;H)$ is continuous with respect to $\|\cdot\|_{\lebe^{2}(0,T;H)}$ and of sublinear growth, that is, there exist $L\geq 0$ and $b\in\lebe^{2}(0,T;\R_{\geq 0})$ such that 
\begin{align}\label{eq:gbounds}
\|g(\mathbf{v}(t))\|_{H} \leq L\|\mathbf{v}(t)\|_{H} + b(t)\qquad\text{for all}\;\mathbf{v}\in \lebe_{\locc}^{1}(0,T;H)\;\text{and $\mathscr{L}^{1}$-a.e.}\;0<t<T. 
\end{align}
\end{enumerate}
\begin{proposition}[Arendt \& Hauer, {\cite[Thm. 1.2, Rem. 1.3]{ArendtHauer}}]\label{prop:arendt}
Let $H$ be a real Hilbert space, and suppose that $\Phi\colon H\to (-\infty,\infty]$ satisfies  \emph{\ref{item:hikaru1}--\ref{item:hikaru3}}. Moreover, let $\uu_{0}\in\mathrm{dom}(\Phi)$. Then, for every $0<T<\infty$, there exists a solution $\uu\in\sobo^{1,2}(0,T;H)$ of the evolution equation 
\begin{align}\label{eq:auxsystem1}
\begin{cases}
\partial_{t}\uu + \mathcal{A}\uu + \mathbf{f} = g(\uu)&\quad\text{$\mathscr{L}^{1}$-a.e. in $(0,T)$}, \\ 
\uu(0,\cdot)= \uu_{0}&\quad\text{for}\;t=0, 
\end{cases}
\end{align}
where $\mathcal{A}\coloneqq\partial\Phi$ is the subdifferential of $\Phi$ defined in analogy with \eqref{eq:subgradientdef}. Moreover, we have $\Phi(\uu)\in\lebe^{1}((0,T))$. 
\end{proposition}
\begin{remark}
In the case where $g\equiv 0$, the previous proposition can be inferred from the slightly less general results to be found, e.g., in Showalter \cite[Chapter 3]{Showalter}. Let us moreover note that the aforementioned by Arendt and Hauer also works in the non-smooth context. Even though our functionals are non-smooth in the context considered here, this is only of limited use; see Remark \ref{rem:AHlimitations} for more detail. 
\end{remark}
\section{Relaxations and the bulk approximation of boundary terms}\label{sec:relaxthebulk}
In this section, we collect lower semicontinuity results for the relaxed energies and establish a bulk approximation of the boundary penalization terms in the relaxed energy. 

\subsection{The relaxed Hibler energy} 
We begin with recording various background facts on relaxations. The subsequent results might be clear to the experts, but since we require them in the sequel, we pause to give the underlying framework in detail. To this end, we recall from \eqref{eq:lingrowthmodelling} that we consider energy integrands $F\colon\RR\to\R$ which are of \emph{linear growth} in the sense that there exist constants $c_{1},c_{2},c_{3}>0$ such that 
\begin{align}\label{eq:lgresh}
c_{1}|z|-c_{2}\leq F(z)\leq c_{3}(1+|z|)\qquad\text{for all}\;z\in\rsym^{2\times 2}. 
\end{align}
In the case of constant pressure, the standard Hibler energy fits into this setting by putting $F(z)\coloneqq\frac{P}{2}|z|$, see \eqref{eq:hiblerenergy1}.

By the linear growth hypothesis on $F$ (see, e.g., \eqref{eq:lgresh} below) and Lemma \ref{lem:auxBD}\ref{item:aux1}, we need to extend the Hibler energy to $\bd(\Omega)$. In order to accomplish limit passages in the nonlinear energies, it is thus natural to extend $\mathscr{F}[-;\Omega]$ from $\ld(\Omega)$ to $\bd(\Omega)$ by lower semicontinuity.  
\begin{lemma}[Lower semicontinuity]\label{lem:LSC}
Let $F\in\hold(\rsym^{2\times 2})$ be convex and of linear growth in the sense of \eqref{eq:lgresh}. Moreover, define the recession function $F^{\infty}\in\hold(\RR)$ as in \eqref{eq:recessionfunction}. If $\Omega\subset\R^{2}$ is open and bounded and $\uu\in\bd(\Omega)$, let 
\begin{align*}
\mathbb{T}\uu = \mathbb{T}^{a}\uu + \mathbb{T}^{s}\uu \eqqcolon \mathscr{T}\uu\, \mathscr{L}^{2} + \frac{\dif\mathbb{T}^{s}\uu}{\dif|\mathbb{T}^{s}\uu|}|\mathbb{T}^{s}\uu| 
\end{align*}
be the Radon-Nikod\'{y}m decomposition of $\mathbb{T}\uu\in\mathrm{RM}_{\mathrm{fin}}(\Omega;\RR)$ with respect to $\mathscr{L}^{2}$, see \eqref{eq:hiblerdecomp}, and put
\begin{align*}
{\mathscr{F}}^{*}[\uu;\Omega] \coloneqq \int_{\Omega} F(\mathbb{T}\uu) \coloneqq \int_{\Omega}F(\mathscr{T}\mathbf{u})\dif x + \int_{\Omega}F^{\infty}\Big(\frac{\dif\mathbb{T}^{s}\mathbf{u}}{\dif|\mathbb{T}^{s}\mathbf{u}|}\Big)\dif|\mathbb{T}^{s}\mathbf{u}|. 
\end{align*}
Then $\mathscr{F}^{*}[-;\Omega]$ is \emph{lower semicontinuous with respect to weak*-convergence on} $\bd(\Omega)$ and satisfies $\mathscr{F}[\uu;\Omega]=\mathscr{F}^{*}[\uu;\Omega]$ provided that $\uu\in\ld(\Omega)$. 
\end{lemma}
\begin{proof}
If $\uu,\uu_{1},\uu_{2},...\in\bd(\Omega)$
 are such that $\uu_{j}\stackrel{*}{\rightharpoonup}\uu$ in the weak*-sense on $\bd(\Omega)$, then we particularly have $\E\uu_{j}\stackrel{*}{\rightharpoonup}\E\uu$ and so $\mathbb{T}\uu_{j}\stackrel{*}{\rightharpoonup}\mathbb{T}\uu$ in $\mathrm{RM}_{\mathrm{fin}}(\Omega;\RR)$; note that $\mathbb{T}\uu=T[\E\uu]$ with the linear map $T$ as in \eqref{eq:Tdef}. Hence, the claimed lower semicontinuity directly follows from Reshetnyak's lower semicontinuity theorem, see Lemma \ref{lem:resh}. Finally, if $\uu\in\ld(\Omega)$, then $\mathbb{T}\uu = \mathbb{T}^{a}\uu$, and we immediately obtain $\mathscr{F}[\uu;\Omega]=\mathscr{F}^{*}[\uu;\Omega]$. The proof is complete.  
 \end{proof}
The proof of the following result is probably clear to the experts; for the sake of completeness, we include the detailed argument.
\begin{corollary}\label{cor:LSCplussbdryvalues}
Let $F\in\hold(\R_{\mathrm{sym}}^{2\times 2})$ be convex and of linear growth, see \eqref{eq:lgresh}, and let $\Omega\subset\R^{2}$ be open and bounded with Lipschitz boundary, the latter being oriented by the outer unit normal $\nu_{\partial\Omega}\colon\partial\Omega\to\mathbb{S}^{1}$. Then the following hold: 
\begin{enumerate}
    \item\label{item:relaxdoit1} The \emph{relaxed functional with boundary terms}
\begin{align}\label{eq:relaxedbdry}
\mathscr{F}_{0}^{*}[\mathbf{u};\Omega] \coloneqq \int_{\Omega} F(\mathbb{T}\uu) + \int_{\partial\Omega}F^{\infty}(\mathrm{tr}_{\partial\Omega}(-\mathbf{u})\otimes_{\mathbb{T}}\nu_{\partial\Omega})\dif\mathscr{H}^{1},\qquad\uu\in\bd(\Omega), 
\end{align}
is \emph{lower semicontinuous with respect to weak*-convergence on $\bd(\Omega)$}. 
\item\label{item:relaxdoit2} Let $\rho_{\varepsilon}(x)\coloneqq \varepsilon^{-2}\rho(\frac{x}{\varepsilon})$ be the $\varepsilon$-rescaled variant of a rotationally symmetric standard mollifier $\rho\in\hold_{c}^{\infty}(\R^{2};\R_{\geq 0})$ on $\R^{2}$. If $\uu\in\bd(\Omega)$ has compact essential support in $\Omega$, in formulas $\uu\in\bd_{c}(\Omega)$, then 
\begin{align}\label{eq:relaxedbdrycont}
\lim_{\varepsilon\searrow 0} \mathscr{F}_{0}^{*}[\rho_{\varepsilon}*\uu;\Omega] = \mathscr{F}_{0}^{*}[\uu;\Omega] = \int_{\Omega} F(\mathbb{T}\uu). 
\end{align}
\end{enumerate}
\end{corollary}
\begin{proof}
On \ref{item:relaxdoit1}. 
Let $\widetilde{\Omega}\subset\R^{2}$ be open and bounded such that $\Omega\Subset\widetilde{\Omega}$. For $\mathbf{v}\in\bd(\Omega)$, we let 
\begin{align*}
\widetilde{\mathbf{v}}\coloneqq \begin{cases} \mathbf{v}&\;\text{in}\;\Omega,\\ 
0&\;\text{in}\;\widetilde{\Omega}\setminus\overline{\Omega} 
\end{cases}
\end{align*}
be its trivial extension to $\widetilde{\Omega}$. By Lemma \ref{lem:traceoperator}\ref{item:tracex2}, we have $\widetilde{\vv}\in\bd(\widetilde{\Omega})$ and 
\begin{align}\label{eq:tabeadecomposes}
\mathbb{T}\widetilde{\mathbf{v}}=\mathbb{T}\mathbf{v}\mres\Omega -\mathrm{tr}_{\partial\Omega}(\mathbf{v})\otimes_{\mathbb{T}}\nu_{\partial\Omega}\,\mathscr{H}^{1}\mres\partial\Omega\qquad\text{in}\;\mathscr{D}'(\widetilde{\Omega};\R_{\mathrm{sym}}^{2\times 2}). 
\end{align}
Now let $\mathbf{u},\mathbf{u}_{1},...\in\bd(\Omega)$ such that $\mathbf{u}_{j}\stackrel{*}{\rightharpoonup}\mathbf{u}$ in $\bd(\Omega)$. We claim that $\widetilde{\mathbf{u}}_{j}\stackrel{*}{\rightharpoonup}\widetilde{\mathbf{u}}$ in $\bd(\widetilde{\Omega})$. Indeed, we clearly have $\widetilde{\mathbf{u}}_{j}\to\widetilde{\mathbf{u}}$ strongly in $\lebe^{1}(\widetilde{\Omega};\R^{2})$, and thus conclude for an arbitrary $\bm{\varphi}\in\hold_{c}^{1}(\widetilde{\Omega};\R_{\mathrm{sym}}^{2\times 2})$ that 
\begin{align}\label{eq:tabeaprovesweak*}
\int_{\widetilde{\Omega}}\bm{\varphi}\cdot\dif \mathrm{E}\widetilde{\mathbf{u}}= -\int_{\widetilde{\Omega}}\mathrm{div}(\bm{\varphi})\cdot\mathbf{\widetilde{u}}\dif x = -\lim_{j\to\infty} \int_{\widetilde{\Omega}}\mathrm{div}(\bm{\varphi})\cdot\mathbf{\widetilde{u}}_{j}\dif x = \lim_{j\to\infty}\int_{\widetilde{\Omega}}\bm{\varphi}\cdot\dif\mathrm{E}\mathbf{u}_{j}.
\end{align}
Since $\hold_{c}^{1}(\widetilde{\Omega};\R_{\mathrm{sym}}^{2\times 2})$ is dense in $(\hold_{0}(\widetilde{\Omega};\R_{\mathrm{sym}}^{2\times 2}),\|\cdot\|_{\sup})$, the claimed weak*-convergence follows. Hence, Lemma \ref{lem:LSC} gives us 
\begin{align*}
\int_{\widetilde{\Omega}}F(\mathbb{T}\widetilde{\mathbf{u}}) \leq \liminf_{j\to\infty}\int_{\widetilde{\Omega}} F(\mathbb{T}\widetilde{\mathbf{u}}_{j}). 
\end{align*}
Rewriting both sides by use of \eqref{eq:tabeadecomposes} and cancelling the integrals over $\widetilde{\Omega}\setminus\overline{\Omega}$, this yields the claimed lower semicontinuity for $\mathscr{F}_{0}^{*}[-;\Omega]$.

On \ref{item:relaxdoit2}. For $\uu\in\bd_{c}(\Omega)$, we have $\mathscr{F}_{0}[\uu;\Omega]=\mathscr{F}^{*}[\uu;\Omega]$. Because $\spt(\uu)$ is compactly contained in $\Omega$, there exists $\varepsilon_{0}>0$ such that $\spt(\rho_{\varepsilon}*\uu_{0})\subset\Omega$ for all $0<\varepsilon<\varepsilon_{0}$. We choose a sequence $(\varepsilon_{j})\subset (0,1)$ such that $\varepsilon_{j}\searrow 0$ and $\lim_{j\to\infty}\mathscr{F}_{0}^{*}[\rho_{\varepsilon_{j}}*\uu;\Omega]=\liminf_{\varepsilon\searrow 0}\mathscr{F}_{0}^{*}[\rho_{\varepsilon}*\uu;\Omega]$. An elementary computation confirms that $\uu_{j}\coloneqq \rho_{\varepsilon_{j}}*\uu$ satisfies $\uu_{j}\stackrel{*}{\rightharpoonup}\uu$ in the weak*-sense on $\bd(\Omega)$. By \ref{item:relaxdoit1} and $\mathrm{tr}_{\partial\Omega}(\uu)=\mathrm{tr}_{\partial\Omega}(\uu_{j})=0$ for all $j\in\mathbb{N}$, we infer that 
\begin{align*}
\mathscr{F}_{0}^{*}[\uu;\Omega] \leq \liminf_{j\to\infty}\mathscr{F}_{0}^{*}[\uu_{j};\Omega] & = \liminf_{j\to\infty}\mathscr{F}^{*}[\uu_{j};\Omega] \\ & = \lim_{j\to\infty}\int_{\Omega} F(\rho_{\varepsilon_{j}}*\mathbb{T}\uu)\dif x \stackrel{\eqref{eq:Jensen}}{\leq}  \int_{\Omega} F(\mathbb{T}\uu) = \mathscr{F}_{0}^{*}[\uu;\Omega], 
\end{align*}
where we have used Jensen's inequality \eqref{eq:Jensen} for measures from Lemma \ref{lem:jensen}. In particular, note that we may integrate over the entire $\Omega$ (instead of passing to $\Omega_{\varepsilon}$ as in \eqref{eq:Jensen}) since $\uu$ is compactly supported and $0<\varepsilon<\varepsilon_{0}$. The proof is complete. 
\end{proof}
\begin{remark}\label{rem:bdryintegral}
As in the maybe more familiar $\bv$-case, the boundary integral acts as a term which \emph{penalizes} the deviation from zero boundary values. Such terms are necessary in the relaxed energy functionals considered here because the trace operator is \emph{not} continuous with respect to weak*-convergence; see Remark \ref{rem:boundarypenal} for more detail. 
\end{remark}
\subsection{Bulk approximation of boundary terms}
We now come to the main result of the present section, which shall play an important role in the existence proof in Section \ref{sec:main}.
\begin{theorem}[Boundary term approximation]\label{thm:bdryapprox}
Let $\Omega\subset\R^{2}$ be open and bounded with Lipschitz boundary, the latter being oriented by the outer unit normal $\nu_{\partial\Omega}\colon\partial\Omega\to\mathbb{S}^{1}$. Moreover, let $F\in\hold(\R_{\mathrm{sym}}^{2\times 2})$ be convex and of linear growth, see \eqref{eq:lgresh}. Then there exists a sequence $(\eta_{\delta})\subset\mathrm{Lip}_{0}(\Omega;[0,1])$ such that 
\begin{align}\label{eq:recessomaino}
\lim_{\delta\searrow 0}\int_{\Omega}F(\mathbb{T}(\eta_{\delta}\mathbf{u})) = \mathscr{F}_{0}^{*}[\uu;\Omega] =\int_{\Omega}F(\mathbb{T}\mathbf{u}) + \int_{\partial\Omega}F^{\infty}(-\mathrm{tr}_{\partial\Omega}(\mathbf{u})\otimes_{\mathbb{T}}\nu_{\partial\Omega})\dif\mathscr{H}^{1}
\end{align}
holds for all $\uu\in\bd(\Omega)$. Moreover, it is possible to arrange that each $\eta_{\delta}$ is compactly supported in $\Omega$. 
\end{theorem} 
\begin{figure}
\begin{center}
\begin{tikzpicture}[scale=1.3]
\draw[-,fill=red!10!white,black!10!white] (-1,1) [out = 0, in = 180] to (2.5,-0.5) -- (3,0) [out=180, in = 0] to (-0.5,1.5) -- (-1,1);
\draw[-,black] (-1,1) [out = 0, in = 180] to (2.5,-0.5);
\draw[-] (-0.9,1.1) [out = 0, in = 180] to (2.6,-0.4);
\draw[-] (-0.8,1.2) [out = 0, in = 180] to (2.7,-0.3);
\draw[-] (-0.7,1.3) [out = 0, in = 180] to (2.8,-0.2);
\draw[-] (-0.6,1.4) [out = 0, in = 180] to (2.9,-0.1);
\draw[-,black!10!white,fill=black!10!white,opacity=0.3] (-0.5,1.5) [out=180,in=0] to (-2,1.75) -- (-2,-1) -- (4,-1) -- (4,-0.25) [out= 150, in = 0] to (3,0) [out=180, in = 0] to (-0.5,1.5);
\draw[-] (-1,1) [out =180, in =0] to (-2,1.25);
\draw[-] (2.5,-0.5) [out =0, in =180] to (4,-0.75);
\draw[-,black,thick] (3,0) [out=180, in = 0] to (-0.5,1.5);
\draw[-,black,thick] (-0.5,1.5) [out=180,in=0] to (-2,1.75);
\draw[-,black,thick] (4,-0.25) [out= 150, in = 0] to (3,0);
\node[rotate=-9] at (3,-0.26) {$...$};
\node[rotate=-8] at (-0.95,1.3) {$...$};
\node[black!30!white] at (-1.5,-0.5) {\LARGE  $\mathsf{\Omega}$};
\node[black] at (-1.5,1.95) { $\mathsf{\partial\Omega}$};
\node[black] at (4.95,-0.25) {\footnotesize $\Phi(0,\partial\Omega)=\mathsf{\partial\Omega}$};
\node[black] at (4.65,-0.75) {\footnotesize $\Phi(1,\partial\Omega)$};
\end{tikzpicture}
\end{center}
\caption{The construction of the spatial cut-offs in the bulk-approximation of boundary terms.}\label{fig:collars}
\end{figure}
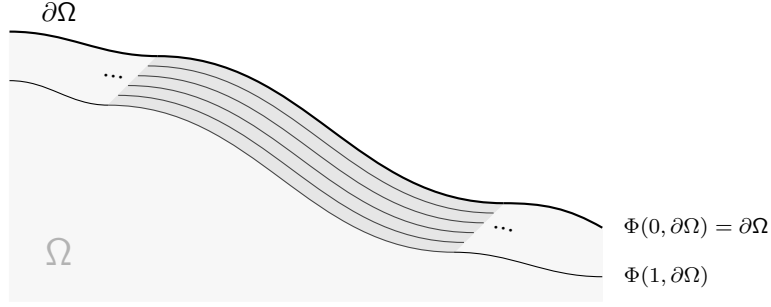
\begin{proof}
We begin with some preparations. Since $\partial\Omega$ is Lipschitz, we find a collar or deformation map $\Phi\colon [0,1]\times\partial\Omega\to\mathcal{O}\cap\overline{\Omega}$, where $\mathcal{O}\subset\R^{2}$ is a suitable open neighbourhood of $\partial\Omega$ in $\R^{n}$, with the following properties; these can be traced back to Hofman et al. \cite[Prop. 4.19]{Hofman} and Doktor \cite{Doktor}: 


\begin{enumerate}[label=(P\arabic*)]
    \item\label{item:coll1} $\Phi$ is Bi-Lipschitz onto its image, and the Bi-Lipschitz constants are uniformly bounded in $0\leq t\leq 1$, and the bounds only depend on the Lipschitz character of $\partial\Omega$. 
    \item\label{item:coll2} For any $t\in[0,1]$, $\Omega_{t}\coloneqq \Omega\setminus\Phi([0,t]\times\partial\Omega)$ has Lipschitz boundary $\partial\Omega_{t}$, and the underlying Lipschitz constants are uniformly bounded in $t\in[0,1]$.  
    \item\label{item:coll2a} There exists a constant $c>1$ with the following property: For all $s,t\in[0,1]$ and any $x\in\partial\Omega_{t}$, we have  
    \begin{align*}
   \frac{1}{c}|s-t|\leq \mathrm{dist}(x,\partial\Omega_{s})\leq c\,|s-t|.
    \end{align*}
    \item\label{item:coll3} $\Phi(0,\partial\Omega)=\partial\Omega$. 
    \item\label{item:coll4} \emph{Coordinate cylinder coverings}: There exists a finite family $\mathscr{Z}$ of coordinate cylinders $\bm{\mathscr{C}}$ which also form a family of coordinate cylinders for $\partial\Omega_{t}$ for any $0\leq t \leq 1$. Moreover, for each such coordinate cylinder $\mathscr{C}$, the following holds: If $f$ and $f_{t}$ are the Lipschitz graph parametrisations of $\partial\Omega$ and $\partial\Omega_{t}$ in $\mathscr{C}$, then there exists $c>0$ such that 
    \begin{align}\label{eq:robertweizen}
    \sup_{0\leq t\leq 1}\|\nabla f_{t}\|_{\lebe^{\infty}} \leq c\|\nabla f\|_{\lebe^{\infty}}\;\;\;\text{and}\;\;\;\nabla f_{t}\to \nabla f\;\;\;\text{$\mathscr{L}^{1}$-a.e.}. 
    \end{align}
\end{enumerate}
See Figure \ref{fig:collars} for this geometric set-up. 
In the following, we put $S_{\delta}\coloneqq \Phi((0,t)\times\partial\Omega)$. Next, for any sufficiently small $0<\delta<1$, we define a cut-off function  $\eta_{\delta}\colon\overline{\Omega}\to[0,1]$ via 
\begin{align}\label{eq:etacutoff}
\eta_{\delta}(x)\coloneqq \begin{cases} 
\frac{t}{\delta}&\;\text{if}\;x\in\partial\Omega_{t},\;0\leq t\leq\delta,\\ 
1&\;\text{if}\;x\in \Omega_{\delta}. 
\end{cases}
\end{align}
By \ref{item:coll1}, $\eta_{\delta}$ is well-defined. 

\emph{Step 1. A Lipschitz estimate.} We first claim that every $\eta_{\delta}$ is Lipschitz, and there exists $c>0$ such that 
\begin{align}\label{eq:tabealipschitz}
\|\nabla\eta_{\delta}\|_{\lebe^{\infty}(\Omega)}\leq \frac{c}{\delta}\qquad\text{for all $0<\delta<1$}. 
\end{align}
To see \eqref{eq:tabealipschitz}, we distinguish several constellations between $x,y\in\overline{\Omega}$.  \begin{itemize}
\item If $x\in\Omega\setminus\Phi([0,\delta]\times\partial\Omega)$ and $y\in\Phi(\{t\}\times\partial\Omega)$ for some $t\in[0,\delta]$, then \begin{align*}
|\eta_{\delta}(x)-\eta_{\delta}(y)|=\frac{|\delta-t|}{\delta} \stackrel{\text{\ref{item:coll2a}}}{\leq} \frac{c}{\delta}\,\mathrm{dist}(y,\partial\Omega_{\delta}) \stackrel{x\in\Omega\setminus\Phi([0,\delta]\times\partial\Omega)}{\leq} \frac{c}{\delta}|x-y|.
\end{align*} 
\item If $x\in\Phi(\{s\}\times\partial\Omega)$ and $y\in\Phi(\{t\}\times\partial\Omega)$ for some $0\leq s,t\leq\delta$, then 
\begin{align*}
|\eta_{\delta}(x)-\eta_{\delta}(y)| = \frac{|s-t|}{\delta} \stackrel{\text{\ref{item:coll2a}}}{\leq} \frac{c}{\delta}\mathrm{dist}(x,\partial\Omega_{t}) \stackrel{y\in\partial\Omega_{t}}{\leq} \frac{c}{\delta}|x-y|.
\end{align*} 
\end{itemize}
By symmetry, these are the only two non-trivial cases, and \eqref{eq:tabealipschitz} follows. Since $\eta_{\delta}|_{\partial\Omega}=0$, we have that $\eta_{\delta}\in\mathrm{Lip}_{0}(\Omega)$. 

 \emph{Step 2. Reduction to an upper bound.} Let $\mathbf{u}\in\bd(\Omega)$, and let $0<\delta<1$ be sufficiently small. Since $\eta_{\delta}\in\mathrm{Lip}_{0}(\Omega)$, we have $\eta_{\delta}\mathbf{u}\in\bd(\Omega)$ and, by  \eqref{eq:prodruleweak}, the following product rule:
\begin{align}\label{eq:prodrule}
\mathbb{T}(\eta_{\delta}\mathbf{u})=\eta_{\delta}\mathbb{T}\mathbf{u} + \mathbf{u}\otimes_{\mathbb{T}}\nabla\eta_{\delta}\mathscr{L}^{2}\mres S_{\delta}\qquad\text{as $\RR$-valued measures on $\Omega$}. 
\end{align}
We now split 
\begin{align}
\int_{\Omega} F(\mathbb{T}(\eta_{\delta}\mathbf{u})) & \stackrel{\eqref{eq:prodrule}}{=} \int_{\Omega} F(\eta_{\delta}\mathbb{T}\mathbf{u} + \mathbf{u}\otimes_{\mathbb{T}}\nabla\eta_{\delta}\mathscr{L}^{2}\mres S_{\delta}) \notag\\ 
& \;\; = \int_{\Omega} F(\eta_{\delta}\mathbb{T}\mathbf{u}) + \Big(\int_{S_{\delta}} F(\eta_{d}\mathbb{T}\mathbf{u} + \mathbf{u}\otimes_{\mathbb{T}}\nabla\eta_{\delta}\mathscr{L}^{2}) - \int_{S_{\delta}} F(\mathbf{u}\otimes_{\mathbb{T}}\nabla\eta_{\delta})\dif x\Big)\label{eq:splittin1} \\ 
& \;\, + \int_{S_{\delta}} F(\mathbf{u}\otimes_{\mathbb{T}}\nabla\eta_{\delta})\dif x - \int_{S_{\delta}} F(\eta_{\delta}\mathbb{T}\mathbf{u}) \eqqcolon \mathrm{I}_{\delta}+\mathrm{II}_{\delta}+\mathrm{III}_{\delta}-\mathrm{IV}_{\delta}.\notag
\end{align}
By the linear growth hypothesis on $F$ and dominated convergence, we have that 
\begin{align}\label{eq:split1}
\lim_{\delta\searrow 0}\mathrm{I}_{\delta} = \int_{\Omega}F(\mathbb{T}\mathbf{u}).
\end{align}
Next, since $F$ is convex and of linear growth, it is Lipschitz with some constant $L>0$; see Lemma \ref{lem:subdif}\ref{item:RobertDenk}. Because $\mathbb{T}\mathbf{u}\in\mathrm{RM}_{\mathrm{fin}}(\Omega;\R_{\mathrm{sym}}^{2\times 2})$ and $S_{\delta}\to\emptyset$ as $\delta\searrow 0$, we conclude by use of \eqref{eq:lgresh} that
\begin{align}\label{eq:split2}
\limsup_{\delta\searrow 0} (|\mathrm{II}_{\delta}| + |\mathrm{IV}_{\delta}|)\leq \limsup_{\delta\searrow 0} \Big((L+c_{3}) |\mathbb{T}\mathbf{u}|(S_{\delta})+c_{3}\mathscr{L}^{2}(S_{\delta})\Big) =0. 
\end{align}
Hence, in view of \eqref{eq:splittin1}--\eqref{eq:split2}, we arrive at 
\begin{align}\label{eq:tabea1}
\limsup_{\delta\searrow 0} \int_{\Omega}F(\mathbb{T}(\eta_{\delta}\mathbf{u})) \leq  \int_{\Omega}F(\mathbb{T}\mathbf{u}) + \limsup_{\delta\searrow 0}\int_{S_{\delta}}F(\mathbf{u}\otimes_{\mathbb{T}}\nabla\eta_{\delta})\dif x. 
\end{align}
Now assume for the time being that we are able to establish that 
\begin{align}\label{eq:tabea2}
\limsup_{\delta\searrow 0}\int_{S_{\delta}}F(\mathbf{u}\otimes_{\mathbb{T}}\nabla\eta_{\delta})\dif x \leq \int_{\partial\Omega}F^{\infty}(-\mathrm{tr}_{\partial\Omega}(\mathbf{u})\otimes_{\mathbb{T}}\nu_{\partial\Omega})\dif\mathscr{H}^{1}. 
\end{align}
We claim that this implies the assertion of the theorem. Since $\eta_{\delta}\mathbf{u}\to\mathbf{u}$ strongly in $\lebe^{1}(\Omega;\R^{2})$, we conclude as in \eqref{eq:tabeaprovesweak*} to find that $\eta_{\delta}\mathbf{u}\stackrel{*}{\rightharpoonup}\mathbf{u}$ in $\bd(\Omega)$. By Corollary \ref{cor:LSCplussbdryvalues}, the functional 
\begin{align*}
\mathscr{F}_{0}^{*}[\mathbf{w}] \coloneqq \int_{\Omega}F(\mathbb{T}\mathbf{w}) + \int_{\partial\Omega}F^{\infty}(-\mathrm{tr}_{\partial\Omega}(\mathbf{w})\otimes_{\mathbb{T}}\nu_{\partial\Omega})\dif\mathscr{H}^{1},\qquad\mathbf{w}\in\bd(\Omega),
\end{align*}
is lower semicontinuous with respect to weak*-convergence on $\bd(\Omega)$. Since $\mathrm{tr}_{\partial\Omega}(\eta_{\delta}\mathbf{u})=0$ on $\partial\Omega$ for any $0<\delta<1$, this immediately gives us 
\begin{align}\label{eq:split4}
\int_{\Omega}F(\mathbb{T}\mathbf{u}) + \int_{\partial\Omega}F^{\infty}(-\mathrm{tr}_{\partial\Omega}(\mathbf{u})\otimes_{\mathbb{T}}\nu_{\partial\Omega})\dif\mathscr{H}^{1} \leq \liminf_{\delta\searrow 0}\int_{\Omega}F(\mathbb{T}(\eta_{\delta}\mathbf{u})). 
\end{align}
Now, subject to \eqref{eq:tabea2}, \eqref{eq:tabea1} yields that the right-hand side of \eqref{eq:split4} can equally be bounded by its left-hand side, and this implies the claim. Hence, it remains to establish \eqref{eq:tabea2}. 

\emph{Step 2.} To this end, we claim in an intermediate step that it suffices to prove \eqref{eq:tabea2} subject to a  higher $\sobo^{1,1}$-regularity assumption. Namely, suppose that we can show  that 
\begin{align}\label{eq:tabea3}
\limsup_{\delta\searrow 0}\int_{S_{\delta}}F(\mathbf{v}\otimes_{\mathbb{T}}\nabla\eta_{d})\dif x \leq \int_{\partial\Omega}F^{\infty}(-\mathrm{tr}_{\partial\Omega}(\mathbf{v})\otimes_{\mathbb{T}}\nu_{\partial\Omega})\dif\mathscr{H}^{1} 
\end{align}
for all $\mathbf{v}\in\sobo^{1,1}(\Omega;\R^{2})$. We recall that, since $\partial\Omega$ is Lipschitz, $\sobo^{1,1}(\Omega;\R^{2})$ and $\bd(\Omega)$ have the same boundary trace space $\lebe^{1}(\partial\Omega;\R^{2})$; this means that the boundary trace operator is onto $\lebe^{1}(\partial\Omega;\R^{2})$ in both cases, see Lemma \ref{lem:traceoperator}\ref{item:tracex1}. Moreover, the restriction of the boundary trace operator on $\bd(\Omega)$ to $\sobo^{1,1}(\Omega;\R^{2})$ coincides with the boundary trace operator on $\sobo^{1,1}(\Omega;\R^{2})$. We thus may write $\mathrm{tr}_{\partial\Omega}$ for both boundary trace operators  without ambiguities.  For $\mathbf{u}\in\bd(\Omega)$, we may therefore choose a map  $\mathbf{v}\in\sobo^{1,1}(\Omega;\R^{2})$ such that 
\begin{align}\label{eq:tabea4}
\mathrm{tr}_{\partial\Omega}(\mathbf{u})=\mathrm{tr}_{\partial\Omega}(\mathbf{v})\qquad\text{$\mathscr{H}^{1}$-a.e. on $\partial\Omega$}.
\end{align}
Indeed, we may put $\vv\coloneqq\mathrm{ext}(\mathrm{tr}_{\partial\Omega}(\uu))$ with the extension operator from Remark \ref{rem:peetre}.

We claim that there exists a constant $c>0$ such that 
\begin{align}\label{eq:tabea5} 
\int_{S_{\delta}}|\mathbf{u}-\mathbf{v}|\dif x \leq c\,\delta\,(|\mathrm{E}\mathbf{u}|(S_{\delta})+\|\nabla\mathbf{v}\|_{\lebe^{1}(S_{\delta})})\qquad\text{for all sufficiently small $\delta>0$}. 
\end{align}
To this end, recall from \ref{item:coll2} that any $\Omega_{\delta}$ has Lipschitz boundary with Lipschitz characters  uniformly bounded in $0<\delta<1$. We now apply Lemma \ref{lem:GSTLemma} to the particular choice $U\coloneqq S_{\delta}$ and $\Gamma\coloneqq\partial\Omega\subset\partial U$. Moreover, by the Bi-Lipschitz property of $\Phi$, see \ref{item:coll1}, there exists $0<\theta<1$ such that 
\begin{align}\label{eq:unifHausTabea}
\theta\mathscr{H}^{1}(\Gamma)\leq \theta\mathscr{H}^{1}(\partial  S_{\delta}) \leq \mathscr{H}^{1}(\Gamma) \qquad \text{for all sufficiently small $\delta>0$}. 
\end{align}
Hence, Lemma \ref{lem:GSTLemma} implies that there exists $c=c(S_{\delta})>0$ such that 
\begin{align}\label{eq:tabea5a}
\|\mathbf{w}\|_{\lebe^{1}(S_{\delta})}\leq c\Big(\|\mathrm{tr}_{\partial\Omega}(\mathbf{w})\|_{\lebe^{1}(\partial\Omega)} + |\mathrm{E}\mathbf{w}|(S_{\delta})\Big)\qquad \text{for all}\;\mathbf{w}\in\bd(\Omega). 
\end{align}
Now consider  $\mathbf{w}\in\bd(\Omega)$ such that $\mathrm{tr}_{\partial\Omega}(\mathbf{w})=0$ $\mathscr{H}^{n-1}$-a.e. on $\partial\Omega$, so that \eqref{eq:tabea5a} becomes
\begin{align}\label{eq:tabea5b}
\|\mathbf{w}\|_{\lebe^{1}(S_{\delta})}\leq c\, |\mathrm{E}\mathbf{w}|(S_{\delta})\qquad \text{for all}\;\mathbf{w}\in\bd(\Omega)\;\text{with}\;\mathrm{tr}_{\partial\Omega}(\uu)=0.
\end{align}
Moreover, by \ref{item:coll1}--\ref{item:coll2a},  $\mathscr{L}^{2}(S_{\delta})$ is uniformly proportional to $\delta\mathscr{H}^{1}(\partial\Omega)$. The latter can be seen by the coarea formula from Lemma \ref{lem:coarea}, and we refer the reader to Step 3 from below for a similar argument. Based on these properties and keeping in mind the uniform boundedness of the Lipschitz characters of $\partial S_{\delta}$ and \eqref{eq:unifHausTabea}, 
a moment's reflection on scaling yields that the constant $c=c(S_{\delta})$ in \eqref{eq:tabea5b} can be chosen to be of the form $c\delta$ with some $c>0$ independent of $\delta$ and $\mathbf{w}$. In conclusion, we obtain that 
\begin{align}\label{eq:tabea6}
\|\mathbf{w}\|_{\lebe^{1}(S_{\delta})}\leq c\,\delta |\mathrm{E}\mathbf{w}|(S_{\delta})\qquad\text{for all}\;\mathbf{w}\in\bd(\Omega)\;\text{with}\;\mathrm{tr}_{\partial\Omega}(\mathbf{w})=0\;\text{$\mathscr{H}^{1}$-a.e. on $\partial\Omega$}. 
\end{align}
Based on \eqref{eq:tabea4}, we may apply \eqref{eq:tabea6} to $\mathbf{w}\coloneqq \mathbf{u}-\mathbf{v}$. Since $|z^{\mathrm{sym}}|\leq |z|$ for all $z\in\mathbb{R}^{2\times 2}$, we conclude \eqref{eq:tabea5}. 

By our present assumption, \eqref{eq:tabea3} is available for $\mathbf{v}\in\sobo^{1,1}(\Omega;\R^{2})$.  Again recalling that $F$ is Lipschitz with constant $L>0$, we  estimate by use of \eqref{eq:tabealipschitz}:
\begin{align}\label{eq:tabea7}
\begin{split}
\int_{S_{\delta}}|F(\mathbf{u}\otimes_{\mathbb{T}}\nabla\eta_{\delta})-F(\mathbf{v}\otimes_{\mathbb{T}}\nabla\eta_{\delta})|\dif x & \leq cL\int_{S_{\delta}}\frac{|\mathbf{u}-\mathbf{v}|}{\delta}\dif x \\ 
& \!\!\!\! \stackrel{\eqref{eq:tabea5}}{\leq} c(|\mathrm{E}\mathbf{u}|(S_{\delta})+\|\nabla\mathbf{v}\|_{\lebe^{1}(S_{\delta})}) \stackrel{\delta\searrow 0}{\longrightarrow} 0.
\end{split}
\end{align}
In consequence, we arrive at 
\begin{align*}
\limsup_{\delta\searrow 0} \int_{S_{\delta}}F(\mathbf{u}\otimes_{\mathbb{T}}\nabla\eta_{\delta})\dif x & \leq \limsup_{\delta\searrow 0}\int_{S_{\delta}}|F(\mathbf{u}\otimes_{\mathbb{T}}\nabla\eta_{\delta})-F(\mathbf{v}\otimes_{\mathbb{T}}\nabla\eta_{\delta})|\dif x  \\ & + \limsup_{\delta\searrow 0}\int_{S_{\delta}}F(\mathbf{v}\otimes_{\mathbb{T}}\nabla\eta_{\delta})\dif x \\ 
& \!\!\!\! \stackrel{\eqref{eq:tabea7}}{=} \limsup_{\delta\searrow 0}\int_{S_{\delta}}F(\mathbf{v}\otimes_{\mathbb{T}}\nabla\eta_{\delta})\dif x \\ 
& \!\!\!\!\stackrel{\eqref{eq:tabea3}}{\leq} \int_{\partial\Omega}F^{\infty}(-\mathrm{tr}_{\partial\Omega}(\mathbf{v})\otimes_{\mathbb{T}}\nu_{\partial\Omega})\dif\mathscr{H}^{1} \\
&\!\!\!\!\stackrel{\eqref{eq:tabea4}}{=} \int_{\partial\Omega}F^{\infty}(-\mathrm{tr}_{\partial\Omega}(\mathbf{u})\otimes_{\mathbb{T}}\nu_{\partial\Omega})\dif\mathscr{H}^{1},  
\end{align*}
and this is \eqref{eq:tabea2}. 

\emph{Step 3.} By what we have seen above, it suffices to establish \eqref{eq:tabea3} for $\mathbf{v}\in\sobo^{1,1}(\Omega;\R^{2})$. To this end, we firstly note that it is no loss of generality to assume that $F(0)=0$. Namely, if not, we consider \eqref{eq:tabea3} with the integrand $\widetilde{F}(z)\coloneqq F(z)-F(0)$. Subject to \eqref{eq:tabea3} applied to $\widetilde{F}$, the validity of \eqref{eq:tabea3} with the integrand $F$ then follows from  $\mathscr{L}^{2}(S_{\delta})\to 0$ and $\widetilde{F}^{\infty}=F^{\infty}$. Secondly, for any $0<t<\delta$, the outer unit normal to $\partial\Omega_{t}$ is given by 
\begin{align}\label{eq:tabeanormal}
\nu_{\partial\Omega_{t}}=-\frac{\nabla\eta_{\delta}}{|\nabla\eta_{\delta}|}\qquad\text{$\mathscr{H}^{1}$-a.e. on $\partial\Omega_{t}$}. 
\end{align}
Next we employ the coarea formula, see Lemma \ref{lem:coarea}; based on the conventions adopted therein, we put $f\coloneqq \eta_{\delta}$, whereby its Jacobian is given by $\mathbf{J}_{f}=|\nabla \eta_{\delta}|$. For future reference, we note that the convexity of $F$ implies that its difference quotients are increasing. Therefore, $F(0)=0$ gives us 
\begin{align}\label{eq:tabea9a}
\frac{1}{t}F(tz)\leq F^{\infty}(z)\qquad \text{for all}\;z\in\mathbb{R}_{\mathrm{sym}}^{2\times 2}\;\text{and all}\;t>0.
\end{align}
In consequence, we find 
\begin{align}
\int_{S_{\delta}}F(\mathbf{v}\otimes_{\mathbb{T}}\nabla\eta_{\delta})\dif x & = \int_{S_{\delta}\cap\{|\nabla\eta_{\delta}|> 0\}}\frac{1}{|\nabla\eta_{\delta}|}F\Big(|\nabla\eta_{\delta}|\Big(\mathbf{v}\otimes_{\mathbb{T}}\frac{\nabla\eta_{\delta}}{|\nabla\eta_{\delta}|}\Big)\Big)|\nabla\eta_{\delta}|\dif x \notag \\ 
& \!\!\!\! \stackrel{\eqref{eq:tabea9a}}{\leq} \int_{S_{\delta}\cap\{|\nabla\eta_{\delta}|>0\}} F^{\infty}\Big(\mathbf{v}\otimes_{\mathbb{T}}\frac{\nabla\eta_{\delta}}{|\nabla\eta_{\delta}|}\Big)\mathbf{J}_{f}\dif x \label{eq:tabea9}\\ 
& \!\!\!\! \stackrel{\eqref{eq:tabeanormal}}{=} \int_{0}^{1}\int_{\partial\Omega_{\delta s}} F^{\infty}(-\mathrm{tr}_{\partial\Omega_{\delta s}}(\mathbf{v})\otimes_{\mathbb{T}}\nu_{\partial\Omega_{\delta s}})\dif\mathscr{H}^{1}\dif s \eqqcolon \mathrm{V}, \notag
\end{align}
where we have used Lemma \ref{lem:coarea} in the ultimate line. Next, we claim that 
\begin{align}\label{eq:tabea10}
\lim_{\delta\searrow 0} \int_{0}^{1}\int_{\partial\Omega_{\delta s}} F^{\infty}(-\mathrm{tr}_{\partial\Omega_{\delta s}}(\mathbf{v})\otimes_{\mathbb{T}}\nu_{\partial\Omega_{\delta s}})\dif\mathscr{H}^{1}\dif s = \int_{\partial\Omega}F^{\infty}(-\mathrm{tr}_{\partial\Omega}(\mathbf{v})\otimes_{\mathbb{T}}\nu_{\partial\Omega})\dif\mathscr{H}^{1}. 
\end{align}
First suppose that, in addition, $\mathbf{v}\in\hold^{1}(\overline{\Omega};\R^{2})$. We pause to clarify the geometric set-up: Based on \ref{item:coll4}, we may choose (open) coordinate cylinders $\mathscr{C}_{1},...,\mathscr{C}_{N}$ which cover both $\partial\Omega$ and $\partial\Omega_{t}$ for all sufficiently small $0<t<1$. Subject to these coordinate cylinders, we choose a partition of unity $\psi_{1},...,\psi_{N}$; by the joint covering property, we may assume $(\psi_{j})_{j=1}^{N}$ to be a partition of unity of both $\partial\Omega$ and $\partial\Omega_{t}$ for all sufficiently small $0<t<1$; in particular, $\spt(\psi_{k})\subset\mathscr{C}_{k}$ for all $k\in\{1,...,N\}$. Let $0<s<1$ be arbitrary but fixed. We write 
\begin{align*}
&\left\vert\int_{\partial\Omega_{\delta s}} F^{\infty}(-\mathbf{v}(x)\otimes_{\mathbb{T}}\nu_{\partial\Omega_{\delta s}}(x))\dif\mathscr{H}^{1}(x) - \int_{\partial\Omega} F^{\infty}(-\mathbf{v}(x)\otimes_{\mathbb{T}}\nu_{\partial\Omega}(x))\dif\mathscr{H}^{1}(x) \right\vert \\ 
& \leq \sum_{k=1}^{N}\left\vert\int_{\partial\Omega_{\delta s}} \psi_{k}(x)F^{\infty}(-\mathbf{v}(x)\otimes_{\mathbb{T}}\nu_{\partial\Omega_{\delta s}}(x))\dif\mathscr{H}^{1}(x) - \right. \\ & \left. \;\;\;\;\;\;\;\;\;\;\;\;\;\;\;\;\;\;\;\;\;\;\;\;\;\;\;\;\;\;\;\;\;\;\;\;\;\;\;\;\;\;\;\;\;\;\;\;\;\;\;\;\;\;\;\;\,\int_{\partial\Omega} \psi_{k}(x) F^{\infty}(-\mathbf{v}(x)\otimes_{\mathbb{T}}\nu_{\partial\Omega}(x))\dif\mathscr{H}^{1}(x) \right\vert \\ 
& \!\!\!\! \!\!\!\!\!\!\! \stackrel{\spt(\psi_{k})\subset\mathscr{C}_{k}}{\leq} \sum_{k=1}^{N}\left\vert\int_{\mathscr{C}_{k}\cap\partial\Omega_{\delta s}} \psi_{k}(x)F^{\infty}(-\mathbf{v}(x)\otimes_{\mathbb{T}}\nu_{\partial\Omega_{\delta s}}(x))\dif\mathscr{H}^{1}(x) - \right. \\ & \left. \;\;\;\;\;\;\;\;\;\;\;\;\;\;\;\;\;\;\;\;\;\;\;\;\;\;\;\;\;\;\;\;\;\;\;\;\;\;\;\;\;\;\;\;\;\;\;\;\;\;\;\;\;\;\;\;\,\int_{\mathscr{C}_{k}\cap\partial\Omega} \psi_{k}(x) F^{\infty}(-\mathbf{v}(x)\otimes_{\mathbb{T}}\nu_{\partial\Omega}(x))\dif\mathscr{H}^{1}(x) \right\vert \\
& \eqqcolon \sum_{k=1}^{N} \mathrm{VI}_{k}.
\end{align*}
Rotating and translating the single cylinders if necessary, we may directly assume that 
\begin{align*}
\mathscr{C}_{k}=\ball_{r}^{(1)}(0)\times (-R,R), 
\end{align*}
and that $\mathscr{C}_{k}\cap\partial\Omega_{\delta s}$ and $\mathscr{C}_{k}\cap\partial\Omega$ can be written as 
\begin{align*}
\mathscr{C}_{k}\cap\partial\Omega_{\delta s}=\mathrm{graph}(f_{\delta s})\;\;\;\text{and}\;\;\;\mathscr{C}_{k}\cap\partial\Omega = \mathrm{graph}(f)
\end{align*}
for suitable functions $f_{\delta s},f\colon\ball_{r}^{(1)}(0)\to (-R',R')$ with $0<R'<R$; moreover, 
\begin{align*}
\mathscr{C}_{k}\cap\Omega_{\delta s}=\{(x',\theta)\in\mathscr{C}_{k}\colon\;\theta<f_{\delta s}(x')\}\;\;\;\text{and}\;\;\;\mathscr{C}_{k}\cap\Omega=\{(x',\theta)\in\mathscr{C}_{k}\colon\;\theta<f(x')\}. 
\end{align*}
For $\mathscr{L}^{1}$-a.e. $x'\in\ball\coloneqq\ball_{r}^{(1)}(0)$, the unit normal to $\mathrm{graph}(f_{\delta s})$ at $(x',f_{\delta s}(x'))$ is given by 
\begin{align}\label{eq:normaldebra}
\nu_{\partial\Omega_{\delta s}}(x',f_{\delta s}(x'))= \frac{(\nabla f_{\delta s}(x'),1)}{\sqrt{1+|\nabla f_{\delta s}(x')|^{2}}}
\end{align}
and analogously for $\nu_{\partial\Omega}(x',f(x'))$. Lastly, we denote the enclosed strip by 
\begin{align*}
\mathscr{S}_{k}^{\delta s}\coloneqq \{(x',\theta)\in\mathscr{C}_{k}\colon\; f_{\delta s}(x')<\theta < f(x')\}.
\end{align*}
We now estimate as follows:  
\begin{align*}
\mathrm{VI}_{k} & = \left\vert \int_{\ball}\psi_{k}(x',f_{\delta s}(x'))F^{\infty}(-\mathbf{v}(x',f_{\delta s}(x'))\otimes_{\mathbb{T}}\nu_{\partial\Omega_{\delta s}}(x',f_{\delta s}(x')))\sqrt{1+|\nabla f_{\delta s}(x')|^{2}}\right. \\ 
 &\left. \;\;\;\;\;\;\;\; -\psi_{k}(x',f(x'))F^{\infty}(-\mathbf{v}(x',f(x'))\otimes_{\mathbb{T}}\nu_{\partial\Omega}(x',f(x')))\sqrt{1+|\nabla f(x')|^{2}}\dif\mathscr{L}^{1}(x')\right\vert \\ 
 & \leq \left\vert \int_{\ball}\psi_{k}(x',f_{\delta s}(x'))F^{\infty}(-\mathbf{v}(x',f_{\delta s}(x'))\otimes_{\mathbb{T}}\nu_{\partial\Omega_{\delta s}}(x',f_{\delta s}(x')))\sqrt{1+|\nabla f_{\delta s}(x')|^{2}}\right.\\ 
  &\left. \;\;\;\;\;\;\;\; -\psi_{k}(x',f(x'))F^{\infty}(-\mathbf{v}(x',f_{\delta s}(x'))\otimes_{\mathbb{T}}\nu_{\partial\Omega_{\delta s}}(x',f_{\delta s}(x')))\sqrt{1+|\nabla f_{\delta s}(x')|^{2}}\dif\mathscr{L}^{1}(x')\right\vert \\ 
  & + \left\vert \int_{\ball}\psi_{k}(x',f(x'))F^{\infty}(-\mathbf{v}(x',f_{\delta s}(x'))\otimes_{\mathbb{T}}\nu_{\partial\Omega_{\delta s}}(x',f_{\delta s}(x')))\sqrt{1+|\nabla f_{\delta s}(x')|^{2}}\right.\\ 
  &\left. \;\;\;\;\;\;\;\; -\psi_{k}(x',f(x'))F^{\infty}(-\mathbf{v}(x',f(x'))\otimes_{\mathbb{T}}\nu_{\partial\Omega_{\delta s}}(x',f_{\delta s}(x')))\sqrt{1+|\nabla f_{\delta s}(x')|^{2}}\dif\mathscr{L}^{1}(x')\right\vert \\ 
   & + \left\vert \int_{\ball}\psi_{k}(x',f(x'))F^{\infty}(-\mathbf{v}(x',f(x'))\otimes_{\mathbb{T}}\nu_{\partial\Omega_{\delta s}}(x',f_{\delta s}(x')))\sqrt{1+|\nabla f_{\delta s}(x')|^{2}}\right.\\ 
  &\left. \;\;\;\;\;\;\;\; -\psi_{k}(x',f(x'))F^{\infty}(-\mathbf{v}(x',f(x'))\otimes_{\mathbb{T}}\nu_{\partial\Omega}(x',f(x')))\sqrt{1+|\nabla f_{\delta s}(x')|^{2}}\dif\mathscr{L}^{1}(x')\right\vert \\ 
   & + \left\vert \int_{\ball}\psi_{k}(x',f(x'))F^{\infty}(-\mathbf{v}(x',f(x'))\otimes_{\mathbb{T}}\nu_{\partial\Omega}(x',f(x')))\sqrt{1+|\nabla f_{\delta s}(x')|^{2}}\right.\\ 
  &\left. \;\;\;\;\;\;\;\; -\psi_{k}(x',f(x'))F^{\infty}(-\mathbf{v}(x',f(x'))\otimes_{\mathbb{T}}\nu_{\partial\Omega}(x',f(x')))\sqrt{1+|\nabla f(x')|^{2}}\dif\mathscr{L}^{1}(x')\right\vert \\ 
  & =: \int_{\ball} \mathrm{VI}_{k,1}\dif\mathscr{H}^{1} + ... + \int_{\ball}\mathrm{VI}_{k,4}\dif\mathscr{H}^{1} =: \mathrm{VII}_{k,1} + ... + \mathrm{VII}_{k,4}. 
\end{align*}
\emph{On $\mathrm{VII}_{k,1}$.} By Lemma \ref{lem:subdif}\ref{item:RobertDenk}, $F^{\infty}$ is Lipschitz together with $F^{\infty}(0)=0$. By the fundamental theorem of calculus we thus have, for $\mathscr{L}^{1}$-a.e. $x'\in\ball$, 
\begin{align}
|{\mathrm{VI}_{k,1}(x')}| & \leq c\,\mathrm{Lip}(\psi_{k})|f_{\delta s}(x')-f(x')|\,|\mathbf{v}(x',f_{\delta s}(x'))|\,\sqrt{1+|\nabla f_{\delta s}(x')|^{2}} \notag \\ 
& \leq c\,\mathrm{Lip}(\psi_{k})|f_{\delta s}(x')-f(x')|\,|\mathbf{v}(x',f_{\delta s}(x'))-\mathbf{v}(x',f(x'))|\,\sqrt{1+|\nabla f_{\delta s}(x')|^{2}} \notag \\ 
& + c\,\mathrm{Lip}(\psi_{k})|f_{\delta s}(x')-f(x')|\,|\mathbf{v}(x',f(x'))|\,\sqrt{1+|\nabla f_{\delta s}(x')|^{2}} \notag \\ 
& \leq c\,\mathrm{Lip}(\psi_{k})|f_{\delta s}(x')-f(x')|^{2}\,\Big(\int_{0}^{1}|(\partial_{x_{n}}\mathbf{v})(x',f_{\delta s}(x')+\vartheta(f(x')-f_{\delta s}(x')))|\dif\vartheta\Big)\times \label{eq:gehaltsverhandlung1}\\ & \times \,\sqrt{1+|\nabla f_{\delta s}(x')|^{2}} \notag\\ 
& + c\,\mathrm{Lip}(\psi_{k})|f_{\delta s}(x')-f(x')|\,|\mathbf{v}(x',f(x'))|\,\sqrt{1+|\nabla f_{\delta s}(x')|^{2}}. \notag
\end{align}
We recall \eqref{eq:robertweizen}, which gives us by virtue of the coarea formula 
\begin{align}\label{eq:matahari}
\mathrm{VII}_{k,1} \leq c\,\|f_{\delta s}-f\|_{\lebe^{\infty}(\ball)}\int_{\mathscr{S}_{k}^{\delta s}}|\nabla\mathbf{v}|\dif x  + c\,\|f_{\delta s}-f\|_{\lebe^{\infty}(\ball)} \int_{\mathscr{C}_{k}\cap\partial\Omega}|\mathbf{v}|\dif\mathscr{H}^{1}, 
\end{align}
where $c>0$ is independent of $\delta,s$ and $\mathbf{v}$.

\emph{On $\mathrm{VII}_{k,2}$.} Again, we recall that $F^{\infty}$ is Lipschitz. For $\mathscr{L}^{1}$-a.e. $x'\in\ball$, we estimate similarly as in the preceding step:
\begin{align}\label{eq:christine}
\begin{split}
|{\mathrm{VI}_{k,2}(x')}| & \leq c\,|\mathbf{v}(x',f_{\delta s}(x'))-\mathbf{v}(x',f(x'))|\sqrt{1+|\nabla f_{\delta s}(x')|^{2}} \\ 
& \leq c\,\Big(\int_{0}^{1}|(\partial_{x_{n}}\mathbf{v})(x',f_{\delta s}(x')+\vartheta(f(x')-f_{\delta s}(x')))|\dif\vartheta\Big)\times \\ &\times |f(x')-f_{\delta s}(x')|\sqrt{1+|\nabla f_{\delta s}(x')|^{2}}
\end{split}
\end{align}
Again recalling \eqref{eq:robertweizen}, the coarea formula yields 
\begin{align}\label{eq:mathari1a}
\mathrm{VII}_{k,2} & \leq c\,\int_{\mathscr{S}_{k}^{\delta s}}|\nabla\mathbf{v}|\dif x, 
\end{align}
and $c>0$ is independent of $\delta$, $s$ and $\mathbf{v}$. 

\emph{On $\mathrm{VII}_{k,3}$.} For $\mathscr{L}^{1}$-a.e. $x'\in\ball$, we abbreviate $a_{\delta s}(x')\coloneqq \sqrt{1+|\nabla f_{\delta s}(x')|^{2}}$ and 
$a(x')\coloneqq \sqrt{1+|\nabla f(x')|^{2}}$. Since $z\mapsto\sqrt{1+|z|^{2}}$ in itself is convex and of linear growth, it is Lipschitz. Therefore, we have the Lipschitz estimate 
\begin{align}\label{eq:figaro}
|a_{\delta s}(x')-a(x')| \leq c |\nabla f_{\delta s}(x')-\nabla f(x')|.
\end{align}
We now employ \eqref{eq:normaldebra} to find that 
\begin{align}
 |\nu_{\partial\Omega_{\delta s}}(x',f_{\delta s}(x'))-\nu_{\partial\Omega}(x',f(x'))| & \leq \left\vert \frac{(\nabla f_{\delta s}(x'),1)}{a_{\delta s}(x')}-\frac{(\nabla f(x'),1)}{a(x')}\right\vert \notag\\ 
& \leq \frac{1}{a_{\delta s}(x')a(x')}|a(x')(\nabla f_{\delta s}(x'),1)-a_{\delta s}(x')(\nabla f(x'),1)| \notag\\ 
& \leq \frac{1}{a_{\delta s}(x')a(x')}|(a(x')-a_{\delta ,s}(x'))(\nabla f_{\delta,s}(x'),1)| \label{eq:touchdown}\\ 
& + \frac{1}{a(x')}|(\nabla f_{\delta s}(x'),1)-(\nabla f(x'),1)| \notag\\ 
& \!\!\!\!\!\!\!\!\!\! \stackrel{\eqref{eq:robertweizen}, \eqref{eq:figaro}}{\leq}  c |\nabla f_{\delta s}(x')-\nabla f(x')|. \notag
\end{align}
This implies that 
\begin{align*}
|{\mathrm{VI}_{k,3}(x')}| & \leq c\, |\mathbf{v}(x',f(x'))|\,|\nu_{\partial\Omega_{\delta s}}(x',f_{\delta s}(x'))-\nu_{\partial\Omega}(x',f(x'))|\sqrt{1+|\nabla f(x')|^{2}} \\ 
& \!\!\!\! \stackrel{\eqref{eq:touchdown}}{\leq} c\,|\mathbf{v}(x',f(x'))|\, |\nabla f_{\delta s}(x')-\nabla f(x')|\,\sqrt{1+|\nabla f(x')|^{2}},  
\end{align*}
and an integration with respect to $x'\in\ball$ yields 
\begin{align}\label{eq:matahari3}
\mathrm{VII}_{k,3} & \leq c\,\int_{\ball}|\mathbf{v}(x',f(x'))|\, |\nabla f_{\delta s}(x')-\nabla f(x')|\,\sqrt{1+|\nabla f(x')|^{2}}\dif\mathscr{L}^{1}(x'). 
\end{align}
\emph{On $\mathrm{VII}_{k,4}$.} For $\mathscr{L}^{1}$-a.e. $x'\in\ball$, we recall \eqref{eq:figaro} to find 
\begin{align*}
|\mathrm{VI}_{k,4}(x')| & \leq |\mathbf{v}(x',f(x'))|\,|\nabla f_{\delta s}(x')-\nabla f(x')|
\end{align*}
and therefore, again by \eqref{eq:robertweizen}, 
\begin{align}\label{eq:mathari4}
\begin{split}
\mathrm{VII}_{k,4} & \leq c\,\int_{\ball}|\mathbf{v}(x',f(x'))|\, |\nabla f_{\delta s}(x')-\nabla f(x')|\,\sqrt{1+|\nabla f(x')|^{2}}\dif\mathscr{L}^{1}(x').
\end{split}
\end{align}
We now synthesise inequalities  \eqref{eq:matahari}--\eqref{eq:mathari4} to arrive at 
\begin{align}
&\left\vert\int_{\partial\Omega_{\delta s}} F^{\infty}(-\mathbf{v}(x)\otimes_{\mathbb{T}}\nu_{\partial\Omega_{\delta s}}(x))\dif\mathscr{H}^{1}(x) - \int_{\partial\Omega} F^{\infty}(-\mathbf{v}(x)\otimes_{\mathbb{T}}\nu_{\partial\Omega}(x))\dif\mathscr{H}^{1}(x) \right\vert \notag\\
& \leq c\,\|f_{\delta s}-f\|_{\lebe^{\infty}(\ball)}\int_{\mathscr{S}_{k}^{\delta s}}|\nabla\mathbf{v}|\dif x  + c\,\|f_{\delta s}-f\|_{\lebe^{\infty}(\ball)} \int_{\mathscr{C}_{k}\cap\partial\Omega}|\mathbf{v}|\dif\mathscr{H}^{1}\label{eq:zeit}\\ 
& + c\,\int_{\mathscr{S}_{k}^{\delta s}}|\nabla\mathbf{v}|\dif x \notag\\ 
& + c\,\int_{\ball}|\mathbf{v}(x',f(x'))|\, |\nabla f_{\delta s}(x')-\nabla f(x')|\,\sqrt{1+|\nabla f(x')|^{2}}\dif\mathscr{H}^{1}(x')\notag
\end{align}
for $\mathbf{v}\in\hold^{1}(\overline{\Omega};\R^{2})$. Let us now lift the preceding inequality to $\sobo^{1,1}(\Omega;\R^{2})$. Since $\partial\Omega$ is Lipschitz, $\hold^{\infty}(\overline{\Omega};\R^{2})$ is dense in $\sobo^{1,1}(\Omega;\R^{2})$ with respect to the norm topology. On $\sobo^{1,1}(\Omega;\R^{2})$, both the boundary trace operators \emph{as well as} interior trace operators along $(n-1)$-dimensional Lipschitz manifolds $\Sigma$ are continuous as maps 
\begin{align}\label{eq:kennzeichen}
\mathrm{tr}_{\partial\Omega}\colon\sobo^{1,1}(\Omega;\R^{2})\to \lebe^{1}(\partial\Omega;\R^{2})\;\;\;\text{and}\;\;\;\mathrm{tr}_{\Sigma}\colon\sobo^{1,1}(\Omega;\R^{2})\to\lebe^{1}(\Sigma;\R^{2})
\end{align}
with respect to the underlying norm topologies; the latter is important to pass to the limit on the left-hand side of the \eqref{eq:zeit}. By the continuity properties of the traces with respect to the $\sobo^{1,1}$-norm topology and since $F^{\infty}$ is Lipschitz, we infer that \eqref{eq:zeit} holds true for $\mathbf{v}\in\sobo^{1,1}(\Omega;\R^{2})$. 

We take the resulting inequality and integrate from $s=0$ to $s=1$. This gives us 
\begin{align*}
&\int_{0}^{1}\left\vert\int_{\partial\Omega_{\delta s}} F^{\infty}(-\mathbf{v}(x)\otimes_{\mathbb{T}}\nu_{\partial\Omega_{\delta s}}(x))\dif\mathscr{H}^{1}(x) - \int_{\partial\Omega} F^{\infty}(-\mathbf{v}(x)\otimes_{\mathbb{T}}\nu_{\partial\Omega}(x))\dif\mathscr{H}^{1}(x) \right\vert \dif s\notag\\
& \leq c\,\int_{0}^{1}\|f_{\delta s}-f\|_{\lebe^{\infty}(\ball)}\int_{\mathscr{S}_{k}^{\delta s}}|\nabla\mathbf{v}|\dif x  + c\,\|f_{\delta s}-f\|_{\lebe^{\infty}(\ball)} \int_{\mathscr{C}_{k}\cap\partial\Omega}|\mathbf{v}|\dif\mathscr{H}^{1}\dif s\\ 
& + c\,\int_{0}^{1}\int_{\mathscr{S}_{k}^{\delta s}}|\nabla\mathbf{v}|\dif x \dif s\notag\\ 
& + c\,\int_{0}^{1}\int_{\ball}|\mathbf{v}(x',f(x'))|\, |\nabla f_{\delta s}(x')-\nabla f(x')|\,\sqrt{1+|\nabla f(x')|^{2}}\dif\mathscr{L}^{1}(x')\dif s.\notag
\end{align*}
Since $f_{\delta s}\to f$ uniformly on $\ball$, $\mathrm{tr}_{\partial\Omega}(\mathbf{v})\in\lebe^{1}(\partial\Omega;\R^{2})$ and $\nabla\mathbf{v}\in\lebe^{1}(\Omega;\R^{2\times 2})$, it follows from \eqref{eq:robertweizen} and dominated convergence that 
\begin{align*}
\lim_{\delta\searrow 0} \int_{0}^{1}& \left\vert\int_{\partial\Omega_{\delta s}} F^{\infty}(-\mathrm{tr}_{\partial\Omega_{\delta s}}(\mathbf{v})(x)\otimes_{\mathbb{T}}\nu_{\partial\Omega_{\delta s}}(x))\dif\mathscr{H}^{1}(x) \right. \\ & \left.\;\;\;\;\;\;\;\;\;\;\;\;\;\;\;\;\;\;\;\;\;\;\;\;\;\;\;\;\;\;\;\; -\int_{\partial\Omega} F^{\infty}(-\mathrm{tr}_{\partial\Omega}(\mathbf{v})(x)\otimes_{\mathbb{T}}\nu_{\partial\Omega}(x))\dif\mathscr{H}^{1}(x) \right\vert  \dif s = 0, 
\end{align*}
recalling that $\mathscr{L}^{2}(\mathscr{S}_{k}^{\delta s})\to 0$ as $\delta \searrow 0$. Hence, \eqref{eq:tabea10} follows. In view of \eqref{eq:tabea9}, this implies \eqref{eq:tabea3}. 

Lastly, we outline the modifications to achieve compact support. In this case, we replace $\eta_{\delta}$ as given in \eqref{eq:etacutoff} by 
\begin{align*}
\eta_{\delta}(x)\coloneqq \begin{cases} 
0 &\;\text{if}\;x\in \Omega\setminus\Omega_{\delta/2}, \\
\frac{2}{\delta}t-1&\;\text{if}\;x\in\partial\Omega_{t},\;\frac{\delta}{2}\leq t\leq\delta,\\ 
1&\;\text{if}\;x\in \Omega_{\delta}. 
\end{cases}
\end{align*}
We may then follow the steps leading to Step 3, where \eqref{eq:tabea9} is now replaced by 
\begin{align*}
\int_{S_{\delta}}F(\mathbf{v}\otimes_{\mathbb{T}}\nabla\eta_{\delta})\dif x & \leq \int_{0}^{1}\int_{\partial\Omega_{\frac{\delta}{2}(s+1)}} F^{\infty}(-\mathrm{tr}_{\partial\Omega_{\frac{\delta}{2}(s+1)}}(\mathbf{v})\otimes_{\mathbb{T}}\nu_{\partial\Omega_{\frac{\delta}{2}( s+1)}})\dif\mathscr{H}^{1}\dif s \\ 
& \leq \int_{0}^{1} \Big(\int_{\partial\Omega_{\frac{\delta}{2}(s+1)}} F^{\infty}(-\mathrm{tr}_{\partial\Omega_{\frac{\delta}{2}(s+1)}}(\mathbf{v})\otimes_{\mathbb{T}}\nu_{\partial\Omega_{\frac{\delta}{2}( s+1)}})\dif\mathscr{H}^{1} \Big. \\ 
& \Big. \;\;\;\;\;\;\;\;\; - \int_{\partial\Omega_{{\delta}/{2}}} F^{\infty}(-\mathrm{tr}_{\partial\Omega_{{\delta}/{2}}}(\mathbf{v})\otimes_{\mathbb{T}}\nu_{\partial\Omega_{{\delta}/{2}}})\dif\mathscr{H}^{1}\Big)\dif s \\ 
& + \Big( \int_{\partial\Omega_{{\delta}/{2}}} F^{\infty}(-\mathrm{tr}_{\partial\Omega_{{\delta}/{2}}}(\mathbf{v})\otimes_{\mathbb{T}}\nu_{\partial\Omega_{{\delta}/{2}}})\dif\mathscr{H}^{1}\Big. \\ 
& \Big. \;\;\;\;\;\;\;\;\; - \int_{\partial\Omega} F^{\infty}(-\mathrm{tr}_{\partial\Omega}(\mathbf{v})\otimes_{\mathbb{T}}\nu_{\partial\Omega})\dif\mathscr{H}^{1}\Big) \\ 
& + \int_{\partial\Omega} F^{\infty}(-\mathrm{tr}_{\partial\Omega}(\mathbf{v})\otimes_{\mathbb{T}}\nu_{\partial\Omega})\dif\mathscr{H}^{1} \eqqcolon \widetilde{\mathrm{V}}_{1} + \widetilde{\mathrm{V}}_{2} + \widetilde{\mathrm{V}}_{3}.  
\end{align*}
We may then argue as in \eqref{eq:tabea10}ff. to see that $\widetilde{\mathrm{V}}_{1}=0$, and as in \eqref{eq:zeit}ff. to see that $\widetilde{\mathrm{V}}_{2}=0$ too. From  here, the claim follows with compactly supported cut-offs. The proof is complete. 
\end{proof}
For our applications in Section \ref{sec:main}, we moreover record the following uniform bounds. 
\begin{corollary}\label{cor:unibounds}
In the situation of Theorem \ref{thm:bdryapprox}, there exists a constant $c>0$ such that 
\begin{align}\label{eq:unifboundcutoff}
|\mathbb{T}(\eta_{\delta}\uu)|(\Omega)\leq c\,(\mathscr{L}^{n}(\Omega)+\|\uu\|_{\bd(\Omega)})\qquad\text{for all}\;\uu\in\bd(\Omega) 
\end{align}
and all sufficiently small $\delta>0$. 
\end{corollary}
\begin{proof}
We note that, whenever $\mathbf{w}\in\sobo^{1,1}(\Omega;\R^{2})$, then 
\begin{align}\label{eq:bmehldau}
\int_{S_{\delta}}\left\vert\frac{\mathbf{w}}{\delta}\right\vert\dif x \leq c\big(\|\nabla\mathbf{w}\|_{\lebe^{1}(S_{\delta})} + \|\mathrm{tr}_{\partial\Omega}(\mathbf{w})\|_{\lebe^{1}(\partial\Omega)} \big)
\end{align}
holds for all sufficiently small $\delta>0$, where $c>0$ is a constant independent of $\delta$. This follows as, but can be seen much easier than the corresponding variant for the symmetric gradient as used in the previous proof. Let $\uu\in\bd(\Omega)$. Setting $\vv\coloneqq\mathrm{ext}(\mathrm{tr}_{\partial\Omega}(\uu))$ with the (nonlinear) extension operator $\mathrm{ext}\colon\lebe^{1}(\partial\Omega;\R^{2})\to\sobo^{1,1}(\Omega;\R^{2})$ from Remark \ref{rem:peetre}, we have the boundedness property 
\begin{align}\label{eq:boundedness}
\|\vv\|_{\sobo^{1,1}(\Omega)}\leq c\,\|\mathrm{tr}_{\partial\Omega}(\uu)\|_{\lebe^{1}(\partial\Omega)},  
\end{align}
where $c>0$ is independent of $\uu$; see \eqref{eq:tabea4}ff.. Therefore, we conclude that 
\begin{align*}
|\mathbb{T}(\eta_{\delta}\uu)|(\Omega) & \leq \,|\mathbb{T}\uu|(\Omega)+\int_{S_{\delta}}|\uu\otimes_{\mathbb{T}}\nabla\eta_{\delta}|\dif x \\ 
& \leq  |\mathbb{T}\uu|(\Omega) +  \int_{S_{\delta}}|(\uu-\vv)\otimes_{\mathbb{T}}\nabla\eta_{\delta}|\dif x +  \int_{S_{\delta}}|\vv\otimes\nabla\eta_{\delta}|\dif x \\ 
& \!\!\!\!\!\!\!\!\!\! \stackrel{\eqref{eq:tabealipschitz},\,\eqref{eq:tabea5}}{\leq}  |\mathbb{T}\uu|(\Omega) + c\,|\mathrm{E}\uu|(S_{\delta})+ c\|\nabla\mathbf{v}\|_{\lebe^{1}(S_{\delta})} + c\int_{S_{\delta}}|\vv\otimes\nabla\eta_{\delta}|\dif x \\ 
& \!\!\!\!\!\!\!\!\! \stackrel{\eqref{eq:tabealipschitz}, \eqref{eq:bmehldau}}{\leq} c\,|\mathrm{E}\uu|(\Omega) + c\,\|\nabla\mathbf{v}\|_{\lebe^{1}(S_{\delta})} + c\,\|\mathrm{tr}_{\partial\Omega}(\vv)\|_{\lebe^{1}(\partial\Omega)} \\ 
& \!\!\!\! \stackrel{\eqref{eq:boundedness}}{\leq} c\,|\mathrm{E}\mathbf{u}|(\Omega) + \|\mathrm{tr}_{\partial\Omega}(\uu)\|_{\lebe^{1}(\Omega)} \\ 
& \leq c\,\|\uu\|_{\bd(\Omega)}, 
\end{align*}
where $c>0$ is independent of $\delta>0$ and $\uu$; see Lemma \ref{lem:traceoperator}\ref{item:tracex1} for the ultimate inequality. The proof is complete. 
\end{proof}
\begin{remark}\label{rem:bulk1}
One of the key points of the proof of Theorem \ref{thm:bdryapprox}  is the passage from trace of $\bd$-functions to those of $\sobo^{1,1}$-functions, see \eqref{eq:tabea4}. Namely, if we work with $\bd$-functions throughout \emph{without} switching to $\sobo^{1,1}$, then it is precisely at \eqref{eq:christine} where the above proof breaks down; by Ornstein's Non-Inequality, the appearance of the full gradients renders the resulting inequality useless. Moreover, by the complicated structure of $\mathbb{T}\uu$, it is fully unclear how to directly control the emerging terms purely by $\mathbb{T}\uu$ or $\E\uu$, respectively. Here, the above proof admits the reduction to $\sobo^{1,1}$-maps, and essentially works because in view of traces, the passage to $\sobo^{1,1}$-maps comes with no additional creation of energy, see \eqref{eq:tabea7}. This moreover strengthens and extends the informal metaprinciple from \cite{DieningGmeineder} that Ornstein's Non-Inequality becomes invisible when passing to lower order trace estimates.
\end{remark}
\begin{remark}\label{rem:bulk2}
Let us note that, in the proof of Theorem \ref{thm:bdryapprox},  the passage to $\sobo^{1,1}$-maps (and not $\bv$-maps) is favourable due to \eqref{eq:kennzeichen}. Indeed, whereas the boundary trace operator is continuous with respect to the strict topology, the interior trace operators are not; see, e.g., \cite[Thm. 6.9(b)]{GmRaVS}. 
\end{remark}
Lastly, the particular structure of the Hibler deformation tensor $\mathbb{T}\uu$ is not essential for the above argument; in fact, one can employ any $\mathbb{C}$-elliptic operator. Since the proof is fully analogous, we confine ourselves to stating the result; to this end, note that the boundary trace space of $\bv^{\mathbb{A}}(\Omega)$ on open and bounded sets $\Omega$ with Lipschitz boundaries is $\lebe^{1}(\partial\Omega;V)$.
\begin{corollary}\label{cor:Celliptic}
Let $\Omega\subset\R^{n}$ be open and bounded with Lipschitz boundary, the latter being oriented by the outer unit normal $\nu_{\partial\Omega}\colon\partial\Omega\to\mathbb{S}^{n-1}$. Let $\mathbb{A}$ be a $\mathbb{C}$-elliptic differential operator of the form \eqref{eq:diffopform}. Moreover, let $F\in\hold(W)$ be convex and of linear growth, meaning that \eqref{eq:lgresh} holds for all $z\in W$. Then there exists a sequence $(\eta_{\delta})\subset\mathrm{Lip}_{0}(\Omega;[0,1])$ such that 
\begin{align}\label{eq:recessomainoCell}
\lim_{\delta\searrow 0}\int_{\Omega}F(\mathbb{A}(\eta_{\delta}\mathbf{u})) = \mathscr{F}_{0}^{*}[\uu;\Omega] \coloneqq\int_{\Omega}F(\mathbb{A}\mathbf{u}) + \int_{\partial\Omega}F^{\infty}(-\mathrm{tr}_{\partial\Omega}(\mathbf{u})\otimes_{\mathbb{A}}\nu_{\partial\Omega})\dif\mathscr{H}^{n-1}
\end{align}
holds for all $\uu\in\bv^{\mathbb{A}}(\Omega)$. Moreover, it is possible to arrange that each $\eta_{\delta}$ is compactly supported in $\Omega$. Lastly, Corollary \eqref{cor:unibounds} remains valid with the natural modifications.
\end{corollary}

\section{Variational solutions of the momentum balance equation}\label{sec:main}
\subsection{Variational solutions and main result}\label{sec:varsol} 
In this section, we introduce a suitable  notion of solutions for the momentum balance equation \eqref{eq:hibler1a} that is driven by the relaxed energies from Section \ref{sec:relaxthebulk}. To this end, we now specify our assumptions on the potentials, data and forces. In order to be able to incorporate different stress-strain relations, we suppose that the potential 
$F\colon\RR\to\R$ satisfies the following: 
\begin{enumerate}[label=(F\arabic*)]
\item\label{item:Fprop1} $F$ is convex and of linear growth, meaning that there exist constants $c_{1},c_{3}>0$ and $c_{2}\in\R$ such that 
\begin{align}\label{eq:lingrowth}
c_{1}|z|-c_{2}\leq F(z)\leq c_{3}(1+|z|)\qquad\text{for all}\;z\in\rsym^{2\times 2},  
\end{align}
\item\label{item:Fprop2} $F(0)=0$ and 
\item\label{item:Fprop3} $F\in\hold^{1}(\RR\setminus\{0\})$. 
\end{enumerate}
Both stress-strain relations as discussed in Section \ref{sec:modelling} fit into this framework. Throughout, let $\Omega\subset\R^{2}$ be open and bounded with Lipschitz boundary, and let $T>0$ be a final time. Pertaining to our discussion in Sections \ref{sec:hiblerintro} and \ref{sec:modelling}, we assume that 
\begin{align}\label{eq:atmosphericmain}
\mathbf{f}\coloneqq -mg\nabla H + \bm{\tau}_{\mathrm{atm}} \in \lebe^{2}(\Omega_{T};\R^{2}). 
\end{align}
Recalling  \eqref{eq:cutoff} and \eqref{eq:simpleeta}, we  moreover suppose that the underlying ocean forces for a given horizontal velocity field $\uu\colon\Omega_{T}\to\R^{2}$ are of the form 
\begin{align}\label{eq:oceanicmain}
\tocean(\uu)\coloneqq \widetilde{\eta}(\mathbf{U}_{\mathrm{ocean}}-\uu)\coloneqq c\eta(|\mathbf{U}_{\mathrm{ocean}}-\uu|)R_{\mathrm{ocean}}(\mathbf{U}_{\mathrm{ocean}}-\uu),
\end{align}
where $c>0$ is a constant, $R_{\mathrm{ocean}}\in\mathrm{SO}(2)$, and $\eta\in\hold^{1}(\R_{\geq 0};\R_{\geq 0})$ satisfies, for some fixed $\gamma\in (0,1)$ and constants $0<N_{1}<N_{2}<\infty$, 
\begin{align}\label{eq:oceanicassump}
\eta(0)=0,\;\;\;\eta\;\text{is linear on $[0,N_{1}]$ and}\;\;\;\eta(s)=s^{-\gamma}\;\text{for all}\;s\in[N_{2},\infty). 
\end{align}
We moreover assume that 
\begin{align}\label{eq:oceanicmean}
\mathbf{U}_{\mathrm{ocean}}\in \lebe^{\infty}(0,T;\lebe^{2}(\Omega;\R^{2})).
\end{align}
and that the initial datum satisfies 
\begin{align}\label{eq:initialvalue}
\uu_{0}\in\bd_{c}(\Omega),\;\;\;\text{meaning that}\;\uu_{0}\in\bd(\Omega)\;\text{and}\;\spt(\uu_{0})\;\text{is compact in $\Omega$}. 
\end{align}
For the following, it is convenient to record that \eqref{eq:oceanicassump} implies that 
\begin{align}\label{eq:oceanicHoelder}
\widetilde{\eta}\colon \R^{2} \to\R^{2}\qquad\text{is both globally Lipschitz and $(1-\gamma)$-H\"{o}lder}. 
\end{align}
We emphasize that these two regularity assumptions are independent. More precisely, the $(1-\gamma)$-H\"{o}lder property  gives better bounds for large values of the arguments.

Based on \ref{item:Fprop1}--\ref{item:Fprop3} and \eqref{eq:atmosphericmain}--\eqref{eq:oceanicHoelder}, we now derive the notion of solutions to be used in the sequel. This, in turn, is strongly inspired by Lichnewsky \& Temam \cite{LichnewskyTemam} and the corresponding variational solutions for the related total variation flow, see, e.g., \cite{HZ}. Since our setting is vectorial and  different in various respects, we give the details: For the time being, let us moreover assume that $F\in\hold^{1}(\RR)$.  Let $\mathbf{u}\colon(0,T)\times\overline{\Omega}\to\R^{n}$ be a sufficiently regular solution of 
\begin{align}\label{eq:motiv}
\partial_{t}\mathbf{u} = \mathbb{T}^{*}(F'(\mathbb{T}\uu))+\mathbf{f}+\tocean(\uu)\qquad\text{in}\;\Omega_{T}
\end{align}
subject to $\mathbf{u}(t,\cdot)=0$ on $\partial\Omega$ for all $0<t<T$. For $\vv\in\hold^{\infty}([0,T]\times\overline{\Omega};\R^{2})$, we multiply \eqref{eq:motiv} with $\bm{\varphi}\coloneqq\mathbf{v}-\mathbf{u}$ and integrate over $\Omega$. For a fixed time $0<t<T$, we then find by virtue of $\mathbf{u}(t,\cdot)=0$ on $\partial\Omega$: 
\begin{align*}
\int_{\Omega}&(\partial_{t}\mathbf{u}(t,\cdot))(\mathbf{v}(t,\cdot)-\mathbf{u}(t,\cdot))\dif x    = \int_{\Omega}\mathbb{T}^{*}(F'(\mathbb{T}\uu(t,\cdot)))\cdot (\mathbf{v}(t,\cdot)-\mathbf{u}(t,\cdot)) \dif x  \\ 
& + \int_{\Omega}\mathbf{f}(t,\cdot)\cdot(\vv(t,\cdot)-\uu(t,\cdot))\dif x + \int_{\Omega}\tocean(\uu(t,\cdot))\cdot(\vv(t,\cdot)-\uu(t,\cdot))\dif x \\
 &  \!\!\!\! \stackrel{\eqref{eq:IBPscalar}}{=} \int_{\partial\Omega}F'(\mathbb{T}\uu(t,\cdot))\cdot (\mathbf{v}(t,\cdot)\otimes_{\mathbb{T}}\nu_{\partial\Omega})\dif\mathscr{H}^{1} -\int_{\Omega}F'(\mathbb{T}\uu(t,\cdot))\cdot\mathbb{T}(\mathbf{v}(t,\cdot)-\mathbf{u}(t,\cdot))\dif x \\ 
 & + \int_{\Omega}\mathbf{f}(t,\cdot)\cdot(\vv(t,\cdot)-\uu(t,\cdot))\dif x + \int_{\Omega}\tocean(\uu(t,\cdot))\cdot(\vv(t,\cdot)-\uu(t,\cdot))\dif x \\ 
 & =:  \mathrm{T}_{1} - \mathrm{T}_{2} + \mathrm{T}_{3} + \mathrm{T}_{4}.
\end{align*}
Since $F$ is convex, we have \eqref{eq:subgradientdef} and thus for any $z_{0}\in\R_{\mathrm{sym}}^{2\times 2}$ that 
\begin{align}\label{eq:convex}
F(z_{0})+F'(z_{0})\cdot (z-z_{0})\leq F(z)\qquad \text{for all}\;z\in\rsym^{2\times 2}. 
\end{align}
In the present situation, this inequality gives us 
\begin{align}\label{eq:motiv2}
-\mathrm{T}_{2} \geq \int_{\Omega}F(\mathbb{T}\uu(t,\cdot))\dif x - \int_{\Omega}F(\mathbb{T}\vv(t,\cdot))\dif x. 
\end{align}
For the corresponding boundary integrals, we use Lemma \ref{lem:subdif}\ref{item:RobertDenk1} and obtain 
\begin{align}\label{eq:nosferatu}
\mathrm{T}_{1} = -\int_{\partial\Omega}F'(\mathbb{T}\uu(t,\cdot))\cdot (-\mathbf{v}(t,\cdot)\otimes_{\mathbb{T}}\nu_{\partial\Omega})\dif\mathscr{H}^{1}\geq - \int_{\partial\Omega}F^{\infty}(-\vv(t,\cdot)\otimes\nu_{\partial\Omega})\dif\mathscr{H}^{1}.
\end{align}
We now employ the relaxed energy functionals from Section \ref{sec:relaxthebulk}. 
Recalling that we assumed $\uu(t,\cdot)=0$ on $\partial\Omega$ and integrating with respect to time, the smoothness of $\uu$ and $\vv$ allows us to rewrite the resulting inequality as 
\begin{align*}
\int_{0}^{s}\int_{\Omega}&(\partial_{t}\mathbf{u}(t,\cdot))(\mathbf{v}(t,\cdot)-\mathbf{u}(t,\cdot))\dif x\dif t + \int_{0}^{s}\mathscr{F}_{0}^{*}[\vv(t,\cdot);\Omega]\dif t \geq  \int_{0}^{s}\mathscr{F}_{0}^{*}[\uu(t,\cdot);\Omega]\dif t \\ 
& + \int_{0}^{s}\int_{\Omega}\mathbf{f}(t,\cdot)\cdot(\vv(t,\cdot)-\uu(t,\cdot))\dif x\dif t + \int_{0}^{s}\int_{\Omega}\tocean(\uu(t,\cdot))\cdot(\vv(t,\cdot)-\uu(t,\cdot))\dif x\dif t 
\end{align*}
for all $0<s<T$. In this formulation, no smoothness beyond continuity is required for $F$, and the energy integrals merely require spatial $\bd$-regularity. In order to further relax the formulation towards less regularity with respect to time, we compute in the smooth setting:
\begin{align}
\int_{0}^{s}\int_{\Omega}(\partial_{t}\mathbf{u}(t,\cdot))(\mathbf{v}(t,\cdot)-\mathbf{u}(t,\cdot)) \dif x \dif t & = \int_{0}^{s}\int_{\Omega}\partial_{t}(\mathbf{u}(t,\cdot)-\mathbf{v}(t,\cdot))(\mathbf{v}(t,\cdot)-\mathbf{u}(t,\cdot))\dif x\dif t \notag\\ & + \int_{0}^{s}\int_{\Omega}(\partial_{t}\mathbf{v}(t,\cdot))\cdot(\mathbf{v}(t,\cdot)-\mathbf{u}(t,\cdot))\dif x\dif t \notag\\ 
& = - \frac{1}{2}\int_{\Omega}\Big(|\mathbf{u}-\mathbf{v}|^{2}(s,\cdot) - |\mathbf{u}-\mathbf{v}|^{2}(0,\cdot)\Big) \dif x \label{eq:intbyparts} \\ & +\int_{0}^{s}\int_{\Omega}(\partial_{t}\mathbf{v}(t,\cdot))(\mathbf{v}(t,\cdot)-\mathbf{u}(t,\cdot))\dif x\dif t.\notag
\end{align}
In combination with the preceding inequality, \eqref{eq:intbyparts} gives rise to the following energy-driven notion of solution of the momentum balance equation. 
\begin{definition}[Variational solutions for Hibler]\label{def:varsol1}
Let $\Omega\subset\R^{2}$ be open and bounded with Lipschitz boundary $\partial\Omega$ and let $T>0$. Moreover, let $\mathbf{u}_{0}\in\mathrm{BD}(\Omega)$ be an initial datum with compact support in $\Omega$, and let the data satisfy  \emph{\ref{item:Fprop1}--\ref{item:Fprop3}} as well as \eqref{eq:atmosphericmain}--\eqref{eq:initialvalue}. We say that $\mathbf{u}\in\lebe_{\mathrm{w}^{*}}^{1}(0,T;\bd(\Omega))\cap\lebe^{\infty}(0,T;\lebe^{2}(\Omega;\R^{2}))$ is a \emph{variational solution of Hibler's momentum balance equation} if the following hold:
\begin{enumerate}
\item\label{item:Hiblermain1} For $\mathscr{L}^{1}$-a.e. $0<s<T$, the \emph{evolutionary variational inequality}
\begin{align}\label{eq:varsolmain}
\begin{split}
\int_{0}^{s}\int_{\Omega}(\partial_{t}\mathbf{v})(\mathbf{v}-\mathbf{u})\dif t\dif x & + \int_{0}^{s}\mathscr{F}_{0}^{*}[\mathbf{v}(t,\cdot);\Omega]\dif t  \geq \int_{0}^{s}\mathscr{F}_{0}^{*}[\mathbf{u}(t,\cdot);\Omega]\dif t \\ 
& + \int_{0}^{s}\int_{\Omega}\mathbf{f}\cdot(\vv-\uu)\dif x\dif t + \int_{0}^{s}\int_{\Omega}\bm{\tau}_{\mathrm{ocean}}(\uu)(\vv-\uu)\dif x \dif t \\ 
& + \frac{1}{2}\int_{\Omega}\Big(|\mathbf{u}(s,\cdot)-\mathbf{v}(s,\cdot)|^{2} - |\mathbf{u}_{0}-\mathbf{v}(0,\cdot)|^{2}\Big) \dif x 
\end{split}
\end{align}
holds for every $\vv\in\lebe_{\mathrm{w}^{*}}^{1}(0,T;\bd(\Omega))\cap\sobo^{1,2}(0,T;\lebe^{2}(\Omega;\R^{2}))$. Here, $\mathscr{F}_{0}^{*}[-;\Omega]$ is as in \eqref{eq:relaxedbdry}.  
\item\label{item:Hiblermain2} The \emph{initial values} are attained in the $\lebe^{2}$-Lebesgue sense, meaning that 
\begin{align*}
\lim_{t\searrow 0}\dashint_{0}^{t}\|\uu(s,\cdot)-\uu_{0}\|_{\lebe^{2}(\Omega)}\dif s \coloneqq \lim_{t\searrow 0}\frac{1}{t}\int_{0}^{t}\|\uu(s,\cdot)-\uu_{0}\|_{\lebe^{2}(\Omega)}\dif s = 0. 
\end{align*}
\end{enumerate} 
\end{definition}
We are now ready to state our main result:
\begin{theorem}[Main result]\label{thm:main}
Subject to the assumptions of Definition \ref{def:varsol1}, \emph{there exists a variational solution $\uu\in\lebe_{\mathrm{w}^{*}}^{1}(0,T;\bd(\Omega))\cap\sobo^{1,2}(0,T;\lebe^{2}(\Omega;\R^{2}))$} of Hibler's momentum balance equation. 
\end{theorem}
Theorem \ref{thm:main} shall be established in  Section \ref{sec:existence}, dealing with the energy-driven evolution (see Definition \ref{def:varsol1}\ref{item:Hiblermain1}), and in Section \ref{sec:initial}, dealing with the initial values (see Definition \ref{def:varsol1}\ref{item:Hiblermain2}). Before embarking on the proof, two comments on Definition \ref{def:varsol1} are in order. Firstly, note that membership of variational solutions in $\lebe^{\infty}(0,T;\lebe^{2}(\Omega;\R^{2}))$ implies that the ultimate term on the right-hand side of \eqref{eq:varsolmain} is indeed well-defined for $\mathscr{L}^{1}$-a.e. $0<s<T$. Secondly, besides capturing the low spatial $\bd$-regularity, the relaxed functional also incorporates the Dirichlet zero boundary conditions: 
\begin{remark}\label{rem:boundarypenal} 
As is typical for linear growth problems, the Dirichlet zero boundary condition is included in the boundary penalization terms in the relaxed functional. Yet, the variational solution does not necessarily attain zero boundary values. This is phenomenological (see, e.g., the discussion by the second author in \cite[Section 1]{Gm1solo} or Beck \& Schmidt \cite[Section 3]{Beckschmidt13} in the stationary vectorial $\bv$-context). More precisely, the non-attainment is due to the fact that the trace operator on $\bd(\Omega)$ is \emph{not} continuous with respect to weak*-convergence, so the sort of convergence that allows for compactness assertions; see the discussion after Lemma \ref{lem:traceoperator}. 
\end{remark}

\subsection{Existence of a variational solution}\label{sec:existence}
In order to establish the existence of a variational solution of the momentum balance equation in the sense of Definition \ref{def:varsol1}, we employ three regularizations or stabilizations, respectively. This refers to the parameters $\delta,\varepsilon$ and $\zeta$ which we describe next. For $\delta,\varepsilon>0$, we define 
\begin{align}\label{eq:Fepsdelta}
F_{\varepsilon}(z)\coloneqq \sqrt{\varepsilon+|F(z)|^{2}}\;\;\;\text{and}\;\;\; F_{\delta,\varepsilon}(z)\coloneqq F_{\varepsilon}(z) + \frac{\delta}{2}|z|^{2},\qquad z\in\mathbb{R}_{\mathrm{sym}}^{2\times 2},  
\end{align}
where $F$ satisfies \ref{item:Fprop1}--\ref{item:Fprop3}.  
\begin{lemma}[$\varepsilon$-regularizations]\label{lem:approximations}
Let $F\colon\RR\to\R$ satisfy \emph{\ref{item:Fprop1}--\ref{item:Fprop3}}. Then, for any $\varepsilon>0$, $F_{\varepsilon}\colon\RR\to\R$ is convex, of linear growth (with potentially different parameters than in \eqref{eq:lingrowth}) and is differentiable with 
\begin{align*}
M\coloneqq \sup_{0<\varepsilon<1}\|F'_{\varepsilon}\|_{\sup}<\infty
\end{align*}
and $F'_{\varepsilon}(0)=0$. Moreover, we have $F_{\varepsilon}^{\infty}=F^{\infty}$ on $\RR$. Lastly, there exist constants $c_{4},c_{6}>0$ and $c_{4}\in\R$ such that 
\begin{align}\label{eq:shadyacres}
c_{4}|z| - c_{5} \leq F'_{\varepsilon}(z)\cdot z \leq c_{6}(1+|z|)\qquad\text{for all}\;z\in\RR\;\text{and all}\;0<\varepsilon<1. 
\end{align}
\end{lemma} 
\begin{proof} 
Convexity and linear growth are obvious. By Lemma \ref{lem:subdif}\ref{item:RobertDenk}, \ref{item:Fprop1} implies that $F$ is Lipschitz with some $L\geq 0$. Hence, by \ref{item:Fprop2}, $F(0)=0$, and so we have for all $\xi\in\RR\setminus\{0\}$: 
\begin{align*}
\frac{1}{|\xi|}|F^{2}(\xi)-F^{2}(0)|\leq L(|F(\xi)|+|F(0)|)\stackrel{|\xi|\searrow 0}{\longrightarrow} 2L|F(0)|=0. 
\end{align*}
Moreover, by \ref{item:Fprop1}--\ref{item:Fprop3}, we have for $\xi\in\RR$ that 
\begin{align*}
|(F^{2})'(\xi)|\leq 2|F'(\xi)|\,|F(\xi)| \leq 2L|F(\xi)|\stackrel{|\xi|\searrow 0}{\longrightarrow} 0, 
\end{align*}
so that $F^{2}\in\hold^{1}(\RR)$. From here, we deduce by use of the chain rule that $F_{\varepsilon}$ has bounded derivatives together with $F'_{\varepsilon}(0)=0$.  Finally, we compute 
\begin{align*}
F_{\varepsilon}^{\infty}(z) = \lim_{t\searrow 0}t \sqrt{\varepsilon+F^{2}\Big(\frac{z}{t}\Big)} = \lim_{t\searrow 0}\sqrt{\varepsilon t^{2} + \Big(t\,F\Big(\frac{z}{t}\Big)\Big)^{2}}= |F^{\infty}(z)| = F^{\infty}(z), \qquad z\in\RR,
\end{align*}
where we note that \eqref{eq:lingrowth} implies that $F^{\infty}\geq 0$. We finally deal with \eqref{eq:shadyacres}. A  direct computation gives us $M\coloneqq \sup_{0<\varepsilon<1}\|F'_{\varepsilon}\|_{\sup}<\infty$. Moreover, since $F_{\varepsilon}$ is differentiable and convex, we have $F_{\varepsilon}(z) + F'_{\varepsilon}(z)\cdot(w-z) \leq F_{\varepsilon}(w)$ for all $w,z\in\RR$. Applying the preceding inequality to $w=0$, we find by virtue of \ref{item:Fprop2}:
\begin{align*}
c_{1}|z|-c_{2} \stackrel{\eqref{eq:lingrowth}}{\leq} F(z) \leq F_{\varepsilon}(z) \leq F_{\varepsilon}(0) + F'_{\varepsilon}(z)\cdot z \leq \sqrt{\varepsilon} + M|z|, \qquad z\in\RR,
\end{align*}
from where \eqref{eq:shadyacres} follows. The proof is complete. 
\end{proof}
Next, we smoothly approximate the initial values $\uu_{0}\in\bd_{c}(\Omega)$. To this end, we require Young's convolution inequality in a form that involves measures.
\begin{lemma}
Let $\rho\in\hold_{c}^{\infty}(\ball_{1}(0))$ and $\mathbf{w}\in\bd_{c}(\R^{2})$. Then we have 
\begin{align}\label{eq:YoungRough}
\|\rho*\mathbb{T}\mathbf{w}\|_{\lebe^{2}(\R^{2})}\leq \|\rho\|_{\lebe^{2}(\R^{2})}|\mathbb{T}\mathbf{w}|(\R^{2}). 
\end{align}
\end{lemma}
\begin{proof} 
By the classical convolution inequality of Young, $\|f*g\|_{\lebe^{r}(\R^{2})}\leq \|f\|_{\lebe^{p}(\R^{2})}\|g\|_{\lebe^{q}(\R^{2})}$ holds whenever $1+\frac{1}{r}=\frac{1}{p}+\frac{1}{q}$ and $f\in\lebe^{p}(\R^{2})$ as well as  $g\in\lebe^{q}(\R^{2};\R^{2})$. Now let $\varrho_{d}$ be the $d$-rescaled variant of a standard mollifier on $\R^{2}$. Then Lemma \ref{lem:jensen} tells us that $|(\varrho_{d}*\mathbb{T}\mathbf{w})\mathscr{L}^{2}|(\R^{2})\leq |\mathbb{T}\mathbf{w}|(\R^{2})$. Since $\varrho_{d}*\mathbf{w}\to \mathbf{w}$ strongly in $\lebe^{1}(\R^{2};\R^{2})$, Lemma \ref{lem:auxBD}\ref{item:aux0A} gives us $|\mathbb{T}\mathbf{w}|(\R^{2})\leq \liminf_{d\searrow 0}|\mathbb{T}(\varrho_{d}*\mathbf{w})|(\R^{2})$. In conclusion, setting $\mathbf{w}_{d}\coloneqq \varrho_{d}*\mathbf{w}$, we have 
\begin{align}\label{eq:YoungApprox1}
|\mathbb{T}\mathbf{w}_{d}|(\R^{2})\to|\mathbb{T}\mathbf{w}|(\R^{2})\qquad\text{as $d\searrow 0$}.   
\end{align}
Lastly, $\rho*\mathbf{w}\in\hold_{c}^{\infty}(\R^{2})$ and so $\varrho_{d}*(\rho*\mathbf{w})\to \rho*\mathbf{w}$ strongly in $\lebe^{2}(\R^{2};\R^{2})$ as $d\searrow 0$. We now apply Young's inequality with $p=r=2$ and $q=1$. 
Using the commutativity of convolution, we thus find 
\begin{align*}
\|\rho*\mathbb{T}\mathbf{w}\|_{\lebe^{2}(\R^{2})} & = \lim_{d\searrow 0}\|\varrho_{d}*(\rho*\mathbb{T}\ww)\|_{\lebe^{2}(\R^{2})}= \lim_{d\searrow 0}\|\rho*(\varrho_{d}*\mathbb{T}\ww)\|_{\lebe^{2}(\R^{2})} \\ 
& \!\!\!\!\! \stackrel{\text{Young}}{\leq} \limsup_{d\searrow 0} \|\rho\|_{\lebe^{2}(\R^{2})}\|\mathbb{T}\ww_{d}\|_{\lebe^{1}(\R^{2})} \stackrel{\eqref{eq:YoungApprox1}}{=} \|\rho\|_{\lebe^{2}(\R^{2})}|\mathbb{T}\ww|(\R^{2}). 
\end{align*}
This completes the proof. 
\end{proof}
Recalling \eqref{eq:initialvalue},  we firstly extend $\uu_{0}\in\bd_{c}(\Omega)$ by zero to $\widetilde{\uu}_{0}\in\bd_{c}(\R^{2})$, whereby 
\begin{align}\label{eq:extendboundedly}
\|\widetilde{\uu}_{0}\|_{\bd(\R^{2})}\leq \|\uu_{0}\|_{\bd(\Omega)}.
\end{align}
Next let $\eta\in\hold_{c}^{\infty}(\ball_{1}(0))$ be a non-negative standard mollifier and denote by $\eta_{\zeta}$ its $\zeta$-rescaled version. By \eqref{eq:YoungRough}, we conclude that 
\begin{align*}
\|\mathbb{T}(\eta_{\zeta}*\widetilde{\uu}_{0})\|_{\lebe^{2}(\Omega)} & \leq \|\mathbb{T}(\eta_{\zeta}*\widetilde{\uu}_{0})\|_{\lebe^{2}(\R^{2})} = \|\eta_{\zeta}*\mathbb{T}\widetilde{\uu}_{0}\|_{\lebe^{2}(\R^{n})} \\ 
& \leq \|\eta_{\zeta}\|_{\lebe^{2}(\R^{2})} |\mathbb{T}\widetilde{\uu}_{0}|(\R^{2}) \leq \frac{c}{\zeta}\|\uu_{0}\|_{\bd(\Omega)}, 
\end{align*}
where $c>0$ is independent of $\zeta>0$ and $\uu_{0}$. By the compact support of $\uu_{0}$ in $\Omega$, there exists $d_{0}>0$ such that $\spt(\rho_{\zeta}*\widetilde{\uu}_{0})\subset\Omega$ for all $0<\zeta<d_{0}$. Hence, setting $\uu_{0}^{\zeta}\coloneqq (\eta_{\zeta}*\widetilde{\uu}_{0})|_{\Omega}$ for $0<\zeta<d_{0}$ and recalling that $\bd(\Omega)\hookrightarrow\lebe^{2}(\Omega;\R^{2})$ in $n=2$ dimensions, we have 
\begin{align}\label{eq:initialvalueapproximation}
\begin{split}
&\uu_{0}^{\zeta}\in\hold_{c}^{\infty}(\Omega;\R^{2}),\\ 
& \|\uu_{0}^{\zeta}\|_{\lebe^{2}(\Omega)}\leq c \|\uu_{0}\|_{\bd(\Omega)},\\ 
&\|\mathbb{T}\uu_{0}^{\zeta}\|_{\lebe^{2}(\Omega)}\leq \frac{c}{\zeta}\|\uu_{0}\|_{\bd(\Omega)}, \\ 
&\|\mathbb{T}\uu_{0}^{\zeta}\|_{\lebe^{1}(\Omega)}\leq c\, \|\uu_{0}\|_{\bd(\Omega)}\qquad\text{for all}\;0<\zeta<d_{0}.  
\end{split}
\end{align}
Here, $c>0$ still is independent of $0<\zeta<d_{0}$ and $\uu_{0}$. In particular, let us note that $\eqref{eq:initialvalueapproximation}_{2}$ and the scaling in $\zeta$ in $\eqref{eq:initialvalueapproximation}_{3}$ is specific to the two-dimensional case. In what follows, we fix approximation parameters 
\begin{align}\label{eq:allchoose}
0<\zeta<d_{0},\;\;\;0<\delta<\zeta^{2}\;\;\;\text{and}\;\;\;0<\varepsilon<1.
\end{align}
We now consider a stabilized system by adding an artificial viscosity term as follows: 
\begin{align}\label{eq:motiv1}
\begin{cases}
\displaystyle\partial_{t}\mathbf{u}_{\delta,\varepsilon}^{\zeta} = \mathbb{T}^{*}(F'_{\delta,\varepsilon}(\mathbb{T}\uu_{\delta,\varepsilon}^{\zeta})) + \delta\Delta_{\mathbb{T}}\mathbf{u}_{\delta,\varepsilon}^{\zeta} + \tocean(\uu_{\delta,\varepsilon}^{\zeta}) + \mathbf{f}& \text{in}\;(0,T)\times\Omega,\\ 
\mathbf{u}_{\delta,\varepsilon}^{\zeta} = 0 & \text{on}\,(0,T)\times \partial\Omega,\\ 
\mathbf{u}_{\delta,\varepsilon}^{\zeta}(0,\cdot)=\mathbf{u}_{0}^{\zeta}&\text{in}\;\Omega, 
\end{cases}
\end{align}
where $F_{\delta,\varepsilon}$ is as in \eqref{eq:Fepsdelta} and $\Delta_{\mathbb{T}}\coloneqq \mathbb{T}^{*}\mathbb{T}$ with the formal adjoint $\mathbb{T}^{*}$ from \eqref{eq:formaladjoint}. Whereas other stabilizations (e.g., with $\Delta$) are possible, this viscosity stabilization is favourable since then $\eqref{eq:motiv}_{1}$ is entirely formulated in terms of $\mathbb{T}$. We now record that the nonlinear system \eqref{eq:motiv1} is well-posed in the following sense: 
\begin{lemma}\label{lem:weaksolapproximate} Subject to the assumptions of Definition \ref{def:varsol1}, let $\delta,\varepsilon,\zeta>0$ be adjusted according to \eqref{eq:allchoose}. Then there exists a \emph{unique weak solution} 
\begin{align}\label{eq:appsolreg}
\mathbf{u}_{\delta,\varepsilon}^{\zeta}\in \lebe^{2}(0,T;\sobo_{0}^{1,2}(\Omega;\R^{2}))\cap \sobo^{1,2}(0,T;\lebe^{2}(\Omega;\R^{2}))
\end{align}
in the following sense: For all $0<s\leq T$ and all $\bm{\psi}\in\lebe^{2}(0,s;\sobo_{0}^{1,2}(\Omega;\R^{2}))$, we have 
\begin{align}\label{eq:weakform}
\begin{split}
\int_{0}^{s}\int_{\Omega}\partial_{t}\mathbf{u}_{\delta,\varepsilon}^{\zeta}\cdot\bm{\psi}\dif x \dif t &+ \int_{0}^{s}\int_{\Omega}F'_{\delta,\varepsilon}(\mathbb{T}\uu_{\delta,\varepsilon}^{\zeta})\cdot\mathbb{T}\bm{\psi}\dif x \dif t \\ 
& = \int_{0}^{s}\int_{\Omega}\mathbf{f}\cdot\bm{\psi}\dif x \dif t + \int_{0}^{s}\int_{\Omega}\tocean(\uu_{\delta,\varepsilon}^{\zeta})\cdot\bm{\psi}\dif x\dif t.
\end{split}
\end{align}
Moreover, $\uu_{\delta,\varepsilon}^{\zeta}(0,\cdot)=\uu_{0}^{\zeta}$. 
\end{lemma}
\begin{proof}
We reformulate  $\eqref{eq:motiv1}_{1}$ in the language of Proposition \ref{prop:arendt}. To this end, we put $H\coloneqq \lebe^{2}(\Omega;\R^{2})$ and define $\Phi\colon\lebe^{2}(\Omega;\R^{2})\to (-\infty,\infty]$ by 
\begin{align}\label{eq:energyinproof}
\Phi[\mathbf{v}] \coloneqq \begin{cases}\displaystyle \int_{\Omega}F_{\delta,\varepsilon}(\mathbb{T}\vv)\dif x&\;\text{if}\;\vv\in\sobo_{0}^{1,2}(\Omega;\R^{2}),\\ 
+\infty&\;\text{if}\;\vv\in \lebe^{2}(\Omega;\R^{2})\setminus\sobo_{0}^{1,2}(\Omega;\R^{2}). 
\end{cases}
\end{align}
Then we have $\mathrm{dom}(\Phi)=\sobo_{0}^{1,2}(\Omega;\R^{2})$. Clearly, $\Phi$ satisfies \ref{item:hikaru1}. For \ref{item:hikaru2}, note that $\sobo_{0}^{1,2}(\Omega;\R^{2})\hookrightarrow\hookrightarrow\lebe^{2}(\Omega;\R^{2})$ by the Rellich-Kondrachov theorem. In particular, if $\lambda\geq 0$ and $(\vv_{j})\subset E_{\lambda}$, then Korn's and Poincar\'{e}'s inequalities in $\sobo_{0}^{1,2}(\Omega;\R^{2})$ firstly imply that 
\begin{align*}
\int_{\Omega}|\vv_{j}|^{2}\dif x \leq c \int_{\Omega}|\nabla\vv_{j}|^{2}\dif x & \leq c \int_{\Omega}|\sg\vv_{j}|^{2}\dif x \stackrel{\eqref{eq:pointwisecompa}}{\leq} c \int_{\Omega}|\mathbb{T}\vv_{j}|^{2}\dif x \\ & \leq c \int_{\Omega}F_{\delta,\varepsilon}(\mathbb{T}\vv_{j})\dif x = c\,\Phi[\mathbf{v}_{j}] \leq c\lambda,   
\end{align*}
where the ultimate constant satisfies $c=c(\mathrm{diam}(\Omega),\delta)>0$. Hence, $(\vv_{j})$ is bounded in $\sobo_{0}^{1,2}(\Omega;\R^{2})$. By the Banach-Alaoglu and Rellich-Kondrachov theorems, there exists $\vv\in\sobo_{0}^{1,2}(\Omega;\R^{2})$ and a subsequence $(\vv_{j(i)})\subset(\vv_{j})$ such that $\vv_{j(i)}\rightharpoonup\vv$ weakly in $\sobo_{0}^{1,2}(\Omega;\R^{2})$ and $\vv_{j(i)}\to\vv$ strongly in $\lebe^{2}(\Omega;\R^{2})$. Since $\Phi|_{\sobo_{0}^{1,2}(\Omega;\R^{2})}$ is sequentially weakly lower semicontinuous on $\sobo_{0}^{1,2}(\Omega;\R^{2})$ (see, e.g., \cite[Chapter 5]{Giusti}), we deduce that 
\begin{align*}
\int_{\Omega}F_{\delta,\varepsilon}(\mathbb{T}\vv)\dif x \leq \liminf_{i\to\infty}\int_{\Omega}F_{\delta,\varepsilon}(\mathbb{T}\vv_{j(i)})\dif x \leq \lambda, 
\end{align*}
and so $\vv\in E_{\lambda}$ too. In conclusion, $E_{\lambda}$ is compact with respect to the norm topology on $\lebe^{2}(\Omega;\R^{2})$, and this is \ref{item:hikaru2}. For \ref{item:hikaru3}, we put $g(\vv(t))\coloneqq \tocean(\vv(t))+\mathbf{f}$, where now $\vv\in\lebe^{2}(0,T;\lebe^{2}(\Omega;\R^{2}))$. By \eqref{eq:oceanicHoelder}, we have for $\mathscr{L}^{1}$-a.e. $0<t<T$:
\begin{align*}
\|g(\vv(t,\cdot))\|_{\lebe^{2}(\Omega)} & \leq c\Big(\int_{\Omega}|\vv(t,\cdot)|^{2}\dif x\Big)^{\frac{1}{2}} + \|\mathbf{f}(t,\cdot)\|_{\lebe^{2}(\Omega)} \eqqcolon L\|\vv(t,\cdot)\|_{\lebe^{2}(\Omega)} + b(t)
\end{align*}
with an obvious definition of $L$ and $b$. Since $\mathbf{f}\in\lebe^{2}(0,T;\lebe^{2}(\Omega;\R^{2}))$, $b\in\lebe^{2}((0,T))$, and so \eqref{eq:gbounds} follows. Lastly, since $\tocean$ is Lipschitz by \eqref{eq:oceanicHoelder}, $g$ is continuous on $\lebe^{2}(0,T;\lebe^{2}(\Omega;\R^{2}))$. Hence, \ref{item:hikaru3} is satisfied. In conclusion, Proposition \ref{prop:arendt} provides us with a solution $\uu_{\delta,\varepsilon}^{\zeta}\in\sobo^{1,2}(0,T;\lebe^{2}(\Omega;\R^{2}))$ such that $\uu_{\delta,\varepsilon}^{\zeta}(0,\cdot)=\uu_{0}^{\zeta}$ in $\lebe^{2}(\Omega;\R^{2})$ and 
\begin{align}\label{eq:natalieteeger}
\partial_{t}\uu_{\delta,\varepsilon}^{\zeta}(t,\cdot) + \mathcal{A}\uu_{\delta,\varepsilon}^{\zeta}(t,\cdot) = g(\uu_{\delta,\varepsilon}^{\zeta}(t,\cdot))\qquad\text{$\mathscr{L}^{1}$-a.e.  $t\in(0,T)$}. 
\end{align}
Moreover, Proposition \ref{prop:arendt} tells us that $\Phi[\uu_{\delta,\varepsilon}^{\zeta}]\in\lebe^{1}((0,T))$, whereby Korn's and Poincar\'{e}'s inequalities combine to $\uu_{\delta,\varepsilon}^{\zeta}\in\lebe^{2}(0,T;\sobo_{0}^{1,2}(\Omega;\R^{2}))$ by the very definition of $\Phi$. Now let $0<s< T$ and  $\bm{\psi}\in\lebe^{2}(0,s;\sobo_{0}^{1,2}(\Omega;\R^{2}))$ be arbitrary. For $\mathscr{L}^{1}$-a.e. $0<t<s$, we multiply \eqref{eq:natalieteeger} with $\bm{\psi}(t,\cdot)$ and integrate over $\Omega_{s}$. This gives us 
\begin{align}\label{eq:teenager}
\int_{\Omega_{s}}(\partial_{t}\uu_{\delta,\varepsilon}^{\zeta})\cdot\bm{\psi}\dif\,(t,x) + \int_{\Omega_{s}}\mathcal{A}\uu_{\delta,\varepsilon}^{\zeta}\cdot\bm{\psi}\dif\,(t,x) = \int_{\Omega_{s}}g(\uu_{\delta,\varepsilon}^{\zeta})\bm{\psi}\dif\,(t,x). 
\end{align}
Since $F_{\delta,\varepsilon}$ is differentiable by Lemma \ref{lem:approximations}, we have 
\begin{align*}
\int_{\Omega}\mathcal{A}\uu(t,\cdot)\cdot\bm{\psi}\dif x = \int_{\Omega}F'_{\delta,\varepsilon}(\mathbb{T}\uu(t,\cdot))\cdot\mathbb{T}\bm{\psi}(t,\cdot)\dif x \qquad\text{for $\mathscr{L}^{1}$-a.e. $0<t<s$}, 
\end{align*}
and so \eqref{eq:teenager} yields \eqref{eq:weakform}.

As to uniqueness, let $\uu^{(1)},\uu^{(2)}$ be two weak solutions with the regularity displayed in \eqref{eq:appsolreg}; in particular, they satisfy \eqref{eq:weakform} with the respective modifications and  $\uu^{(1)}=\uu^{(2)}=\uu_{0}^{\zeta}$. We test both equations with $\bm{\psi}\coloneqq \uu^{(1)}-\uu^{(2)}$ and subtract them from each other. Due to the quadratic viscosity stabilization, there exists $\mu=\mu(\delta)>0$ such that 
\begin{align*}
\mu |z-z|^{2} \leq (F'_{\delta,\varepsilon}(z)-F'_{\delta,\varepsilon}(z'))\cdot(z-z')\qquad\text{for all}\;z,z'\in\RR.
\end{align*}
For an arbitrary $0<s<T$, we thus end up with 
\begin{align*}
 \frac{1}{2}&\int_{0}^{s}\partial_{t}\|\uu^{(1)}(t,\cdot)-\uu^{(2)}(t,\cdot)\|_{\lebe^{2}(\Omega)}^{2}\dif t + \mu  \int_{0}^{s}\|\mathbb{T}\uu^{(1)}(t,\cdot)-\mathbb{T}\uu^{(2)}(t,\cdot)\|_{\lebe^{2}(\Omega)}^{2}\dif t \\ 
 & \leq \int_{0}^{s}\int_{\Omega}|\tocean(\uu^{(1)}(t,\cdot))-\tocean(\uu^{(2)}(t,\cdot))|\uu^{(1)}(t,\cdot)-\uu^{(2)}(t,\cdot)|\dif x \dif t \\ 
 & \leq c \int_{0}^{s}\|\uu^{(1)}(t,\cdot)-\uu^{(2)}(t,\cdot)\|_{\lebe^{2}(\Omega)}^{2}\dif t, 
\end{align*}
where we employed the global Lipschitz estimate on $\widetilde{\eta}$ from \eqref{eq:oceanicHoelder} in the last step. Using that $\uu^{(1)},\uu^{(2)}\in\sobo^{1,2}(0,T;\lebe^{2}(\Omega;\R^{2}))$ and $\uu^{(1)}(0,\cdot)=\uu^{(2)}(0,\cdot)=\uu_{0}^{\zeta}$, we find that 
\begin{align*}
\|\uu^{(1)}(s,\cdot)-\uu^{(2)}(s,\cdot)\|_{\lebe^{2}(\Omega)}^{2} \leq 2\ell \int_{0}^{s}\|\uu^{(1)}(t,\cdot)-\uu^{(2)}(t,\cdot)\|_{\lebe^{2}(\Omega)}^{2}\dif t . 
\end{align*}
Gronwall's inequality then implies that $\uu^{(1)}(s,\cdot)=\uu^{(2)}(s,\cdot)$ for all $0\leq s\leq T$. The proof is complete. 
\end{proof}
\begin{remark}\label{rem:AHlimitations}
Lemma \ref{lem:weaksolapproximate} can be established in various ways. For instance, if one considers the simplified situation where $\tocean(\uu_{\delta,\varepsilon}^{\zeta})$ is absent, the existence of strong (and hence weak) solutions can be deduced from Showalter {\cite[Chapter III, Prop. 4.2]{Showalter}}. However, the key point in the argument below is the limit passage $\varepsilon\searrow 0$. In this regard, we emphasize that Proposition \ref{prop:arendt} does not imply the existence of weak solutions of \eqref{eq:weakform} when $\delta=0$. Indeed, the substitute of \eqref{eq:energyinproof} would be 
\begin{align*}
\widetilde{\Phi}[\vv] \coloneqq \begin{cases} 
\displaystyle \int_{\Omega}F_{\varepsilon}(\mathbb{T}\vv)\dif x&\;\text{if}\;\vv\in\ld_{0}(\Omega), \\ 
+\infty&\;\text{if}\;\vv\in\lebe^{2}(\Omega;\R^{2})\setminus\ld_{0}(\Omega),
\end{cases}
\end{align*}
but then condition \ref{item:hikaru2} is not necessarily fulfilled anymore. This is due to the fact that $\ld_{0}(\Omega)\hookrightarrow\lebe^{2}(\Omega;\R^{2})$, see Lemma \ref{lem:poincaresobolev},  but the embedding is not compact. 
\end{remark}
\begin{remark}
For the existence part of Theorem \ref{thm:main}, the uniqueness part of Lemma \ref{lem:weaksolapproximate} is not necessary. However, in view of a potential numerical implementation of our strategy, the unique solvability for the approximate problems is essential. Lastly note that the unique solvability does \emph{not} come as a consequence of Proposition \ref{prop:arendt}. Indeed, if the sublinear function e.g. is only H\"{o}lder but not Lipschitz continuous, uniqueness might fail (see \cite{ArendtHauer}). In our situation, uniqueness is a consequence of the specific hypothesis on $\tocean$, see \eqref{eq:oceanicHoelder}.
\end{remark}
Based on the weak solution $\uu_{\delta,\varepsilon}^{\zeta}$ from Lemma \ref{lem:weaksolapproximate}, we now derive the evolutionary variational inequality satisfied by $\uu_{\delta,\varepsilon}^{\zeta}$. To this end, let $0<s<T$ and let the test map $\mathbf{v}\in\lebe^{2}(0,T;\sobo_{0}^{1,2}(\Omega;\R^{2}))\cap\sobo^{1,2}(0,T;\lebe^{2}(\Omega;\R^{2}))$ be arbitrary. We test \eqref{eq:weakform} with $\bm{\psi} \coloneqq \mathbf{v}-\mathbf{u}_{\delta,\varepsilon}^{\zeta}\in\lebe^{2}(0,T;\sobo_{0}^{1,2}(\Omega;\R^{2}))$ to obtain 
\begin{align}
\int_{0}^{s}\int_{\Omega}\partial_{t}\mathbf{u}_{\delta,\varepsilon}^{\zeta}(\mathbf{v}-\mathbf{u}_{\delta,\varepsilon}^{\zeta})\dif x\dif t & = - \int_{0}^{s}\int_{\Omega}F'_{\delta,\varepsilon}(\mathbb{T}\uu_{\delta,\varepsilon}^{\zeta})\cdot\mathbb{T}(\mathbf{v}-\mathbf{u}_{\delta,\varepsilon}^{\zeta})\dif x\dif t \notag \\ & \!\!\!\!\!\!\!\!\!\!\!\!\!\!\!\!\!\!\!\!\!\!\!\!\!\!\!\!\!\!\!\!\!\!\!\!\!\!\!\! + \int_{0}^{s}\int_{\Omega}\mathbf{f}\cdot(\mathbf{v}-\uu_{\delta,\varepsilon}^{\zeta})\dif x\dif t + \int_{0}^{s}\int_{\Omega}\tocean(\uu_{\delta,\varepsilon}^{\zeta})\cdot(\vv-\uu_{\delta,\varepsilon}^{\zeta})\dif x \dif t \notag \\ 
& \!\!\!\!\!\!\!\!\!\!\!\!\!\!\!\!\!\!\!\!\!\!\!\!\!\!\!\!\!\!\!\!\!\!\!\!\!\!\!\! \!\!\! \! \stackrel{\eqref{eq:subgradientdef}}{\geq} \int_{0}^{s}\int_{\Omega}F_{\delta,\varepsilon}(\mathbb{T}\uu_{\delta,\varepsilon}^{\zeta})\dif x\dif t - \int_{0}^{s}\int_{\Omega}F_{\delta,\varepsilon}(\mathbb{T}\vv)\dif x \dif t
 \label{eq:wtanner} \\ & \!\!\!\!\!\!\!\!\!\!\!\!\!\!\!\!\!\!\!\!\!\!\!\!\!\!\!\!\!\!\!\!\!\!\!\!\!\!\!\! + \int_{0}^{s}\int_{\Omega}\mathbf{f}\cdot(\mathbf{v}-\uu_{\delta,\varepsilon}^{\zeta})\dif x\dif t + \int_{0}^{s}\int_{\Omega}\tocean(\uu_{\delta,\varepsilon}^{\zeta})\cdot(\vv-\uu_{\delta,\varepsilon}^{\zeta})\dif x \dif t. \notag 
\end{align}
By the regularity of $\uu_{\delta,\varepsilon}^{\zeta}$ and $\vv$, we may rewrite the term on the very left-hand side of the previous chain of inequalities in the same way as in \eqref{eq:intbyparts}. This gives us the evolutionary variational inequality 
\begin{align}
\int_{0}^{s}\int_{\Omega}\partial_{t}\mathbf{v}\cdot(\mathbf{v}-\mathbf{u}_{\delta,\varepsilon}^{\zeta})\dif x\dif t & + \int_{0}^{s}\int_{\Omega}F_{\delta,\varepsilon}(\mathbb{T}\mathbf{v})\dif x \dif t   - \int_{0}^{s}\int_{\Omega}\mathbf{f}\cdot\mathbf{v}\dif x \dif t \notag\\ & \geq \int_{0}^{s}\int_{\Omega}F_{\delta,\varepsilon}(\mathbb{T}\mathbf{u}_{\delta,\varepsilon}^{\zeta})\dif x\dif t -  \int_{0}^{s}\int_{\Omega}\mathbf{f}\cdot\mathbf{u}_{\delta,\varepsilon}^{\zeta}\dif x \dif t  \label{eq:approxineq}\\
& + \int_{0}^{s}\int_{\Omega}\tocean(\uu_{\delta,\varepsilon}^{\zeta})\cdot(\vv-\uu_{\delta,\varepsilon}^{\zeta})\dif x\dif t  \notag\\ 
&  + \frac{1}{2}\int_{\Omega}\Big(|\mathbf{u}_{\delta,\varepsilon}^{\zeta}-\mathbf{v}|^{2}(s,\cdot) - |\mathbf{u}_{\delta,\varepsilon}^{\zeta}-\mathbf{v}|^{2}(0,\cdot)\Big) \dif x\notag
\end{align}
for all $\mathscr{L}^{1}$-a.e. $0<s<T$. Next, we collect some useful a priori estimates: 
\begin{lemma}[A priori estimates]\label{lem:usefulapriori} Subject to the assumptions of Theorem \ref{thm:main}, there exists a constant $c>0$, independent of the parameters $\delta,\varepsilon,\zeta>0$ from \eqref{eq:allchoose}, such that the following hold: 
\begin{enumerate}
\item\label{eq:unifobd1} $\|\mathbf{u}_{\delta,\varepsilon}^{\zeta}\|_{\lebe^{\infty}(0,T;\lebe^{2}(\Omega))}\leq c$, 
\item\label{eq:unifobd2} $\|\mathbf{u}_{\delta,\varepsilon}^{\zeta}\|_{\lebe^{1}(0,T;\mathrm{LD}(\Omega))}\leq c$, 
\item\label{eq:unifobd3} $\sqrt{\delta}\|\mathbf{u}_{\delta,\varepsilon}^{\zeta}\|_{\lebe^{2}(0,T;\sobo^{1,2}(\Omega))}\leq c$, 
\item\label{eq:unifobd4} $\|\partial_{t}\mathbf{u}_{\delta,\varepsilon}^{\zeta}\|_{\lebe^{2}(0,T;\sobo^{-1,2}(\Omega))}\leq c$. 
\end{enumerate}
\end{lemma}
\begin{proof}
We begin with a preliminary estimate. Based on \eqref{eq:oceanicassump} and the $(1-\gamma)$-H\"{o}lder property of $\widetilde{\eta}$ from \eqref{eq:oceanicHoelder}, we find by Young's inequality with some $\theta>0$ to be specified later: 
\begin{align}\label{eq:VRock}
\begin{split}
|\tocean(\uu_{\delta,\varepsilon}^{\zeta})||\uu_{\delta,\varepsilon}^{\zeta}| & \leq c |\mathbf{U}_{\mathrm{ocean}}-\mathbf{u}_{\delta,\varepsilon}^{\zeta}|^{1-\gamma}|\uu_{\delta,\varepsilon}^{\zeta}| \\ 
& \leq c\,|\mathbf{U}_{\mathrm{ocean}}|\,|\uu_{\delta,\varepsilon}^{\zeta}| + c\,|\uu_{\delta,\varepsilon}^{\zeta}|^{2-\gamma} \\ 
& \leq c(\theta)(1+|\mathbf{U}_{\mathrm{ocean}}|^{2}) + \theta|\uu_{\delta,\varepsilon}^{\zeta}|^{2}, 
\end{split}
\end{align}
where we used that $0<\gamma<1$. The weak solution $\mathbf{u}_{\delta,\varepsilon}^{\zeta}$ of \eqref{eq:motiv1} is an admissible test function in \eqref{eq:weakform}. Based on its regularity as displayed in \eqref{eq:appsolreg}, we obtain for $\mathscr{L}^{1}$-a.e.  $0\leq s \leq T$:  
\begin{align*}
&\frac{1}{2}\Big(\|\mathbf{u}_{\delta,\varepsilon}^{\zeta}(s,\cdot)\|_{\lebe^{2}(\Omega)}^{2}- \|\mathbf{u}_{0}^{\zeta}\|_{\lebe^{2}(\Omega)}^{2} \Big) +  \int_{0}^{s}\int_{\Omega}F'_{\varepsilon}(\mathbb{T}\uu_{\delta,\varepsilon}^{\zeta})\cdot\mathbb{T}\uu_{\delta,\zeta}\dif x \dif t + \frac{\delta}{2}\int_{0}^{s}\int_{\Omega}|\mathbb{T}\mathbf{u}_{\delta,\varepsilon}^{\zeta}|^{2}\dif x \dif t \\ & 
=\frac{1}{2}\int_{0}^{s}\frac{\dif}{\dif t}\int_{\Omega}|\mathbf{u}_{\delta,\varepsilon}^{\zeta}(t)|^{2}\dif x \dif t  +  \int_{0}^{s}\int_{\Omega}F'_{\varepsilon}(\mathbb{T}\uu_{\delta,\varepsilon}^{\zeta})\cdot\mathbb{T}\uu_{\delta,\varepsilon}^{\zeta}\dif x \dif t + \frac{\delta}{2}\int_{0}^{s}\int_{\Omega}|\mathbb{T}\mathbf{u}_{\delta,\varepsilon}^{\zeta}|^{2}\dif x \dif t \\ 
& \!\!\!\! \stackrel{\eqref{eq:weakform}}{=} \int_{0}^{s}\int_{\Omega}\mathbf{f}\cdot\mathbf{u}_{\delta,\varepsilon}^{\zeta}\dif x \dif t + \int_{0}^{s}\int_{\Omega}\tocean(\uu_{\delta,\varepsilon}^{\zeta})\cdot\uu_{\delta,\varepsilon}^{\zeta}\dif x\dif t.
\end{align*}
We then pass to the essential supremum over $0\leq s\leq T$ on both sides of the preceding equation. By virtue of Young's inequality with some $0<\theta<1$, this yields  
\begin{align*}
\frac{1}{2}\sup_{0<s<T}\Big(\|\mathbf{u}_{\delta,\varepsilon}^{\zeta}(s,\cdot)\|_{\lebe^{2}(\Omega)}^{2} & - \|\mathbf{u}_{0}^{\zeta}\|_{\lebe^{2}(\Omega)}^{2} \Big) \\ & +  \int_{0}^{T}\int_{\Omega}F'_{\varepsilon}(\mathbb{T}\uu_{\delta,\varepsilon}^{\zeta})\cdot\mathbb{T}\uu_{\delta,\varepsilon}^{\zeta}\dif x \dif t + \frac{\delta}{2}\int_{0}^{T}\int_{\Omega}|\mathbb{T}\mathbf{u}_{\delta,\varepsilon}^{\zeta}|^{2}\dif x \dif t \\ 
& = \int_{0}^{T}\int_{\Omega}|\mathbf{f}\cdot\mathbf{u}_{\delta,\varepsilon}^{\zeta}|\dif x \dif t + \int_{0}^{T}\int_{\Omega}|\tocean(\uu_{\delta,\varepsilon}^{\zeta})|\,|\uu_{\delta,\varepsilon}^{\zeta}|\dif x \dif t \\ & \!\!\!\! \stackrel{\text{\eqref{eq:VRock}}}{\leq}  c(\theta)\int_{0}^{T}\int_{\Omega}|\mathbf{f}|^{2}\dif x \dif t + c(\theta)\int_{0}^{T}\int_{\Omega}(1+|\mathbf{U}_{\mathrm{ocean}}|^{2})\dif x\dif t \\ 
& + \theta T\sup_{0<s<T}\|\mathbf{u}_{\delta,\varepsilon}^{\zeta}(s,\cdot)\|_{\lebe^{2}(\Omega)}^{2}.
\end{align*}
Choosing $\theta>0$ sufficiently small, the ultimate term may be absorbed into the very first term on the left-hand side of the preceding inequality. By our integrability assumptions \eqref{eq:atmosphericmain}, \eqref{eq:oceanicmain} on $\mathbf{f}$ and $\mathbf{U}_{\mathrm{ocean}}$ in conjunction with \eqref{eq:initialvalue} and $\eqref{eq:initialvalueapproximation}_{2}$, this yields \ref{eq:unifobd1}. Inequality \eqref{eq:shadyacres} yields assertion \ref{eq:unifobd2} and, by use of Korn's inequality and Poincar\'{e}'s inequality in $\sobo_{0}^{1,2}(\Omega;\R^{2})$, assertion \ref{eq:unifobd3}. 

To see \ref{eq:unifobd4}, we recall that $\lebe^{2}(0,T;\sobo^{-1,2}(\Omega;\R^{2}))\cong\lebe^{2}(0,T;\sobo_{0}^{1,2}(\Omega;\R^{2}))'$. Going back to \eqref{eq:weakform}, which holds true for $s=T$, let $\bm{\psi}\in\lebe^{2}(0,T;\sobo_{0}^{1,2}(\Omega;\R^{2}))$. With $M>0$ as in Lemma \ref{lem:approximations}, we have   
\begin{align*}
\left\vert \int_{0}^{T}\int_{\Omega}\partial_{t}\uu_{\delta,\varepsilon}^{\zeta}\cdot\bm{\psi}\dif x \dif t \right\vert & \leq C(T,M,\mathscr{L}^{2}(\Omega))\|\nabla\bm{\psi}\|_{\lebe^{2}(0,T;\lebe^{2}(\Omega))} \\ & + \delta\|\nabla \uu_{\delta,\varepsilon}^{\zeta}\|_{\lebe^{2}(0,T;\lebe^{2}(\Omega))}\|\nabla\bm{\psi}\|_{\lebe^{2}(0,T;\lebe^{2}(\Omega))} \\ 
& + \|\mathbf{f}\|_{\lebe^{2}(0,T;\lebe^{2}(\Omega))}\|\bm{\psi}\|_{\lebe^{2}(0,T;\lebe^{2}(\Omega))} \\ & + \|\tocean(\uu_{\delta,\varepsilon}^{\zeta})\|_{\lebe^{2}(0,T;\lebe^{2}(\Omega))}\|\bm{\psi}\|_{\lebe^{2}(0,T;\lebe^{2}(\Omega))} \leq  c\|\bm{\psi}\|_{\lebe^{2}(0,T;\sobo^{1,2}(\Omega))}, 
\end{align*}
where $0<c<\infty$ holds due to \ref{eq:unifobd1}--\ref{eq:unifobd3}. Specifically, note that by \eqref{eq:oceanicmain} and \eqref{eq:oceanicassump}, the oceanic force terms are equally bounded due to \ref{eq:unifobd1}. 
Hence, \ref{eq:unifobd4} follows, and the proof is complete. 
\end{proof}
In the following, we shall fix $0<\zeta<d_{0}$ and then send, in this order, $\delta\searrow 0$ and $\varepsilon\searrow 0$. To this end, we now examine the limiting behaviour of $(\uu_{\delta,\varepsilon}^{\zeta})$.
\begin{corollary}\label{cor:compa1}
Let $0<\zeta<d_{0}$. For each $\varepsilon>0$, there exists a map $\mathbf{u}_{\varepsilon}^{\zeta}\in\lebe_{\mathrm{w}^{*}}^{1}(0,T;\bd(\Omega))\cap\lebe^{\infty}(0,T;\lebe^{2}(\Omega;\R^{2}))$ and a sequence $(\delta_{j})\subset (0,\zeta^{2})$ with $\delta_{j}\searrow 0$ as $j\to\infty$ such that the following hold:
\begin{enumerate}
\item\label{eq:compa1} $\mathbf{u}_{\delta_{j},\varepsilon}^{\zeta}\stackrel{*}{\rightharpoonup}\mathbf{u}_{\varepsilon}^{\zeta}$ in $\lebe^{\infty}(0,T;\lebe^{2}(\Omega;\R^{2}))$, 
\item\label{eq:compa1a} $\mathbf{u}_{\delta_{j},\varepsilon}^{\zeta}\rightharpoonup \mathbf{u}_{\varepsilon}^{\zeta}$ in $\lebe^{2}(0,T;\lebe^{2}(\Omega;\R^{2}))$, 
\item\label{eq:compa3} $\mathbf{u}_{\delta_{j},\varepsilon}^{\zeta}\to \mathbf{u}_{\varepsilon}^{\zeta}$ in $\lebe^{1}(0,T;\lebe^{r}(\Omega;\R^{2}))$ for any $1\leq r<2$, 
\item\label{eq:compa4} $\partial_{t}\uu_{\delta_{j},\varepsilon}^{\zeta}\stackrel{*}{\rightharpoonup}\partial_{t}\uu_{\varepsilon}^{\zeta}$ in $\lebe^{2}(0,T;\sobo^{-1,2}(\Omega;\R^{2}))$. 
\end{enumerate}
Moreover, there exists a constant $c>0$ independent of $(\delta_{j})$, $\varepsilon$ and $\zeta$ such that 
\begin{align}\label{eq:epsbounds}
\sup_{0<\varepsilon<1}\Big(\|\uu_{\varepsilon}^{\zeta}\|_{\lebe_{\mathrm{w}^{*}}^{1}(0,T;\bd(\Omega))} + \|\uu_{\varepsilon}^{\zeta}\|_{\lebe^{\infty}(0,T;\lebe^{2}(\Omega))}+\|\partial_{t}\uu_{\varepsilon}^{\zeta}\|_{\lebe^{2}(0,T;\sobo^{-1,2}(\Omega))}\Big) \leq c. 
\end{align}
As a consequence, for each $0<\zeta<d_{0}$, there exists a map $\uu^{\zeta}\in \lebe_{\mathrm{w}^{*}}^{1}(0,T;\bd(\Omega))\cap\lebe^{\infty}(0,T;\lebe^{2}(\Omega;\R^{2}))$ and a sequence $(\varepsilon_{j})$ with $\varepsilon_{j}\searrow 0$ as $j\to\infty$ such that the following hold: 
\begin{enumerate}\addtocounter{enumi}{4}
\item\label{eq:compa5}$\uu_{\varepsilon_{j}}^{\zeta}\stackrel{*}{\rightharpoonup}\uu^{\zeta}$ in $\lebe^{\infty}(0,T;\lebe^{2}(\Omega;\R^{2}))$, 
\item\label{eq:compa6} $\uu_{\varepsilon_{j}}^{\zeta} \rightharpoonup \uu^{\zeta}$ in $\lebe^{2}(0,T;\lebe^{2}(\Omega;\R^{2}))$, 
\item \label{eq:compa7}$\uu_{\varepsilon_{j}}^{\zeta}\to\uu^{\zeta}$ in $\lebe^{1}(0,T;\lebe^{r}(\Omega;\R^{2}))$ for any $1\leq r<2$, 
\item\label{eq:compa8} $\partial_{t}\uu_{\varepsilon_{j}}^{\zeta}\stackrel{*}{\rightharpoonup}  \partial_{t}\uu^{\zeta}$ in $\lebe^{2}(0,T;\sobo^{-1,2}(\Omega;\R^{2}))$. 
\end{enumerate}
Moreover, there exists a constant $c>0$ independent of $(\varepsilon_{j})$ and $\zeta$ such that 
\begin{align}\label{eq:zetabounds}
\sup_{0<\varepsilon<1}\Big(\|\uu^{\zeta}\|_{\lebe_{\mathrm{w}^{*}}^{1}(0,T;\bd(\Omega))} + \|\uu^{\zeta}\|_{\lebe^{\infty}(0,T;\lebe^{2}(\Omega))}+\|\partial_{t}\uu^{\zeta}\|_{\lebe^{2}(0,T;\sobo^{-1,2}(\Omega))}\Big) \leq c. 
\end{align}
Finally, there exists a map $\uu\in\lebe_{\mathrm{w}^{*}}^{1}(0,T;\bd(\Omega))\cap\lebe^{\infty}(0,T;\lebe^{2}(\Omega;\R^{2}))$ and a sequence $(\zeta_{j})\subset (0,d_{0})$ with $\zeta_{j}\searrow 0$ as $j\to\infty$ such that the following hold: 
\begin{enumerate}\addtocounter{enumi}{8}
\item\label{eq:compa9} $\uu^{\zeta_{j}}\stackrel{*}{\rightharpoonup}\uu$ in $\lebe^{\infty}(0,T;\lebe^{2}(\Omega;\R^{2}))$, 
\item\label{eq:compa10} $\uu^{\zeta_{j}} \rightharpoonup \uu$ in $\lebe^{2}(0,T;\lebe^{2}(\Omega;\R^{2}))$, 
\item\label{eq:compa11} $\uu^{\zeta_{j}}\to\uu$ in $\lebe^{1}(0,T;\lebe^{r}(\Omega;\R^{2}))$ for any $1\leq r<2$, 
\item\label{eq:compa12} $\partial_{t}\uu^{\zeta_{j}}\stackrel{*}{\rightharpoonup} \partial_{t}\uu$ in $\lebe^{2}(0,T;\sobo^{-1,2}(\Omega;\R^{2}))$.
\end{enumerate}
\end{corollary}
\begin{proof}
By Lemma \ref{lem:usefulapriori} and the Banach-Alaoglu theorem, there exists $\uu_{\varepsilon}^{\zeta}\in\lebe^{\infty}(0,T;\lebe^{2}(\Omega;\R^{2}))$ and a suitable sequence $(\delta_{j})$ such that \ref{eq:compa1}--\ref{eq:compa1a} hold. For \ref{eq:compa3}, we recall that $X_{0}\coloneqq \mathrm{LD}(\Omega)\hookrightarrow\hookrightarrow\lebe^{r}(\Omega;\R^{2})$ for all $1\leq r <2$, see Lemma \ref{lem:poincaresobolev}\ref{item:PoincareSobolev1}. Let $1< r <2$ be arbitrary, so that $\sobo_{0}^{1,2}(\Omega;\R^{2})\hookrightarrow\lebe^{r'}(\Omega;\R^{2}) \eqqcolon X'$; recall that $n=2$. Then $X\coloneqq \lebe^{r}(\Omega;\R^{2})\hookrightarrow \sobo^{-1,2}(\Omega;\R^{2})\eqqcolon X_{1}$, and so we have that $X_{0}\hookrightarrow\hookrightarrow X \hookrightarrow X_{1}$. By the classical Aubin-Lions lemma, we have 
\begin{align*}
\mathbb{W}\coloneqq \{\vv\in\lebe^{p}(0,T;X_{0})\colon\;\dot{\vv}\in \lebe^{q}(0,T;X_{1})\}\hookrightarrow\hookrightarrow \lebe^{p}(0,T;X), 
\end{align*}
where the space on the left-hand side is endowed with its canonical norm. We then set $p=1,q=2$ and so deduce \ref{eq:compa3} from Lemma \ref{lem:usefulapriori}. 

It remains to establish that $\uu_{\varepsilon}^{\zeta}\in\lebe_{\mathrm{w}^{*}}^{1}(0,T;\bd(\Omega))$. To this end, recall that the total Hibler deformation is lower semicontinuous with respect to $\lebe_{\locc}^{1}$-convergence, see Lemma \ref{lem:auxBD}\ref{item:aux0A}. By \ref{eq:compa3} and passing to another subsequence if required, we may assume that $\uu_{\delta_{j},\varepsilon}^{\zeta}(t,\cdot)\to \uu_{\varepsilon}^{\zeta}(t,\cdot)$ in $\lebe^{1}(\Omega;\R^{2})$ for $\mathscr{L}^{1}$-a.e. $t\in (0,T)$. Hence, for each such $t$, 
\begin{align}\label{eq:preFatou}
|\mathbb{T}\uu_{\varepsilon}^{\zeta}(t,\cdot)|(\Omega) \leq \liminf_{j\to\infty} \int_{\Omega}|\mathbb{T}\uu_{\delta_{j},\varepsilon}^{\zeta}(t,x)|\dif x. 
\end{align} 
Next, denote by $\mathscr{X}$ a countable dense subset of the maps $\bm{\varphi}\in\hold_{c}^{\infty}(\Omega;\RR)$ with $|\bm{\varphi}|\leq 1$, so that 
\begin{align}\label{eq:totvarpointwiseT}
|\mathbb{T}\uu_{\varepsilon}^{\zeta}(t,\cdot)|(\Omega) = \sup_{\varphi\in\mathscr{X}}\int_{\Omega}\uu_{\varepsilon}^{\zeta}(t,\cdot)\cdot\mathbb{T}^{*}\bm{\varphi}\dif x. 
\end{align}
Note that, since $\uu_{\varepsilon}^{\zeta}\in\lebe^{2}(0,T;\lebe^{2}(\Omega;\RR))$, the map 
\begin{align*}
(0,T)\ni t\mapsto \int_{\Omega}\uu_{\varepsilon}^{\zeta}(t,\cdot)\cdot\mathbb{T}^{*}\bm{\varphi}\dif x \qquad\text{is $\mathscr{L}^{1}$-measurable} 
\end{align*}
for every $\bm{\varphi}\in\hold_{c}^{\infty}(\Omega;\RR)$. In view of \eqref{eq:totvarpointwiseT}, it follows that $t\mapsto |\mathbb{T}\uu_{\varepsilon}^{\zeta}(t,\cdot)|(\Omega)$ is $\mathscr{L}^{1}$-measurable as the pointwise supremum over a countable set. In particular, we may integrate \eqref{eq:preFatou} with respect to $t\in (0,T)$ and subsequently use Fatou's lemma to deduce 
\begin{align}\label{eq:johnzorn}
\int_{0}^{T}|\mathbb{T}\uu_{\varepsilon}^{\zeta}(t,\cdot)|(\Omega)\dif t \leq \liminf_{j\to\infty}\int_{0}^{T}\int_{\Omega}|\mathbb{T}\uu_{\delta_{j},\varepsilon}^{\zeta}(t,x)|\dif x \dif t \stackrel{\text{Lemma \ref{lem:usefulapriori}\ref{eq:unifobd2}}}{\leq} c <\infty. 
\end{align}
In conjunction with $\uu_{\varepsilon}\in\lebe^{1}(0,T;\lebe^{1}(\Omega;\R^{2}))$, we conclude that $\uu_{\varepsilon}\in\lebe_{\mathrm{w}^{*}}^{1}(0,T;\bd(\Omega))$. We come to \ref{eq:compa4}. To this end, we observe that, by Lemma \ref{lem:usefulapriori}, $(\partial_{t}\uu_{\delta,\varepsilon}^{\zeta})$ is uniformly bounded in $\lebe^{2}(0,T;\sobo^{-1,2}(\Omega;\R^{2}))$. Hence, passing to another non-relabelled subsequence if required, there exists $\vv\in\lebe^{2}(0,T;\sobo^{-1,2}(\Omega;\R^{2}))$ such that 
\begin{align}\label{eq:hieb}
\partial_{t}\uu_{\delta_{j},\varepsilon}^{\zeta}\stackrel{*}{\rightharpoonup} \vv \qquad\text{in}\;\lebe^{2}(0,T;\sobo^{-1,2}(\Omega;\R^{2})). 
\end{align}
For any $\bm{\psi}\in\hold_{c}^{\infty}(\Omega_{T};\R^{2})$, we have 
\begin{align*}
\int_{0}^{T}\langle\vv(t,\cdot),\bm{\psi}(t,\cdot)\rangle \dif t & = \lim_{j\to\infty}\int_{0}^{T}\langle\partial_{t}\uu_{\delta_{j},\varepsilon}^{\zeta}(t,\cdot),\bm{\psi}(t,\cdot)\rangle\dif t \\ 
& = - \lim_{j\to\infty} \int_{0}^{T}\langle \uu_{\delta_{j},\varepsilon}^{\zeta}(t,\cdot),\partial_{t}\bm{\psi}(t,\cdot)\rangle\dif t = - \int_{0}^{T}\langle\uu_{\varepsilon}^{\zeta}(t,\cdot),\partial_{t}\bm{\psi}(t,\cdot)\rangle\dif t, 
\end{align*}
where we have used \ref{eq:compa1}. In particular, by density, $\uu_{\varepsilon}^{\zeta}$ is weakly differentiable in time with $\partial_{t}\uu_{\varepsilon}^{\zeta}=\vv$, and so \eqref{eq:hieb} yields \ref{eq:compa4}. The estimate \eqref{eq:epsbounds} is a routine consequence of Lemma \ref{lem:usefulapriori}, the preceding convergences and the lower semicontinuity of norms with respect to the underlying weak or weak*-convergence, respectively. 

Based on \eqref{eq:epsbounds}, we may argue as in \ref{eq:compa1}--\ref{eq:compa4} and as for \eqref{eq:epsbounds} to deduce \ref{eq:compa5}--\ref{eq:compa8} and \eqref{eq:zetabounds}.  From here, \ref{eq:compa9}--\ref{eq:compa12} follow too. The proof is complete. 
\end{proof}
At this stage, we go back to \eqref{eq:approxineq} and write the inequality for $\uu_{\delta_{j},\varepsilon}^{\zeta}$ and test maps $\mathbf{v}\in\lebe^{2}(0,T;\sobo_{0}^{1,2}(\Omega;\R^{2}))\cap\sobo^{1,2}(0,T;\lebe^{2}(\Omega;\R^{2}))$ as 
\begin{align}
\int_{0}^{s}\int_{\Omega}&(\partial_{t}\mathbf{v})\cdot(\mathbf{v}-\mathbf{u}_{\delta_{j},\varepsilon}^{\zeta})\dif x\dif t  + \int_{0}^{s}\int_{\Omega}F_{\varepsilon}(\mathbb{T}\mathbf{v})\dif x\dif t + \frac{\delta_{j}}{2}\int_{0}^{s}\int_{\Omega}|\mathbb{T}\mathbf{v}|^{2}\dif x\dif t - \int_{0}^{s}\int_{\Omega}\mathbf{f}\cdot\mathbf{v}\dif x \dif t \notag\\ 
& \geq \int_{0}^{s}\int_{\Omega}F_{\varepsilon}(\mathbb{T}\mathbf{u}_{\delta_{j},\varepsilon}^{\zeta})\dif x \dif t + \frac{\delta_{j}}{2}\int_{0}^{s}\int_{\Omega}|\mathbb{T}\mathbf{u}_{\delta_{j},\varepsilon}^{\zeta}|^{2}\dif x \dif t - \int_{0}^{s}\int_{\Omega}\mathbf{f}\cdot\mathbf{u}_{\delta_{j},\varepsilon}^{\zeta}\dif x \dif t \notag\\ & + \int_{0}^{s}\int_{\Omega}\tocean(\uu_{\delta_{j},\varepsilon}^{\zeta})\cdot(\vv-\uu_{\delta_{j},\varepsilon}^{\zeta})\dif x \dif t 
 + \frac{1}{2}\int_{\Omega}\Big(|\mathbf{u}_{\delta_{j},\varepsilon}^{\zeta}-\mathbf{v}|^{2}(s,\cdot) - |\mathbf{u}_{0}^{\zeta}-\mathbf{v}(0,\cdot)|^{2}\Big) \dif x \notag \\ 
 & \Longleftrightarrow: \mathrm{I}_{1} + \mathrm{I}_{2} + \mathrm{I}_{3} + \mathrm{I}_{4} \geq \mathrm{I}_{5} + \mathrm{I}_{6} + \mathrm{I}_{7} + \mathrm{I}_{8} + \mathrm{I}_{9}\label{eq:kismet}
\end{align}
for $\mathscr{L}^{1}$-a.e. $0\leq s \leq T$. We deal with the single terms separately: 
\begin{itemize}
    \item On \namedlabel{item:term1}{Term $\mathrm{I}_{1}$}: Here we use $\partial_{t}\mathbf{v}\in\lebe^{2}(0,T;\lebe^{2}(\Omega;\R^{2}))$ and Corollary \ref{cor:compa1}\ref{eq:compa1a} to conclude
    \begin{align*}
    \mathrm{I}_{1} = \int_{0}^{s}\int_{\Omega}&(\partial_{t}\mathbf{v})\cdot(\mathbf{v}-\mathbf{u}_{\delta_{j},\varepsilon}^{\zeta})\dif x\dif t \stackrel{j\to\infty}{\longrightarrow} \int_{0}^{s}\int_{\Omega}(\partial_{t}\mathbf{v})\cdot(\mathbf{v}-\mathbf{u}_{\varepsilon}^{\zeta})\dif x \dif t. 
    \end{align*}
    \item On \namedlabel{item:term2}{Term $\mathrm{I}_{2}$}: We leave term $\mathrm{I}_{2}$ untouched. 
    \item On \namedlabel{item:term3}{Term $\mathrm{I}_{3}$}: Here we use $\mathbf{v}\in\lebe^{2}(0,T;\sobo_{0}^{1,2}(\Omega;\R^{2}))$ to conclude that 
    \begin{align*}
    \mathrm{I}_{3} = \frac{\delta_{j}}{2}\int_{0}^{s}\int_{\Omega}|\nabla\mathbf{v}|^{2}\dif x \dif t \stackrel{j\to\infty}{\longrightarrow} 0. 
    \end{align*}
    \item On \namedlabel{item:term4}{Term $\mathrm{I}_{4}$}: We leave term $\mathrm{I}_{4}$ untouched.
    \item On \namedlabel{item:term5}{Term $\mathrm{I}_{5}$}: For $\mathrm{I}_{5}$, we employ Corollary \ref{cor:compa1}\ref{eq:compa3}. In particular, by passing to a non-relabelled subsequence if necessary, the Riesz-Fischer theorem implies that 
    \begin{align*}
    \mathbf{u}_{\delta_{j},\varepsilon}^{\zeta}(t,\cdot) \stackrel{j\to\infty}{\longrightarrow} \mathbf{u}_{\varepsilon}^{\zeta}(t,\cdot)\qquad\text{strongly in}\;\lebe^{1}(\Omega;\R^{2})\;\text{for}\;\text{$\mathscr{L}^{1}$-a.e.}\;0<t<T. 
    \end{align*}
By the lower semicontinuity result from Corollary \ref{cor:LSCplussbdryvalues}, this implies that 
\begin{align}\label{eq:trinity}
\liminf_{j\to\infty} \int_{\Omega}F_{\varepsilon}(\mathbb{T}\mathbf{u}_{\delta_{j},\varepsilon}^{\zeta}(t,\cdot))\dif x & = \liminf_{j\to\infty} (\mathscr{F}_{\varepsilon})_{0}^{*}[\mathbf{u}_{\delta_{j},\varepsilon}^{\zeta}(t,\cdot);\Omega] \geq (\mathscr{F}_{\varepsilon})_{0}^{*}[\mathbf{u}_{\varepsilon}^{\zeta}(t,\cdot);\Omega]
\end{align}
for $\mathscr{L}^{1}$-a.e. $0<t<T$. Now, Fatou's lemma yields
    \begin{align}\label{eq:tabenstaedt}
    \begin{split}
    \liminf_{j\to\infty}\mathrm{I}_{5} = \liminf_{j\to\infty}\int_{0}^{s}\int_{\Omega}F_{\varepsilon}(\mathbb{T}\uu_{\delta_{j},\varepsilon}^{\zeta})\dif x \dif t & \geq \int_{0}^{s}\liminf_{j\to\infty}\int_{\Omega}F_{\varepsilon}(\mathbb{T}\uu_{\delta_{j},\varepsilon}^{\zeta})\dif x \dif t \\ 
    & \geq \int_{0}^{s}(\mathscr{F}_{\varepsilon})_{0}^{*}[\mathbf{u}_{\varepsilon}^{\zeta}(t,\cdot);\Omega]\dif t.  
    \end{split}
    \end{align}
    We recall from Lemma \ref{lem:approximations} that $F\leq F_{\varepsilon}$ and $F^{\infty}= F_{\varepsilon}^{\infty}$, whereby \eqref{eq:tabenstaedt} gives us
\begin{align*}\liminf_{j\to\infty}\mathrm{I}_{5} \geq \int_{0}^{s}\mathscr{F}_{0}^{*}[\mathbf{u}_{\varepsilon}^{\zeta}(t,\cdot);\Omega]\dif t.  
    \end{align*}
    \item On \namedlabel{item:term6}{Term $\mathrm{I}_{6}$}: We deal with $\mathrm{I}_{6}$ by employing the rough estimate $\mathrm{I}_{6}\geq 0$. 
    \item On \namedlabel{item:term7}{Term $\mathrm{I}_{7}$}: For term $\mathrm{I}_{7}$, we use Corollary \ref{cor:compa1}\ref{eq:compa1a} and $\mathbf{f}\in\lebe^{2}(0,T;\lebe^{2}(\Omega;\R^{2}))$ to obtain 
    \begin{align*}
    \mathrm{I}_{7} = \int_{0}^{s}\int_{\Omega}\mathbf{f}\cdot \mathbf{u}_{\delta_{j},\varepsilon}^{\zeta}\dif x \dif t \stackrel{j\to\infty}{\longrightarrow} \int_{0}^{s}\int_{\Omega}\mathbf{f}\cdot\mathbf{u}_{\varepsilon}^{\zeta}\dif x \dif t. 
    \end{align*}
    \item On \namedlabel{item:term8}{Term $\mathrm{I}_{8}$}: For term $\mathrm{I}_{8}$, we estimate \begin{align}\label{eq:Iransofaraway}
    \begin{split}
        & \left\vert \int_{0}^{T}\int_{\Omega}\tocean(\uu_{\delta_{j},\varepsilon}^{\zeta})\cdot(\vv-\uu_{\delta_{j},\varepsilon}^{\zeta})-\tocean(\uu_{\varepsilon}^{\zeta})\cdot(\mathbf{v}-\uu_{\varepsilon}^{\zeta})\dif x \dif t\right\vert \\ 
        & \;\;\;\;\;\;\;\;\leq \int_{0}^{T}\int_{\Omega}|\tocean(\uu_{\delta_{j},\varepsilon}^{\zeta})-\tocean(\uu_{\varepsilon}^{\zeta})|\,|\vv-\uu_{\delta_{j},\varepsilon}^{\zeta}|\dif x \dif t \\ & \;\;\;\;\;\;\;\; + \int_{0}^{T}\int_{\Omega} |\tocean(\uu_{\varepsilon}^{\zeta})|\,|\uu_{\varepsilon}^{\zeta}-\uu_{\delta_{j},\varepsilon}^{\zeta}|\dif x \dif t \eqqcolon  \mathrm{I}_{8}^{(1)} + \mathrm{I}_{8}^{(2)}. 
        \end{split}
    \end{align}
    We recall \eqref{eq:oceanicmain} and \eqref{eq:oceanicassump}; in particular, $\tocean(\mathbf{w})=\widetilde{\eta}(\mathbf{U}_{\mathrm{ocean}}-\mathbf{w})$ with a $(1-\gamma)$-H\"{o}lder continuous map $\widetilde{\eta}\colon\R^{2}\to\R^{2}$ such that $\widetilde{\eta}(0)=0$. For some $0<\varepsilon'<1$ sufficiently small and to be determined later, we put 
    \begin{align*}
    p \coloneqq \frac{2-\varepsilon'}{1-\gamma}\;\;\;\text{and so}\;\;\;p' = \frac{2-\varepsilon'}{1+\gamma-\varepsilon'}. 
    \end{align*}
    In particular, we have \begin{align}\label{eq:p'compute}
    0<\varepsilon'<2\gamma \Longrightarrow 2-\varepsilon' < 2 (1+\gamma-\varepsilon') \Longrightarrow p'<2. 
    \end{align}
    In consequence, we arrive at 
    \begin{align*}
    \mathrm{I}_{8}^{(1)} & \leq c\int_{0}^{T}\int_{\Omega}|\uu_{\delta_{j},\varepsilon}^{\zeta}-\uu_{\varepsilon}^{\zeta}|^{1-\gamma}|\mathbf{v}-\uu_{\delta_{j},\varepsilon}^{\zeta}|\dif x \dif t \\ 
    & \leq c \int_{0}^{T}\Big(\int_{\Omega}|\uu_{\delta_{j},\varepsilon}^{\zeta}-\uu_{\varepsilon}^{\zeta}|^{2-\varepsilon'}\dif x\Big)^{\frac{1-\gamma}{2-\varepsilon'}}\Big(\int_{\Omega}|\vv-\uu_{\delta_{j},\varepsilon}^{\zeta}|^{\frac{2-\varepsilon'}{1+\gamma-\varepsilon'}}\dif x\Big)^{\frac{1+\gamma-\varepsilon'}{2-\varepsilon'}}\dif t \\ 
    & \leq c\, \Big(\int_{0}^{T}\Big(\int_{\Omega}|\uu_{\delta_{j},\varepsilon}^{\zeta}-\uu_{\varepsilon}^{\zeta}|^{2-\varepsilon'}\dif x \Big)^{\frac{1}{2-\varepsilon'}}\dif t\Big)^{1-\gamma} \\ 
    & \times \Big(\int_{0}^{T}\Big(\int_{\Omega}|\vv-\uu_{\delta_{j},\varepsilon}^{\zeta}|^{\frac{2-\varepsilon'}{1+\gamma-\varepsilon'}}\dif x \Big)^{\frac{1+\gamma-\varepsilon'}{2-\varepsilon'}\frac{1}{\gamma}}\dif t\Big)^{\gamma}. \\ 
    & \!\!\!\! \stackrel{\eqref{eq:p'compute}}{\leq} c\, \Big(\int_{0}^{T}\Big(\int_{\Omega}|\uu_{\delta_{j},\varepsilon}^{\zeta}-\uu_{\varepsilon}^{\zeta}|^{2-\varepsilon'}\dif x \Big)^{\frac{1}{2-\varepsilon'}}\dif t\Big)^{1-\gamma}  \Big(\int_{0}^{T}\Big(\int_{\Omega}|\vv-\uu_{\delta_{j},\varepsilon}^{\zeta}|^{2}\dif x \Big)^{\frac{1}{2}\cdot\frac{1}{\gamma}}\dif t\Big)^{\gamma}. 
    \end{align*}
    By Corollary \ref{cor:compa1}\ref{eq:compa3}, the first term of the ultimate product tends to zero as $j\to\infty$. Since $\vv\in\sobo^{1,2}(0,T;\lebe^{2}(\Omega;\R^{2}))\hookrightarrow\lebe^{\infty}(0,T;\lebe^{2}(\Omega;\R^{2}))$ and $(\uu_{\delta_{j},\varepsilon}^{\zeta})$ is uniformly bounded in $\lebe^{\infty}(0,T;\lebe^{2}(\Omega;\R^{2}))$ by Lemma \ref{lem:usefulapriori}\ref{eq:unifobd1}, the second term of the ultimate product is uniformly bounded. Therefore, $
\lim_{j\to\infty} \mathrm{I}_{8}^{(1)} = 0$. 
For  $\mathrm{I}_{8}^{(2)}$, we recall that $|\widetilde{\eta}(z)|\leq c|z|^{1-\gamma}$ for all $z\in\RR$. We let $\varepsilon'>0$ be such that 
\begin{align}\label{eq:loewe}
0<\varepsilon'<\frac{2\gamma}{1+\gamma}\;\;\;\text{and so}\;\;\;(1-\gamma)\frac{2-\varepsilon'}{1-\varepsilon'}<2. 
\end{align}
It follows that 
\begin{align*}
\mathrm{I}_{8}^{(2)} & \leq c\int_{0}^{T}\int_{\Omega} |\mathbf{U}_{\mathrm{ocean}}-\uu_{\varepsilon}^{\zeta}|^{1-\gamma}|\uu_{\varepsilon}^{\zeta}-\uu_{\delta_{j},\varepsilon}^{\zeta}|\dif x \dif t \\ 
& \leq c \int_{0}^{T}\Big(\Big( \int_{\Omega}|\mathbf{U}_{\mathrm{ocean}}-\mathbf{u}_{\varepsilon}^{\zeta}|^{(1-\gamma)\frac{2-\varepsilon'}{1-\varepsilon'}}\dif x \Big)^{\frac{1-\varepsilon'}{(2-\varepsilon')(1-\gamma)}}\Big)^{1-\gamma}\Big(\int_{\Omega}|\uu_{\varepsilon}^{\zeta}-\uu_{\delta_{j},\varepsilon}^{\zeta}|^{2-\varepsilon'}\dif x \Big)^{\frac{1}{2-\varepsilon'}}\dif t \\ 
& \!\!\!\! \stackrel{\eqref{eq:loewe}}{\leq} c\,\|\mathbf{U}_{\mathrm{ocean}}-\uu_{\varepsilon}^{\zeta}\|_{\lebe^{\infty}(0,T;\lebe^{2}(\Omega))}^{1-\gamma}\|\uu_{\varepsilon}^{\zeta}-\uu_{\delta_{j},\varepsilon}^{\zeta}\|_{\lebe^{1}(0,T;\lebe^{2-\varepsilon'}(\Omega))}. 
\end{align*}
By \eqref{eq:oceanicmean} and Lemma \ref{lem:usefulapriori}\ref{eq:unifobd1}, the first term of the ultimate product is finite. Moreover, by Corollary  \ref{cor:compa1}\ref{eq:compa3}, the second term converges to zero as $j\to\infty$. In conclusion, $\lim_{j\to\infty}\mathrm{I}_{8}^{(2)}=0$ and so $\lim_{j\to\infty}\mathrm{I}_{8}=0$.
    \item On \namedlabel{item:term9}{Term $\mathrm{I}_{9}$}: Here, we use Corollary \ref{cor:compa1}\ref{eq:compa3}. By the Riesz-Fischer theorem, we may assume that $\mathbf{u}_{\delta_{j},\varepsilon}^{\zeta}(s,\cdot)\to \mathbf{u}_{\varepsilon}^{\zeta}(s,\cdot)$ strongly in $\lebe^{1}(\Omega;\R^{2})$ for $\mathscr{L}^{1}$-a.e. $0<s<T$. For any such $0s<<T$, we may pass to a non-relabelled subsequence to achieve $\uu_{\delta_{j},\varepsilon}^{\zeta}(s,\cdot)\to\uu_{\varepsilon}^{\zeta}(s,\cdot)$ $\mathscr{L}^{2}$-a.e. in $\Omega$. Hence, Fatou's lemma implies that 
    \begin{align*}
    \liminf_{j\to\infty} \mathrm{I}_{9} & \geq \frac{1}{2}\int_{\Omega}\Big(|\uu_{\varepsilon}^{\zeta}-\mathbf{v}|^{2}(s,\cdot)-|\uu_{0}^{\zeta}-\mathbf{v}(0,\cdot)|^{2}\Big)\dif s. 
    \end{align*}
\end{itemize}
Combining the previous estimates, we therefore arrive at 
\begin{align}\label{eq:franklundy}
\begin{split}
\int_{0}^{s}\int_{\Omega}(\partial_{t}\mathbf{v})\cdot(\mathbf{v}-\mathbf{u}_{\varepsilon}^{\zeta})\dif x\dif t  & + \int_{0}^{s}\int_{\Omega}F_{\varepsilon}(\mathbb{T}\mathbf{v})\dif x\dif t - \int_{0}^{s}\int_{\Omega}\mathbf{f}\cdot\mathbf{v}\dif x \dif t\\ 
& \geq \int_{0}^{s}\mathscr{F}_{0}^{*}[\mathbf{u}_{\varepsilon}^{\zeta}(t,\cdot);\Omega] \dif t - \int_{0}^{s}\int_{\Omega}\mathbf{f}\cdot\mathbf{u}_{\varepsilon}^{\zeta}\dif x \dif t \\ 
& + \int_{0}^{s}\int_{\Omega}\tocean(\uu_{\varepsilon}^{\zeta})\cdot(\mathbf{v}-\mathbf{u}_{\varepsilon}^{\zeta})\dif x\dif t 
\\ & + \frac{1}{2}\int_{\Omega}\Big(|\mathbf{u}_{\varepsilon}^{\zeta}-\mathbf{v}|^{2}(s,\cdot) - |\mathbf{u}_{0}^{\zeta}-\mathbf{v}(0,\cdot)|^{2}\Big) \dif x \\ 
 & \Longleftrightarrow: \mathrm{I}_{10} + \mathrm{I}_{11} + \mathrm{I}_{12} \geq \mathrm{I}_{13} + \mathrm{I}_{14} + \mathrm{I}_{15} + \mathrm{I}_{16} 
 \end{split}
\end{align}
for all $\mathbf{v}\in \lebe^{2}(0,T;\sobo_{0}^{1,2}(\Omega;\R^{2}))\cap\sobo^{1,2}(0,T;\lebe^{2}(\Omega;\R^{2}))$ and $\mathscr{L}^{1}$-a.e. $0<s<T$. 

At this stage, it is important to remark the competitors $\mathbf{v}$ in \eqref{eq:franklundy} belong spatially to $\sobo_{0}^{1,2}(\Omega;\R^{2})$; this is due to the fact that the weak formulation \eqref{eq:weakform} requires spatial zero boundary values. The Definition \ref{def:varsol1} of variational solutions does  \emph{not} require spatial zero boundary values. Indeed, the latter enter the relaxed functional via boundary penalisation terms. We thus aim to admit
$\mathbf{v}\in\lebe_{\mathrm{w}^{*}}^{1}(0,T;\bd(\Omega))$ with $\partial_{t}\mathbf{v}\in\lebe^{2}(\Omega_{T};\R^{2})$ in  \eqref{eq:franklundy}. For this, we require a particular approximation whose construction we give next. 

Let $\vv\in\lebe_{\mathrm{w}^{*}}^{1}(0,T;\bd(\Omega))\cap\sobo^{1,2}(0,T;\lebe^{2}(\Omega;\R^{2}))$. We choose a smooth cut-off  $\psi\in\hold_{c}^{\infty}((-T;2T);[0,1])$ such that $\mathbbm{1}_{[0,T]}\leq \psi \leq\mathbbm{1}_{(-T,2T)}$. We then put 
\begin{align*}
\mathbf{w}(t,\cdot) \coloneqq \begin{cases}
\vv(-t,\cdot)&\; -T<t\leq 0, \\ 
\vv(t,\cdot)&\;0\leq t\leq T, \\ 
\vv(2T-t,\cdot)&\;T\leq t \leq 2T, 
\end{cases}
\end{align*}
and $\overline{\mathbf{v}}\coloneqq \psi\mathbf{w}$, which we view as a function $\R\to\bd(\Omega)$. By construction, we have both $\overline{\mathbf{v}}\in\lebe_{\mathrm{w}^{*}}^{1}(0,T;\bd(\Omega))$ and $\partial_{t}\overline{\mathbf{v}}\in\lebe^{2}(\R\times\Omega;\R^{2})$; the latter follows from the construction of $\mathbf{w}$ as the even reflection of $\vv$ at $0$ or $T$, respectively. 

Let $\varrho_{\theta}\in\hold_{c}^{\infty}(\R;\R_{\geq 0})$ be the $\theta$-rescaled variant of a standard mollifier with respect to time, and let $\rho_{r}\in\hold_{c}^{\infty}(\R^{2};\R_{\geq 0})$ be the $r$-rescaled variant of a standard mollifier with respect to space. Moreover, let $\eta_{d}\in\mathrm{Lip}_{c}(\Omega;[0,1])$ be as in Theorem \ref{thm:bdryapprox}. For a sequence $(d_{j})\subset\R_{>0}$ with $d_{j}\searrow 0$, we choose a double sequence $(r_{k,j})\subset\R_{>0}$ such that $\spt(\rho_{r_{k,j}}*\eta_{d_{j}})\Subset\Omega$ for all $k,j\in\mathbb{N}$ and $\lim_{k\to\infty}r_{k,j}= 0$ for all $j\in\mathbb{N}$. We then define, for an arbitrary but fixed sequence $({\theta_{l}})\subset(0,1)$ with $\theta_{l}\searrow 0$:  
\begin{align}\label{eq:approximationmain}
\left.\mathbf{v}_{j,k,l}\coloneqq \Big(\varrho_{\theta_{l}}*\big(\rho_{r_{k,j}}*(\eta_{d_{j}}\overline{\mathbf{v}})\big)\Big)\right|_{[0,T]\times\Omega}. 
\end{align}
Then, because of $\vv\in\lebe^{1}(0,T;\lebe^{1}(\Omega;\R^{2}))$, we have  $\mathbf{v}_{j,k,l}\in\hold_{b}((0,T);\sobo_{c}^{1,2}(\Omega;\R^{2}))$. 
\begin{lemma}\label{lem:keyapproximation}
Let $\Omega\subset\R^{2}$ be open and bounded with Lipschitz boundary, and let   $\mathbf{v}\in\lebe_{\mathrm{w}^{*}}^{1}(0,T;\bd(\Omega))\cap\sobo^{1,2}(0,T;\lebe^{2}(\Omega;\R^{2}))$. Then the triple sequence $(\vv_{j,k,l})$ from \eqref{eq:approximationmain} has the following properties: 
\begin{enumerate}
\item\label{item:approxmain1} For all $j,k,l\in\mathbb{N}$, $\partial_{t}\mathbf{v}_{j,k,l}\in\lebe^{2}(\Omega_{T};\R^{2})$, and we have 
\begin{align*}
\limsup_{j\to\infty}\big(\limsup_{k\to\infty}\big(\limsup_{l\to\infty}\|\partial_{t}\vv_{j,k,l}-\partial_{t}\vv\|_{\lebe^{2}(0,T;\lebe^{2}(\Omega))}\big)\big)=0. 
\end{align*}
\item\label{item:approxmain2} We have 
\begin{align*}
\limsup_{j\to\infty}\big(\limsup_{k\to\infty}\big(\limsup_{l\to\infty}\|\vv_{j,k,l}-\vv\|_{\lebe^{2}(0,T;\lebe^{2}(\Omega))}\big)\big)=0.
\end{align*}
\item\label{item:approxmain3} Let $F\colon\RR\to\R$ be convex and satisfy the linear growth assumption \eqref{eq:lingrowth}. Then we have 
\begin{align*}
\limsup_{j\to\infty} \limsup_{k\to\infty}\limsup_{l\to\infty}\int_{0}^{T}\int_{\Omega}F(\mathbb{T}\vv_{j,k,l})\dif x \dif t \leq \int_{0}^{T}\mathscr{F}_{0}^{*}[\vv(t,\cdot);\Omega]\dif t. 
\end{align*}
\end{enumerate}
\end{lemma}
\begin{proof}
Properties  \ref{item:approxmain1} and \ref{item:approxmain2} are routine consequences of general results on convolutions; for \ref{item:approxmain1}, we explicitly note that the cut-off $\eta_{d_{j}}$ is purely spatial. As to  \ref{item:approxmain3}, we define for $k,j\in\mathbb{N}$
\begin{align*}
\Phi_{j,k}^{\vv}\colon t\mapsto \int_{\Omega}F(\mathbb{T}(\rho_{r_{k,j}}*(\eta_{d_{j}}\vv))(t,x))\dif x, 
\end{align*}
so that $\Phi_{j,k}^{\vv}(t)\in\lebe^{1}((0,T))$ by \ref{item:Fprop1}. Let $
\mathscr{N}_{j,k}$ be the set of all non-Lebesgue points $0<t<T$ of $\Phi_{j,k}^{\vv}$; then $\mathscr{L}^{1}(\mathscr{N}_{j,k})=0$ for all $k,j\in\mathbb{N}$. Moreover, there exists $\mathscr{N}'\subset(0,T)$ with $\mathscr{L}^{1}(\mathscr{N}')=0$ such that $\vv(t,\cdot)$ -- and so $\eta_{d_{j}}\vv(t,\cdot)$ -- belong to $\bd(\Omega)$ for all  $t\in(0,T)\setminus\mathscr{N}'$. Thus,
\begin{align*}
\mathscr{L}^{1}(\mathscr{N})\coloneqq \mathscr{L}^{1}\Big(\mathscr{N}'\cup\bigcup_{j,k\in\mathbb{N}}\mathscr{N}_{j,k} \Big) = 0. 
\end{align*}
By the definition of $\mathscr{N}$ and the convexity of $F$, we have 
\begin{align}
\begin{split}
\int_{0}^{T}\int_{\Omega}F(\mathbb{T}\mathbf{v}_{j,k,l}(t,x))\dif x \dif t & \stackrel{\text{Jensen}}{\leq}  \int_{[0,T]\setminus\mathscr{N}}\int_{\R}\varrho_{\theta_{l}}(t-s)\Phi_{j,k}^{\mathbf{v}}(s)\dif s \dif t \\ 
& \,\stackrel{l\to\infty}{\longrightarrow} \int_{0}^{T}\Phi_{j,k}^{\mathbf{v}}(t)\dif t \\ 
& \;\;\, =  \int_{0}^{T}\int_{\Omega}F(\rho_{r_{k,j}}*\mathbb{T}(\eta_{d_{j}}\vv(t,\cdot)))\dif x \dif t \\ 
& \!\!\!\!\!\!\!\!\!\!\!  \stackrel{\text{Jensen, Lem. \ref{lem:jensen}}}{ \leq} \int_{0}^{T}\Big(\int_{\Omega}F(\mathbb{T}(\eta_{d_{j}}\vv(t,\cdot)))\Big)\dif t \\ 
& \;\;\, \eqqcolon   \int_{0}^{T}\Phi_{j}^{\vv}(t)\dif t
\end{split}
\end{align}
with an obvious definition of $\Phi_{j}^{\vv}$. In particular, in the application of Jensen's inequality for measures from Lemma \ref{lem:jensen}, we used that $\spt(\rho_{r_{k,j}}*\eta_{d_{j}})\subset\Omega$ for all $k,j\in\mathbb{N}$.

In order to accomplish the limit passage $j\to\infty$, we use the upper part of Fatou's lemma. To this end, we apply Corollary \ref{cor:unibounds}. In conjunction with the linear growth hypothesis \eqref{eq:lingrowth} on $F$, we infer that there exists a constant $c>0$ independent of $\vv$ and $j$ such that 
\begin{align*}
\Phi_{j}^{\vv}(t) \leq c(\mathscr{L}^{2}(\Omega)+\|\vv(t,\cdot)\|_{\bd(\Omega)})\qquad\text{for $\mathscr{L}^{1}$-a.e. $t\in (0,T)$}. 
\end{align*}
Since $\vv\in\lebe_{\mathrm{w}^{*}}^{1}(0,T;\bd(\Omega))$, the right-hand side of the ultimate inequality belongs to $\lebe^{1}((0,T))$, see \eqref{eq:weak*-bochner}. In particular, $\Phi_{j}^{\vv}$ has an integrable majorant. By Theorem \ref{thm:bdryapprox} on the bulk approximation of boundary terms, we thus have 
\begin{align*}
\limsup_{j\to\infty}\int_{\Omega}F(\mathbb{T}(\eta_{d_{j}}\vv(t,\cdot)))\leq \mathscr{F}_{0}^{*}[\vv(t,\cdot);\Omega]\qquad\text{for $\mathscr{L}^{1}$-a.e. $t\in (0,T)$}. 
\end{align*}
Therefore, we may apply the upper part of Fatou's lemma to conclude 
\begin{align}
\limsup_{j\to\infty}\limsup_{k\to\infty}\limsup_{l\to\infty} \int_{0}^{T}\int_{\Omega}F(\mathbb{T}\vv_{j,k,l}(t,x))\dif x \dif t \leq \int_{0}^{T}\mathscr{F}_{0}^{*}[\vv(t,\cdot);\Omega]\dif t. 
\end{align}
This implies \ref{item:approxmain3}, and the proof is complete. 
\end{proof}
\begin{remark}\label{rem:bulk3}
In the estimates below, the actual critical term is \ref{item:term11} and shall be treated by Lemma \ref{lem:keyapproximation}. In particular, this is the key point for which we require the bulk approximation of boundary terms as established in Theorem \ref{thm:bdryapprox}. 
\end{remark}
Now let $\vv\in\lebe_{\mathrm{w}^{*}}^{1}(0,T;\bd(\Omega))\cap\sobo^{1,2}(0,T;\lebe^{2}(\Omega;\R^{2}))$ be arbitrary and let $(\vv_{j,k,l})\subset \lebe^{2}(0,T;\sobo_{0}^{1,2}(\Omega;\R^{2}))\cap \sobo^{1,2}(0,T;\lebe^{2}(\Omega;\R^{2}))$ be the triple sequence from Lemma \ref{lem:keyapproximation}, where we work with $F_{\varepsilon}$ as convex integrand of linear growth in Lemma \ref{lem:keyapproximation}\ref{item:approxmain3}. In particular, for all $(j,k,l)\in\mathbb{N}^{3}$, $\vv_{j,k,l}$ is admissible in \eqref{eq:franklundy}. We organise the resulting inequality as 
\begin{align*}
\mathrm{I}_{10}^{(j,k,l)}+\mathrm{I}_{11}^{(j,k,l)}+\mathrm{I}_{12}^{(j,k,l)} \geq \mathrm{I}_{13} + \mathrm{I}_{14} + \mathrm{I}_{15}^{(j,k,l)} + \mathrm{I}_{16}^{(j,k,l)}. 
\end{align*}
Again, we deal with the single terms separately: 
\begin{itemize}
\item On \namedlabel{item:term10}{Term $\mathrm{I}_{10}^{(j,k,l)}$}: For term $\mathrm{I}_{10}^{(j,k,l)}$, we send (in this order) $l\to\infty$, $k\to\infty$, $j\to\infty$. The strong convergences in $\lebe^{2}(\Omega_{T};\R^{2})$ as displayed in Lemma \ref{lem:keyapproximation}\ref{item:approxmain1} and \ref{item:approxmain2} imply 
\begin{align*}
\lim_{j\to\infty}\lim_{k\to\infty}\lim_{l\to\infty} \mathrm{I}_{10}^{(j,k,l)} = \int_{0}^{s}\int_{\Omega}\partial_{t}\mathbf{v}\cdot(\mathbf{v}-\mathbf{u}_{\varepsilon}^{\zeta})\dif x \dif t. 
\end{align*}
\item On \namedlabel{item:term11}{Term $\mathrm{I}_{11}^{(j,k,l)}$}: For the limit passage in  $\mathrm{I}_{11}^{(j,k,l)}$, we use Lemma \ref{lem:keyapproximation}\ref{item:approxmain3}. We thus conclude that 
\begin{align*}
\limsup_{j\to\infty}\limsup_{k\to\infty}\limsup_{l\to\infty} \mathrm{I}_{11}^{(j,k,l)} \leq \int_{0}^{T}(\mathscr{F}_{\varepsilon})_{0}^{*}[\vv(t,\cdot);\Omega]\dif t, 
\end{align*}
where $(\mathscr{F}_{\varepsilon})_{0}^{*}[-;\Omega]$ is the relaxed energy \eqref{eq:relaxedbdry} with energy integrand $F_{\varepsilon}$. 
\item On \namedlabel{item:term12}{Term $\mathrm{I}_{12}^{(j,k,l)}$}: For term $\mathrm{I}_{12}^{(j,k,l)}$, we note that $\mathbf{v}_{j,k,l}\to\mathbf{v}$ strongly in $\lebe^{2}(\Omega_{T};\R^{2})$ by Lemma \ref{lem:keyapproximation}\ref{item:approxmain1} as $j,k,l\to\infty$. Therefore, our assumption $\mathbf{f}\in\lebe^{2}(\Omega_{T};\R^{2})$, see \eqref{eq:atmosphericmain}, implies 
\begin{align*}
\lim_{j\to\infty}\lim_{k\to\infty}\lim_{l\to\infty}\mathrm{I}_{12}^{(j,k,l)} = \lim_{j\to\infty}\lim_{k\to\infty}\lim_{l\to\infty} \int_{0}^{s}\int_{\Omega}\mathbf{f}\cdot\mathbf{v}_{j,k,l}\dif x \dif t = \int_{0}^{s}\int_{\Omega}\mathbf{f}\cdot\mathbf{v}\dif x \dif t. 
\end{align*}
\item On \namedlabel{item:term1314}{Terms  $\mathrm{I}_{13},\mathrm{I}_{14}$}: We leave terms $\mathrm{I}_{13},\mathrm{I}_{14}$ untouched. 
\item On \namedlabel{item:term15}{Term $\mathrm{I}_{15}^{(j,k,l)}$}: For term $\mathrm{I}_{15}^{(j,k,l)}$, we recall \eqref{eq:oceanicmain}, \eqref{eq:oceanicassump} and \eqref{eq:oceanicmean}. In particular, we have $\tocean(\uu_{\varepsilon}^{\zeta})\in\lebe^{2}(\Omega_{T};\R^{2})$. Hence, by Lemma \ref{lem:keyapproximation}\ref{item:approxmain2}, 
\begin{align*}
\lim_{j\to\infty}\lim_{k\to\infty}\lim_{l\to\infty} \mathrm{I}_{15}^{(j,k,l)}= \int_{0}^{s}\int_{\Omega}\tocean(\uu_{\varepsilon}^{\zeta})\cdot(\vv-\uu_{\varepsilon}^{\zeta})\dif x \dif t.
\end{align*}

\item On \namedlabel{item:term16}{Term $\mathrm{I}_{16}^{(j,k,l)}$}: For term $\mathrm{I}_{16}^{(j,k,l)}$, we note that $\vv_{j,k,l}\to \vv$ strongly in $\lebe^{2}(\Omega_{T};\R^{2})$ and so $\vv_{j,k,l}(s,\cdot)\to\vv(s,\cdot)$ strongly in $\lebe^{2}(\Omega;\R^{2})$ for $\mathscr{L}^{1}$-a.e. $0<s<T$. In consequence, we obtain for $\mathscr{L}^{1}$-a.e. $0<s<T$ that 
\begin{align*}
\lim_{j\to\infty}\lim_{k\to\infty}\lim_{l\to\infty}\mathrm{I}_{16}^{(j,k,l)} = \frac{1}{2}\int_{\Omega}\Big(|\uu_{\varepsilon}^{\zeta}-\mathbf{v}|^{2}(s,\cdot)-|\mathbf{u}_{0}-\mathbf{v}(0,\cdot)|^{2}\Big)\dif x.  
\end{align*}
\end{itemize}
Combining the previous items, we thus arrive at 
\begin{align}\label{eq:debramorgan}
\begin{split}
\int_{0}^{s}\int_{\Omega}(\partial_{t}\mathbf{v})\cdot(\mathbf{v}-\mathbf{u}_{\varepsilon}^{\zeta})\dif x\dif t  & + \int_{0}^{s}(\mathscr{F}_{\varepsilon})_{0}^{*}[\mathbf{v}(t,\cdot);\Omega]\dif t -  \int_{0}^{s}\int_{\Omega}\mathbf{f}\cdot\mathbf{v}\dif x \dif t\\ 
& \geq \int_{0}^{s}\mathscr{F}_{0}^{*}[\mathbf{u}_{\varepsilon}^{\zeta}(t,\cdot);\Omega] \dif t -  \int_{0}^{s}\int_{\Omega}\mathbf{f}\cdot\mathbf{u}_{\varepsilon}^{\zeta}\dif x \dif t
\\ & + \int_{0}^{s}\int_{\Omega}\tocean(\uu_{\varepsilon}^{\zeta})\cdot(\vv-\uu_{\varepsilon}^{\zeta})\dif x \dif t \\  & + \frac{1}{2}\int_{\Omega}\Big(|\mathbf{u}_{\varepsilon}^{\zeta}-\mathbf{v}|^{2}(s,\cdot) - |\mathbf{u}_{0}^{\zeta}-\mathbf{v}(0,\cdot)|^{2}\Big) \dif x \\ 
 & \Longleftrightarrow: \mathrm{I}_{17} + \mathrm{I}_{18} + \mathrm{I}_{19} \geq \mathrm{I}_{20} + \mathrm{I}_{21} + \mathrm{I}_{22} + \mathrm{I}_{23}
 \end{split}
\end{align}
for all $\mathbf{v}\in\lebe^{2}(0,T;\sobo^{1,2}(\Omega;\R^{2})\cap\sobo^{1,2}(0,T;\lebe^{2}(\Omega;\R^{2}))$. Based on \eqref{eq:debramorgan}, we now address the limit passage $\varepsilon\searrow 0$. To this end, we isolate the following elementary observation: 
\begin{lemma}\label{lem:FepsLimit}
Let $F\colon\RR\to\R$ satisfy  \emph{\ref{item:Fprop1}--\ref{item:Fprop3}} and let $\mathbf{v}\in\lebe_{\mathrm{w}^{*}}^{1}(0,T;\bd(\Omega))$. Then, for $\mathscr{L}^{1}$-a.e. $0<s<T$, there holds 
\begin{align}\label{eq:officerkruger}
\int_{0}^{s}(\mathscr{F}_{\varepsilon})_{0}^{*}[\mathbf{v}(t,\cdot)]\dif t \to \int_{0}^{s}\mathscr{F}_{0}^{*}[\mathbf{v}(t,\cdot)]\dif t,\qquad \varepsilon\searrow 0. 
\end{align}
\end{lemma}
\begin{proof}
Let $z\in\RR$. 
We note that $F_{\varepsilon}(z)\to F(z)$ as $\varepsilon\searrow 0$ and $|F_{\varepsilon}(z)|\leq c(1+|z|)$ for all $0<\varepsilon<1$ and some universal $c>0$. For $\vv\in\lebe_{\mathrm{w}^{*}}^{1}(0,T;\bd(\Omega))$ and  $\mathscr{L}^{1}$-a.e.  $0<t<T$, we have the Radon-Nikod\'{y}m  decomposition $\mathbb{T}\vv(t,\cdot)=\mathscr{T}\vv(t,\cdot)\mathscr{L}^{2}+\mathbb{T}^{s}\vv(t,\cdot)$, where  $\mathscr{T}\vv\in\lebe^{1}(0,T;\lebe^{1}(\Omega;\RR))$. In consequence, 
\begin{align}\label{eq:laguerta}
G_{\varepsilon}[\mathbf{v}](t) \coloneqq \left\vert \int_{\Omega}F_{\varepsilon}(\mathscr{T}\mathbf{v}(t,\cdot))\dif x - \int_{\Omega}F(\mathscr{T}\mathbf{v}(t,\cdot))\dif x\right\vert \leq c\Big(1+\int_{\Omega}|\mathscr{T}\mathbf{v}(t,\cdot)|\dif x\Big), 
\end{align}
whereby the right-hand side of \eqref{eq:laguerta} is an integrable majorant for every $G_{\varepsilon}[\mathbf{v}]$. On the other hand, for $\mathscr{L}^{1}$-a.e. $0<t<T$, $c(1+|\mathscr{T}\mathbf{v}(t,\cdot)|)$ is an integrable majorant for all $F_{\varepsilon}(\mathscr{T}\mathbf{v}(t,\cdot))$. By dominated convergence with respect to space, it follows that $G_{\varepsilon}[\mathbf{v}](t)\to 0$. Now, by dominated convergence with respect to time, it follows that 
\begin{align*}
\int_{0}^{T}G_{\varepsilon}[\mathbf{v}](t)\dif t \to 0,\qquad \varepsilon\searrow 0. 
\end{align*}
By Lemma \ref{lem:approximations}, $F_{\varepsilon}^{\infty}=F^{\infty}$, and so \eqref{eq:officerkruger} follows. The proof is complete. 
\end{proof}
Again, let $\vv\in\lebe_{\mathrm{w}^{*}}^{1}(0,T;\bd(\Omega))\cap\sobo^{1,2}(0,T;\lebe^{2}(\Omega;\R^{2}))$. By Corollary \ref{cor:compa1}\ref{eq:compa5}--\ref{eq:compa8}, there exists $\uu^{\zeta}\in\mathrm{L}_{\mathrm{w}^{*}}^{1}(0,T;\bd(\Omega))\cap\lebe^{\infty}(0,T;\lebe^{2}(\Omega;\R^{2}))$ such that the convergences stated therein hold. We now deal with the single terms in \eqref{eq:debramorgan} separately. 
\begin{itemize}
\item On \namedlabel{item:term17}{Term $\mathrm{I}_{17}$}. By Corollary \ref{cor:compa1}\ref{eq:compa6}, $\uu_{\varepsilon_{j}}^{\zeta}\rightharpoonup\uu^{\zeta}$ in $\lebe^{2}(\Omega_{T};\R^{2})$. Since $\partial_{t}\vv\in\lebe^{2}(\Omega_{T};\R^{2})$, we infer that  
\begin{align*}
\lim_{j\to\infty} \mathrm{II}_{17}=\lim_{j\to\infty}\int_{0}^{s}\int_{\Omega}(\partial_{t}\vv)\cdot(\vv-\uu_{\varepsilon_{j}}^{\zeta})\dif x \dif t = \int_{0}^{s}\int_{\Omega}(\partial_{t}\vv)\cdot(\vv-\uu^{\zeta})\dif x \dif t. 
\end{align*}
\item On \namedlabel{item:term18}{Term $\mathrm{I}_{18}$}. By Lemma \ref{lem:FepsLimit}, we have 
\begin{align*}
\lim_{j\to\infty}\mathrm{I}_{18} = \lim_{j\to\infty} \int_{0}^{s}(\mathscr{F}_{\varepsilon_{j}})_{0}^{*}[\vv(t,\cdot);\Omega]\dif t = \int_{0}^{s}\mathscr{F}_{0}^{*}[\vv(t,\cdot);\Omega]\dif t. 
\end{align*}
\item On \namedlabel{item:term19}{Term $\mathrm{I}_{19}$}. We leave term $\mathrm{I}_{19}$ untouched. 
\item On \namedlabel{item:term20}{Term $\mathrm{I}_{20}$}. For term $\mathrm{I}_{20}$, we first note that, by Corollary \ref{cor:compa1}\ref{eq:compa7}, $\uu_{\varepsilon_{j}}^{\zeta}(t,\cdot)\to\uu^{\zeta}(t,\cdot)$ in $\lebe^{1}(\Omega;\R^{2})$ for $\mathscr{L}^{1}$-a.e. $t\in(0,T)$. Therefore, by Fatou's lemma and Corollary \ref{cor:LSCplussbdryvalues}\ref{item:relaxdoit1}, 
\begin{align*}
\liminf_{j\to\infty}\int_{0}^{s}\mathscr{F}_{0}^{*}[\uu_{\varepsilon_{j}}^{\zeta}(t,\cdot);\Omega]\dif t & \geq \int_{0}^{s}\liminf_{j\to\infty}\mathscr{F}_{0}^{*}[\uu_{\varepsilon_{j}}^{\zeta}(t,\cdot);\Omega]\dif t \geq \int_{0}^{s}\mathscr{F}_{0}^{*}[\uu^{\zeta}(t,\cdot);\Omega]\dif t. 
\end{align*}
\item On \namedlabel{item:term21}{Term $\mathrm{I}_{21}$}. Since $\uu_{\varepsilon_{j}}^{\zeta}\rightharpoonup \uu^{\zeta}$ in $\lebe^{2}(\Omega_{T};\R^{2})$, it follows that 
\begin{align*}
\lim_{j\to\infty}\mathrm{I}_{21} = \int_{0}^{s}\int_{\Omega}\mathbf{f}\cdot\uu^{\zeta}\dif x\dif t.
\end{align*}
\item On \namedlabel{item:term22}{Term $\mathrm{I}_{22}$}. For term $\mathrm{I}_{22}$, we note that Corollary \eqref{cor:compa1}\ref{eq:compa5}--\ref{eq:compa6} provide us with the same convergences as required in the limit passage in \eqref{eq:Iransofaraway}ff.. Therefore, we arrive at 
\begin{align*}
\lim_{j\to\infty}\mathrm{I}_{22} = \int_{0}^{s}\int_{\Omega}\tocean(\uu^{\zeta})\cdot(\vv-\uu^{\zeta})\dif x \dif t. 
\end{align*}
\item On \namedlabel{item:term23}{Term $\mathrm{I}_{23}$}. For term $\mathrm{I}_{23}$, we note that Corollary \ref{cor:compa1}\ref{eq:compa7} implies that $\uu_{\varepsilon_{j}}^{\zeta}(s,\cdot)\to\uu^{\zeta}(s,\cdot)$ strongly in $\lebe^{1}(\Omega;\R^{2})$ for $\mathscr{L}^{1}$-a.e. $0<s<T$. Hence, Fatou's lemma gives us 
\begin{align*}
\liminf_{j\to\infty} \mathrm{I}_{23}  \geq \frac{1}{2}\int_{\Omega}\Big(|\uu^{\zeta}-\vv|^{2}(s,\cdot)-|\uu_{0}^{\zeta}-\vv(0,\cdot)|^{2}\Big)\dif x
\end{align*}
for $\mathscr{L}^{1}$-a.e. $s\in (0,T)$. 
\end{itemize}
As an outcome of the preceding convergence assertions, we end up with 
\begin{align}\label{eq:wave103}
\begin{split}
\int_{0}^{s}\int_{\Omega}(\partial_{t}\mathbf{v})\cdot(\mathbf{v}-\mathbf{u}^{\zeta})\dif x\dif t  & + \int_{0}^{s}\mathscr{F}_{0}^{*}[\mathbf{v}(t,\cdot);\Omega]\dif t -  \int_{0}^{s}\int_{\Omega}\mathbf{f}\cdot\mathbf{v}\dif x \dif t\\ 
& \geq \int_{0}^{s}\mathscr{F}_{0}^{*}[\mathbf{u}^{\zeta}(t,\cdot);\Omega] \dif t -  \int_{0}^{s}\int_{\Omega}\mathbf{f}\cdot\mathbf{u}^{\zeta}\dif x \dif t
\\ & + \int_{0}^{s}\int_{\Omega}\tocean(\uu^{\zeta})\cdot(\vv-\uu^{\zeta})\dif x \dif t \\  & + \frac{1}{2}\int_{\Omega}\Big(|\mathbf{u}^{\zeta}-\mathbf{v}|^{2}(s,\cdot) - |\mathbf{u}_{0}^{\zeta}-\mathbf{v}(0,\cdot)|^{2}\Big) \dif x \\ 
 & \Longleftrightarrow: \mathrm{I}_{24} + \mathrm{I}_{25} + \mathrm{I}_{26} \geq \mathrm{I}_{27} + \mathrm{I}_{28} + \mathrm{I}_{29} + \mathrm{I}_{30}.
 \end{split}
\end{align}
Now we pick the map $\uu\in\lebe_{\mathrm{w}^{*}}^{1}(0,T;\bd(\Omega))\cap\lebe^{\infty}(0,T;\lebe^{2}(\Omega;\R^{2}))$ and the sequence $(\uu^{\zeta_{j}})$ from Corollary \ref{cor:compa1}\ref{eq:compa9}--\ref{eq:compa12}. Based on the latter, we may argue as for \ref{item:term17}, \ref{item:term20}, \ref{item:term21}, \ref{item:term22} to accomplish the requisite limit passage in terms $\mathrm{I}_{24},\mathrm{I}_{27},\mathrm{I}_{28},\mathrm{I}_{29}$. For term $\mathrm{I}_{30}$, we first note that $\uu^{\zeta_{j}}(s,\cdot)\to\uu(s,\cdot)$ strongly in $\lebe^{1}(\Omega;\R^{2})$ for $\mathscr{L}^{1}$-a.e. $0<s<T$. Hence, we may deal with the first term in the integrand of $\mathrm{I}_{30}$ as in the estimation of \ref{item:term23}. Moreover, by the definition of $\uu_{0}^{\zeta}$, see \eqref{eq:initialvalueapproximation}, we have $\uu_{0}^{\zeta_{j}}\to\uu_{0}$ strongly in $\lebe^{2}(\Omega;\R^{2})$. Hence, 
\begin{align*}
\liminf_{j\to\infty}\mathrm{I}_{30} \geq \frac{1}{2}\int_{\Omega}\Big(|\mathbf{u}-\mathbf{v}|^{2}(s,\cdot) - |\mathbf{u}_{0}^{\zeta}-\mathbf{v}(0,\cdot)|^{2}\Big) \dif x 
\end{align*}
for $\mathscr{L}^{1}$-a.e. $0<s<T$. In conclusion, for $\mathscr{L}^{1}$-a.e. $0<s<T$, we have 
\begin{align}
\int_{0}^{s}\int_{\Omega}(\partial_{t}\mathbf{v})\cdot(\mathbf{v}-\mathbf{u})\dif x\dif t  & + \int_{0}^{s}\mathscr{F}_{0}^{*}[\mathbf{v}(t,\cdot);\Omega]\dif t - \int_{0}^{s}\int_{\Omega}\mathbf{f}\cdot\mathbf{v}\dif x \dif t \notag \\ 
& \geq \int_{0}^{s}\mathscr{F}_{0}^{*}[\mathbf{u}(t,\cdot);\Omega] \dif t - \int_{0}^{s}\int_{\Omega}\mathbf{f}\cdot\mathbf{u}\dif x \dif t \label{eq:varfinal}
\\ & + \int_{0}^{s}\int_{\Omega}\tocean(\uu)\cdot(\vv-\uu)\dif x \dif t \notag\\  & + \frac{1}{2}\int_{\Omega}\Big(|\mathbf{u}-\mathbf{v}|^{2}(s,\cdot) - |\mathbf{u}_{0}-\mathbf{v}(0,\cdot)|^{2}\Big) \dif x \notag
\end{align}
for all $\vv\in\lebe_{\mathrm{w}^{*}}^{1}(0,T;\bd(\Omega))\cap\sobo^{1,2}(0,T;\lebe^{2}(\Omega;\R^{2}))$. This establishes that $\uu$ is a variational solution of Hibler's momentum balance equation in the sense of Definition \ref{def:varsol1}\ref{item:Hiblermain1}. \hfill $\square$ \\ 

The attainment of the initial values shall be addressed in the following section. For now, we discuss the specific growth of the oceanic force terms as employed in \eqref{eq:oceanicmain}, \eqref{eq:oceanicassump}, and including $x$-dependent integrands.
\begin{remark}\label{rem:criticalcutoff}
The specific growth behaviour of the function $\widetilde{\eta}$ in \eqref{eq:oceanicmain} and \eqref{eq:oceanicassump} with $0<\gamma<1$ enters the proof at several stages. For instance, to get access to the compactness provided by the Aubin-Lions lemma, the estimation of \ref{item:term8} requires $0<\gamma<1$. On the other hand, the evolution equation and the leading energy terms $F_{\delta,\varepsilon}(\mathbb{T}\uu_{\delta,\varepsilon}^{\zeta})$ only provide us with uniform bounds on $\mathbb{T}\uu_{\delta,\varepsilon}^{\zeta}$ in $\lebe^{1}(0,T;\lebe^{1}(\Omega;\RR))$. In particular, since $\bd(\Omega)\hookrightarrow\lebe^{2}(\Omega;\R^{2})$, the maximal spatial integrability entailed by the equation is $\lebe^{2}(\Omega;\R^{2})$. Therefore, a potential superlinear growth of $\widetilde{\eta}$ only leads to well-defined terms which are \emph{uniformly} treatable in the singular limit $\varepsilon\searrow 0$ if one can establish uniform higher (spatial)  integrability estimates on $\uu_{\delta,\varepsilon}^{\zeta}$ beyond $\lebe^{2}(\Omega;\R^{2})$. 

Such estimates, in turn, \emph{cannot stem} from higher differentiability assertions on $\uu$ (e.g., in Sobolev spaces). Since we deal with the singular limit, the limiting integrands typically do not share any ellipticity at all. Hence, such higher differentiability results are neither expected, nor would they be natural in view of incorporating plasticity effects. In consequence, one would need to establish higher spatial integrability directly on the level of the functions themselves. 

In this regard, methods to arrive at uniform spatial $\lebe^{p}$-integrability for $\uu_{\delta,\varepsilon}^{\zeta}$, $p>1$, fail in the setting considered here. First,  in the stationary, vectorial $\bv$-context, a radial structure of the energy integrands $F$ can be used to establish maximum principles or local $\lebe^{\infty}$-bounds via Moser iterations; see, e.g., \cite[Section 3]{Beckschmidt13} and also \cite{HZ} in the parabolic context. It is, however, here where the complicated structure of $\mathbb{T}$ (or even the symmetric gradient) fully destroys the underlying comparison estimates. Aiming for lower order integrability, a natural alternative are lower order Gehring-type improvements; see the next remark.  
\end{remark}
\begin{remark}\label{rem:gehringfail} 
In general, for problems involving energies of linear growth, Gehring-type improvements on the level of derivatives are impossible. More precisely, even when a Caccioppoli inequality is available -- thereby estimating first order quantities against zeroth order quantities -- one would need a sublinear Sobolev-Poincar\'{e} inequality to obtain a self-improvement of gradient integrability. The latter is ruled out by \cite{BK} and indeed, Gehring-type improvements even fail for full gradient problems; see  \cite[Section 1]{Gm1} for more detail.

As discussed in \cite[Section 1]{Gm1}, \emph{lower order} integrability improvements via Gehring are possible in the elliptic context. Here, one firstly employs the usual Sobolev-Poincar\'{e} inequality and \emph{subsequently} the Caccioppoli inequality. By use of boundary regularity techniques, this yields -- even for irregular problems -- spatial $\lebe^{2+\varepsilon}(\Omega;\R^{2})$-integrability. However, in the evolutionary context as considered here, such improvements are precisely ruled out in $n=2$ dimensions for $p=1$. This becomes visible when attempting to generalize the up-to-the-boundary regularity results of, e.g., B\"{o}gelein \& Parviainen \cite{Boegelein} on the level of \emph{functions} (and not first order differential expressions) to the case $p=1$. In essence, the underlying obstruction is that the problems considered here only provide us with a priori $\lebe^{\infty}(0,T;\lebe^{2}(\Omega;\R^{2}))$-bounds on the solutions. Whereas this can yet be made work for problems of superlinear growth, $p>1$, it is precisely $p=1$ in $n=2$ dimensions where such approaches break down. 
\end{remark}

\begin{remark}\label{rem:xdependence}
Appealing to the $x$-dependent variant of Reshetnyak's lower semicontinuity theorem \cite{Reshetnyak}, it is in principle possible to include integrands which correspond to spatially variable pressures. For instance, the particular integrand $F(x,z)=\frac{P(x)}{2}|z|$ with some given continuous pressure $P$ that is bounded from above and away from zero leads to a similar result as Theorem \ref{thm:main}. Since this is not the core aspect of the present paper, we refrain from stating such a version explicitly.
\end{remark}

\subsection{$\hold([0,T];\lebe^{2}(\Omega;\R^{2}))$-regularity and attainment of initial values}\label{sec:initial}
We now establish that the variational solution $\uu$ as constructed in Section \ref{sec:existence} attains the initial values $\uu_{0}$. To this end, we begin with the following proposition: 
\begin{proposition}\label{prop:timebounds}
For $\delta,\varepsilon,\zeta>0$ as in \eqref{eq:allchoose}, let $\uu_{\delta,\varepsilon}^{\zeta}$ be as in the previous subsection, see \eqref{eq:motiv1}ff.. Then there exists a constant $c=c(\gamma,T,F,\uu_{0},\mathbf{f})>0$ such that 
\begin{align}\label{eq:uniftimederivativebounds}
\sup\big\{\|\partial_{t}\uu_{\delta,\varepsilon}^{\zeta}\|_{\lebe^{2}(\Omega_{T})}\colon\;\delta,\varepsilon,\zeta>0\;\text{satisfy \eqref{eq:allchoose}}\big\}<\infty. 
\end{align}
\end{proposition}
\begin{proof}
We shall use a particular test function in \eqref{eq:approxineq}; this specific choice is due to Brezis \cite[\S II, p. 69--70]{Brezis72}. For future reference, we note that the approximate solutions of \eqref{eq:motiv1} satisfy    $\uu_{\delta,\varepsilon}^{\zeta}\in\sobo^{1,2}(0,T;\lebe^{2}(\Omega;\R^{2}))\hookrightarrow X\coloneqq \hold([0,T];\lebe^{2}(\Omega;\R^{2}))$. For $\lambda>0$, we put  
\begin{align}\label{eq:doakesvsdexter}
\uu_{\delta,\varepsilon}^{\zeta,\lambda}(t,\cdot) \coloneqq \mathrm{e}^{-\frac{t}{\lambda}}\uu_{0}^{\zeta} + \frac{1}{\lambda}\int_{0}^{t}\mathrm{e}^{\frac{s-t}{\lambda}}\uu_{\delta,\varepsilon}^{\zeta}(s,\cdot)\dif s,\qquad 0\leq t\leq T. 
\end{align}
Given $0< t<T$, we define a probability measure on $[0,T]$ by 
\begin{align}\label{eq:mudef}
\mu_{t} \coloneqq \mathrm{e}^{-\frac{t}{\lambda}}\delta_{\{s=0\}} + \frac{1}{\lambda}\mathrm{e}^{\frac{s-t}{\lambda}}\mathscr{L}^{1}\mres (0,t). 
\end{align}
Clearly, $\mu_{t}$ is non-negative, and we compute 
\begin{align*}
\mu_{t}([0,T))=\mathrm{e}^{-\frac{t}{\lambda}} + \frac{1}{\lambda}\int_{0}^{t}\mathrm{e}^{\frac{s-t}{\lambda}}\dif s = \mathrm{e}^{-\frac{t}{\lambda}}+\frac{1}{\lambda}\mathrm{e}^{-\frac{t}{\lambda}}\big[\lambda\mathrm{e}^{\frac{s}{\lambda}}\big]_{s=0}^{s=t} =1. 
\end{align*}
Hence, 
\begin{align}\label{eq:zuck1}
\int_{0}^{T}\uu_{\delta,\varepsilon}^{\zeta}(s,\cdot)\dif\mu_{t} \stackrel{\eqref{eq:mudef}}{=} \mathrm{e}^{-\frac{t}{\lambda}}\uu_{0}^{\zeta} + \frac{1}{\lambda}\int_{0}^{t}\mathrm{e}^{\frac{s-t}{\lambda}}\uu_{\delta,\varepsilon}^{\zeta}(s,\cdot)\dif s \stackrel{\eqref{eq:doakesvsdexter}}{=}  \uu_{\delta,\varepsilon}^{\zeta,\lambda}(t,\cdot).
\end{align}
By an integration by parts and recalling that $\uu_{\delta,\varepsilon}^{\zeta}(0,\cdot)=\uu_{0}^{\zeta}$, we obtain 
\begin{align*}
\partial_{t}\uu_{\delta,\varepsilon}^{\zeta,\lambda} & = -\frac{1}{\lambda}\mathrm{e}^{-\frac{t}{\lambda}}\uu_{0}^{\zeta}- \frac{1}{\lambda^{2}}\mathrm{e}^{-\frac{t}{\lambda}}\int_{0}^{t}\mathrm{e}^{\frac{s}{\lambda}}\uu_{\delta,\varepsilon}^{\zeta}(s,\cdot)\dif s + \frac{1}{\lambda}\uu_{\delta,\varepsilon}^{\zeta}(t,\cdot) \\ 
& = -\frac{1}{\lambda}\mathrm{e}^{-\frac{t}{\lambda}}\uu_{0}^{\zeta} - \frac{1}{\lambda^{2}}\mathrm{e}^{-\frac{t}{\lambda}}\Big(\Big[\lambda\mathrm{e}^{\frac{s}{\lambda}}\uu_{\delta,\varepsilon}^{\zeta}(s,\cdot)\Big]_{s=0}^{s=t}-\lambda\int_{0}^{t}\mathrm{e}^{\frac{s}{\lambda}}\dot{\uu}_{\delta,\varepsilon}^{\zeta}(s,\cdot)\dif s \Big) + \frac{1}{\lambda}\uu_{\delta,\varepsilon}^{\zeta}(t,\cdot) \\ 
& = \frac{1}{\lambda}\int_{0}^{t}\mathrm{e}^{\frac{s-t}{\lambda}}\dot{\uu}_{\delta,\varepsilon}^{\zeta}(s)\dif s = \int_{\R}\Big(\frac{1}{\lambda}\mathbbm{1}_{(-\infty,0)}\Big(\frac{s-t}{\lambda}\Big)\mathrm{e}^{\frac{s-t}{\lambda}}\Big)(\mathbbm{1}_{(0,T)}(s)\dot{\uu}_{\delta,\varepsilon}^{\zeta}(s))\dif s \\ 
& \eqqcolon \int_{\R}g_{\lambda}(t-s)\mathbf{w}(s)\dif s,  
\end{align*}
where  
\begin{align*}
g_{\lambda}(t)\coloneqq \frac{1}{\lambda}\mathbbm{1}_{(-\infty,0)}\Big(-\frac{t}{\lambda}\Big)\mathrm{e}^{-\frac{t}{\lambda}} = \frac{1}{\lambda}g\Big(\frac{t}{\lambda}\Big)\;\;\;\;\text{with}\;\;\;\;g(t) \coloneqq  \mathbbm{1}_{(0,\infty)}(t)\mathrm{e}^{-t} 
\end{align*}
and $
\mathbf{w}(t,\cdot)\coloneqq \mathbbm{1}_{(0,T)}(t)\dot{\uu}_{\delta,\varepsilon}^{\zeta}(t,\cdot)$. In particular, we have 
\begin{align}\label{eq:normalisationofg}
\int_{\R}g(t)\dif t =1. 
\end{align}
As a consequence of \eqref{eq:normalisationofg}, {for $\lambda\searrow0$ } we have $\partial_{t}\uu_{\delta,\varepsilon}^{\zeta,\lambda}(t,\cdot)\to\partial_{t}\uu_{\delta,\varepsilon}^{\zeta}(t,\cdot)$ strongly in $\lebe^{2}(\Omega;\R^{2})$ for $\mathscr{L}^{1}$-a.e. $t\in(0,T)$
 {by \cite[Theorem~2.3.8]{HNVW16}.} Hence, by Fatou's lemma,
\begin{align}\label{eq:LSCtime}
\int_{\Omega_{T}}|\partial_{t}\uu_{\delta,\varepsilon}^{\zeta}|^{2}\dif\,(t,x) \leq \liminf_{\lambda\searrow 0}\int_{\Omega_{T}}|\partial_{t}\uu_{\delta,\varepsilon}^{\zeta,\lambda}|^{2}\dif\,(t,x). 
\end{align}
On the other hand, the function $\uu_{\delta,\varepsilon}^{\zeta}$ belongs to $\lebe^{2}(0,T;\sobo_{0}^{1,2}(\Omega;\R^{2}))\cap\sobo^{1,2}(0,T;\lebe^{2}(\Omega;\R^{2}))$, and so does $\uu_{\delta,\varepsilon}^{\zeta,\lambda}$. We let the purely spatial differential operator $\mathbb{T}$ act on \eqref{eq:zuck1} and apply $F_{\delta,\varepsilon}$ to both sides of the resulting equation. We subsequently integrate with respect to $x\in\Omega$; since $F_{\delta,\varepsilon}$ is convex, we thus obtain for $\mathscr{L}^{1}$-a.e. $0<t<T$ by Jensen's inequality applied to the probability measure $\mu_{t}$:  
\begin{align}\label{eq:zeitum}
\begin{split}
\int_{\Omega}F(\mathbb{T}\uu_{\delta,\varepsilon}^{\zeta,\lambda}(t,\cdot))\dif x & \stackrel{\eqref{eq:zuck1}}{\leq} \int_{0}^{T}\int_{\Omega}F(\mathbb{T}\uu_{\delta,\varepsilon}^{\zeta})\dif x \dif\mu_{t} \\ 
& \stackrel{\eqref{eq:mudef}}{=} \mathrm{e}^{-\frac{t}{\lambda}}\int_{\Omega}F(\mathbb{T}\mathbf{u}_{0}^{\zeta})\dif x + \frac{1}{\lambda}\int_{0}^{t}\mathrm{e}^{\frac{s-t}{\lambda}}\int_{\Omega}F(\mathbb{T}\uu_{\delta,\varepsilon}^{\zeta}(s,\cdot))\dif x\dif s. 
\end{split}
\end{align}
By its very definition, $\uu_{\delta,\varepsilon}^{\zeta,\lambda}$ satisfies  
\begin{align}\label{eq:zeitum1}
\begin{split}
&\uu_{\delta,\varepsilon}^{\zeta,\lambda}(0,\cdot) =  \uu_{0}^{\zeta}\;\;\;\text{and} \\ 
&\uu_{\delta,\varepsilon}^{\zeta,\lambda}+\lambda{\dot\uu}_{\delta,\varepsilon}^{\zeta,\lambda} = \uu_{\delta,\varepsilon}^{\zeta},\;\;\;\text{hence}\;\;\; \lambda{\dot\uu}_{\delta,\varepsilon}^{\zeta,\lambda}=\uu_{\delta,\varepsilon}^{\zeta}-\uu_{\delta,\varepsilon}^{\zeta,\lambda}.
\end{split}
\end{align}
We now use the test function $\vv\coloneqq \uu_{\delta,\varepsilon}^{\zeta,\lambda}$ in \eqref{eq:approxineq}. Noting that $\vv(0,\cdot)=\uu_{0}^{\zeta}$, we conclude 
\begin{align*}
\lambda\int_{0}^{T}\int_{\Omega}|{\dot\uu}_{\delta,\varepsilon}^{\zeta,\lambda}|^{2}\dif x \dif t & + \int_{0}^{T}\int_{\Omega}F_{\delta,\varepsilon}(\mathbb{T}\uu_{\delta,\varepsilon}^{\zeta})\dif x \dif t \leq \int_{0}^{T}\int_{\Omega}F_{\delta,\varepsilon}(\mathbb{T}\uu_{\delta,\varepsilon}^{\zeta,\lambda})\dif x \dif t \\ 
& \!\!\!\!\!\!\!\!\!\!\!\!\!\!\!\!\!\!\!\!\!\!\!\! + \lambda  \int_{0}^{T}\int_{\Omega}|\mathbf{f}|\,| \dot{\uu}_{\delta,\varepsilon}^{\zeta,\lambda}|\dif x\dif t + \lambda \int_{0}^{T}\int_{\Omega}|\tocean(\uu_{\delta,\varepsilon}^{\zeta})|\,|\dot{\uu}_{\delta,\varepsilon}^{\zeta,\lambda}|\dif x\dif t\\ 
 & \!\!\!\!\!\!\!\!\!\!\!\!\!\!\!\!\!\!\!\!\!\!\!\!\!\!\!\!  \stackrel{\eqref{eq:zeitum}}{\leq} \int_{0}^{T}\mathrm{e}^{-\frac{s}{\lambda}}\dif s \int_{\Omega}F_{\delta,\varepsilon}(\mathbb{T}\uu_{0}^{\zeta})\dif x + \frac{1}{\lambda}\int_{0}^{T}\int_{0}^{t}\mathrm{e}^{\frac{s-t}{\lambda}}\int_{\Omega}F_{\delta,\varepsilon}(\mathbb{T}\uu_{\delta,\varepsilon}^{\zeta}(s,\cdot))\dif x \dif s \dif t \\ 
& \!\!\!\!\!\!\!\!\!\!\!\!\!\!\!\!\!\!\!\!\!\!\!\! + \lambda  \int_{0}^{T}\int_{\Omega}|\mathbf{f}|^{2}\dif x\dif t + \frac{\lambda}{4} \int_{0}^{T}\int_{\Omega}|\dot{\uu}_{\delta,\varepsilon}^{\zeta,\lambda}|^{2}\dif x \dif t \\ & + \lambda \int_{0}^{T}\int_{\Omega}|\tocean(\uu_{\delta,\varepsilon}^{\zeta})|^{2}\dif x\dif t + \frac{\lambda}{4}\int_{0}^{T}\int_{\Omega}|\dot{\uu}_{\delta,\varepsilon}^{\zeta,\lambda}|^{2}\dif x\dif t\\ 
 & \!\!\!\!\!\!\!\!\!\!\!\!\!\!\!\!\!\!\!\!\!\!\!\! \leq \lambda\underbrace{(1-\mathrm{e}^{-\frac{T}{\lambda}})}_{\leq 1} \int_{\Omega}F_{\delta,\varepsilon}(\mathbb{T}\uu_{0}^{\zeta})\dif x + \int_{0}^{T}\int_{\Omega}F_{\delta,\varepsilon}(\mathbb{T}\uu_{\delta,\varepsilon}^{\zeta})\dif x\dif t \\ 
 & \!\!\!\!\!\!\!\!\!\!\!\!\!\!\!\!\!\!\!\!\!\!\!\! + \lambda  \int_{0}^{T}\int_{\Omega}|\mathbf{f}|^{2}\dif x\dif t + \lambda \int_{0}^{T}\int_{\Omega}|\tocean(\uu_{\delta,\varepsilon}^{\zeta})|^{2}\dif x\dif t \\ & \!\!\!\!\!\!\!\!\!\!\!\!\!\!\!\!\!\!\!\!\!\!\!\! + \frac{\lambda}{2}\int_{0}^{T}\int_{\Omega}|\dot{\uu}_{\delta,\varepsilon}^{\zeta,\lambda}|^{2}\dif x\dif t, 
\end{align*}
and we organize the resulting overall inequality as 
\begin{align*}
\mathrm{I}_{31} + \mathrm{I}_{32} \leq \mathrm{I}_{33} + \mathrm{I}_{34} + \mathrm{I}_{35} + \mathrm{I}_{36} + \mathrm{I}_{37}. 
\end{align*}
Since $\mathrm{I}_{32}=\mathrm{I}_{34}$, we may cancel this term on both sides of the inequality. Moreover, we may  absorb $\mathrm{I}_{37}$ into $\mathrm{I}_{31}$. Recalling  \eqref{eq:atmosphericmain}, \eqref{eq:oceanicmain}--\eqref{eq:oceanicHoelder} and  \eqref{eq:Fepsdelta}ff., dividing the resulting inequality by $\lambda$ gives us 
\begin{align*}
\int_{0}^{T}\int_{\Omega}|{\dot\uu}_{\delta,\varepsilon}^{\zeta,\lambda}|^{2}\dif x \dif t & \leq \int_{\Omega}F_{\delta,\varepsilon}(\mathbb{T}\uu_{0}^{\zeta})\dif x + \|\mathbf{f}\|_{\lebe^{2}(0,T;\lebe^{2}(\Omega))}^{2} +  c\, \|\mathbf{U}_{\mathrm{ocean}}-\uu_{\delta,\varepsilon}^{\zeta}\|_{\lebe^{2}(0,T;\lebe^{2}(\Omega))}^{2} \\ 
& \!\!\!\!\!\!\!\!\!\!\!\stackrel{\text{Lem. \ref{lem:usefulapriori}\ref{eq:unifobd1}}}{\leq}c\int_{\Omega}|\mathbb{T}\uu_{0}^{\zeta}|\dif x + \frac{\delta}{2}\int_{\Omega}|\mathbb{T}\uu_{0}^{\zeta}|^{2}\dif x + c \\ 
& \!\!\!\!\!  \stackrel{\eqref{eq:initialvalueapproximation}_{3}}{\leq} c\, |\mathbb{T}\uu_{0}|(\Omega) + c\frac{\delta}{\zeta^{2}}\|\uu_{0}\|_{\bd(\Omega)}^{2} + c \\ 
& \!\!\!\!\stackrel{\eqref{eq:allchoose}}{\leq}  c\, |\mathbb{T}\uu_{0}|(\Omega) + c\|\uu_{0}\|_{\bd(\Omega)}^{2} + c, 
\end{align*}
where $c>0$ is independent of $\delta,\varepsilon,\zeta$ and $\lambda$. This implies the claim. 
\end{proof}
We now conclude the proof of Theorem \ref{thm:main} by showing the attainment of initial values in the sense of Definition \ref{def:varsol1}\ref{item:Hiblermain2}. Here, we establish an even stronger result as follows. 
\begin{corollary}\label{cor:timeregularity}
Let $\mathbf{u}\in\lebe_{\mathrm{w}^{*}}^{1}(0,T;\bd(\Omega))\cap\lebe^{\infty}(0,T;\lebe^{2}(\Omega;\R^{2}))$ be as in the previous subsection, see Corollary \ref{cor:compa1}. Then we have that 
\begin{align}\label{eq:}
\uu\in\sobo^{1,2}(0,T;\lebe^{2}(\Omega;\R^{2}))\;\;\;\text{and so}\;\;\;\uu\in\hold([0,T];\lebe^{2}(\Omega;\R^{2})),  
\end{align}
whereby $\uu(0,\cdot)$ is a well-defined $\lebe^{2}(\Omega;\R^{2})$-map. Finally, we have 
\begin{align}\label{eq:L2initial}
\uu(0,\cdot)=\uu_{0}\qquad\text{$\mathscr{L}^{2}$-a.e. in $\Omega$}.
\end{align}
In particular, the \emph{attainment of the initial values holds in the sense of Definition \ref{def:varsol1}\ref{item:Hiblermain2}}. 
\end{corollary}
\begin{proof} 
By Corollary \ref{cor:compa1}, we have  
\begin{align}\label{eq:timeweak1}
\uu_{\delta_{j},\varepsilon}^{\zeta} \rightharpoonup \uu_{\varepsilon}^{\zeta},\;\;\;\uu_{\varepsilon_{j}}^{\zeta} \rightharpoonup \uu^{\zeta}\;\;\;\text{and}\;\;\;\uu^{\zeta_{j}} \rightharpoonup \uu\qquad\text{in}\;\lebe^{2}(0,T;\lebe^{2}(\Omega;\R^{2})).
\end{align}
By Proposition \ref{prop:timebounds} and the lower semicontinuity of norms with respect to weak convergence, $(\partial_{t}\uu_{\delta,\varepsilon}^{\zeta})$ and so $(\partial_{t}\uu_{\varepsilon}^{\zeta})$ and $(\partial_{t}\uu^{\zeta})$ are uniformly bounded in $\lebe^{2}(\Omega_{T};\R^{2})$ with respect to the underlying parameters $\delta,\varepsilon$ and $\zeta$. In particular, for all admissible $\varepsilon,\zeta>0$ (see \eqref{eq:allchoose}), there exists $\vv_{\varepsilon}^{\zeta}\in\lebe^{2}(\Omega_{T};\R^{2})$ such that 
\begin{align}\label{eq:timeweak2}
\partial_{t}\uu_{\delta_{j},\varepsilon}^{\zeta}\rightharpoonup\vv_{\varepsilon}^{\zeta}\qquad\text{in}\;\lebe^{2}(\Omega_{T};\R^{2}). 
\end{align}
This implies that $\uu_{\varepsilon}^{\zeta}\in \sobo^{1,2}(0,T;\lebe^{2}(\Omega;\R^{2}))$ together with $\partial_{t}\uu_{\varepsilon}^{\zeta}=\vv_{\varepsilon}^{\zeta}$. Indeed, let $\bm{\varphi}\in\hold_{c}^{\infty}(0,T;\lebe^{2}(\Omega;\R^{2}))$. Then 
\begin{align*}
\int_{0}^{T}\langle \vv_{\varepsilon}^{\zeta},\bm{\varphi}\rangle_{\lebe^{2}(\Omega)}\dif t & = \lim_{j\to\infty} \int_{0}^{T}\langle\partial_{t}\uu_{\delta_{j},\varepsilon}^{\zeta},\bm{\varphi}\rangle_{\lebe^{2}(\Omega)}\dif t = -\lim_{j\to\infty}\int_{0}^{T}\langle\uu_{\delta_{j},\varepsilon}^{\zeta},\partial_{t}\bm{\varphi}\rangle_{\lebe^{2}(\Omega)}\dif t \\ 
& \!\!\!\! \stackrel{\eqref{eq:timeweak1}}{=} -\int_{0}^{T}\langle\uu_{\varepsilon}^{\zeta},\partial_{t}\bm{\varphi}\rangle_{\lebe^{2}(\Omega)}\dif t, 
\end{align*}
from where we deduce that $\uu_{\varepsilon}^{\zeta}\in\sobo^{1,2}(0,T;\lebe^{2}(\Omega;\R^{2}))\hookrightarrow \hold([0,T];\lebe^{2}(\Omega;\R^{2}))$ and $\partial_{t}\uu_{\varepsilon}^{\zeta}=\vv_{\varepsilon}^{\zeta}$. By \eqref{eq:timeweak1}, we also have the corresponding convergences for $(\partial_{t}\uu_{\varepsilon_{j}}^{\zeta})$ and $(\partial_{t}\uu^{\zeta_{j}})$ with the natural modifications. In particular, we infer that $\uu\in\sobo^{1,2}(0,T;\lebe^{2}(\Omega;\R^{2}))$. Since this space embeds into $\hold([0,T];\lebe^{2}(\Omega;\R^{2}))$, $\uu(0,\cdot)$ is unambiguously defined in $\lebe^{2}(\Omega;\R^{2})$. 

Now let $\bm{\psi}\in\hold_{c}^{\infty}(\Omega;\R^{2})$ be arbitrary and choose $\bm{\varphi}\in\hold^{\infty}([0,T]\times\Omega)$ such that $\bm{\varphi}(0,\cdot)=\bm{\psi}$ and $\bm{\varphi}(T,\cdot)=0$. Using that $\uu_{\delta_{j},\varepsilon}^{\zeta}(0,\cdot)=\uu_{0}^{\zeta}$ for all $\varepsilon>0$ and all $j\in\mathbb{N}$, we compute 
\begin{align}\label{eq:mrmonk}
\begin{split}
- \langle \uu_{0}^{\zeta},\bm{\psi}\rangle_{\lebe^{2}(\Omega)} & = \int_{0}^{T}\langle \mathbf{u}_{\delta_{j},\varepsilon}^{\zeta},\dot{\bm{\varphi}}\rangle_{\lebe^{2}(\Omega)} \dif t + \int_{0}^{T}\langle\dot{\uu}_{\delta_{j},\varepsilon}^{\zeta},\bm{\varphi}\rangle_{\lebe^{2}(\Omega)}\dif t \\ 
& \!\!\!\!\!\!\!\!\!\!\stackrel{\eqref{eq:timeweak1},\eqref{eq:timeweak2}}{\longrightarrow}  \int_{0}^{T}\langle\uu_{\varepsilon}^{\zeta},\dot{\bm{\varphi}}\rangle_{\lebe^{2}(\Omega)}\dif t + \int_{0}^{T}\langle\dot{\uu}_{\varepsilon}^{\zeta},\bm{\varphi}\rangle_{\lebe^{2}(\Omega)}\dif t = - \langle\uu_{\varepsilon}^{\zeta}(0,\cdot),\bm{\psi}\rangle_{\lebe^{2}(\Omega)}
\end{split}
\end{align}
as $j\to\infty$. By arbitrariness of $\bm{\psi}$, we conclude that $\uu_{\varepsilon}^{\zeta}(0,\cdot)=\uu_{0}^{\zeta}$ $\mathscr{L}^{2}$-a.e. in $\Omega$ for all admissible $\varepsilon,\zeta>0$. Using the second weak convergence displayed in \eqref{eq:timeweak1}, we similarly conclude that $\uu^{\zeta}(0,\cdot)=\uu_{0}^{\zeta}$. Finally, we have $\uu_{0}^{\zeta}\to\uu_{0}$ strongly in $\lebe^{2}(\Omega;\R^{2})$ as $\zeta\searrow 0$. In consequence, we may conclude as above to arrive at 
\begin{align*}
-\langle\uu_{0},\psi\rangle_{\lebe^{2}(\Omega)} & = - \lim_{j\to\infty} \langle \uu_{0}^{\zeta_{j}},\bm{\psi}\rangle_{\lebe^{2}(\Omega)}  = \lim_{j\to\infty}\Big(\int_{0}^{T}\langle \uu^{\zeta_{j}},\dot{\bm{\varphi}}\rangle_{\lebe^{2}(\Omega)}\dif t + \int_{0}^{T}\langle\dot{\uu}^{\zeta_{j}},\bm{\varphi}\rangle_{\lebe^{2}(\Omega)}\dif t\Big) \\ 
& = \int_{0}^{T}\langle\uu,\dot{\bm{\varphi}}\rangle_{\lebe^{2}(\Omega)}\dif t + \int_{0}^{T}\langle\dot{\uu},\bm{\varphi}\rangle_{\lebe^{2}(\Omega)}\dif t = - \langle \uu(0,\cdot),\bm{\psi}\rangle_{\lebe^{2}(\Omega)}, 
\end{align*}
since we already know at this stage that $\uu\in\sobo^{1,2}(0,T;\lebe^{2}(\Omega;\R^{2}))$. Therefore, $\uu_{0}=\uu(0,\cdot)$ $\mathscr{L}^{2}$-a.e. in $\Omega$. Because of $\uu\in\hold([0,T];\lebe^{2}(\Omega;\R^{2}))$, this particularly implies the attainment of initial values in the sense of Definition \ref{def:varsol1}\ref{item:Hiblermain2}. The proof is complete. 
\end{proof} 
\subsection{The evolutionary system for measures}
Based on the existence of variational solutions to the momentum balance equation, we now derive the associated evolution \emph{equation}. This, in turn, is strongly inspired by the stationary full gradient scenarios considered by Anzellotti \cite{Anz1}. More precisely, we have the following theorem: 
\begin{theorem}[Relaxed evolution equation]\label{thm:relaxedevoleq}
In the situation of Definition \ref{def:varsol1} and Theorem \ref{thm:main}, suppose that $F,F^{\infty}$ are differentiable in $\RR\setminus\{0\}$, and let $\uu\in\mathrm{L}_{\mathrm{w}^{*}}^{1}(0,T;\bd(\Omega))\cap\sobo^{1,2}(0,T;\lebe^{2}(\Omega;\R^{2}))$ be a variational solution of the momentum balance equation. We define the space $\mathrm{Adm}(\mathbf{u})$ of  \emph{admissible test maps} as the collection of all $\bm{\varphi}\in\lebe_{\mathrm{w}^{*}}^{1}(0,T;\bd(\Omega))\cap\sobo^{1,2}(0,T;\lebe^{2}(\Omega;\R^{2}))$ such that the following hold for $\mathscr{L}^{1}$-a.e. $0<t<T$:  
\begin{enumerate}
\item\label{item:EL1} $|\mathbb{T}^{s}\bm{\varphi}(t,\cdot)|\ll |\mathbb{T}^{s}\uu(t,\cdot)|$, 
\item\label{item:EL2} $\mathscr{T}\bm{\varphi}(t,\cdot)=0$ $\mathscr{L}^{2}$-a.e. in $\{x\in\Omega\colon\; \mathscr{T}\uu(t,\cdot)=0\}$, and 
\item\label{item:EL3} $\mathrm{tr}_{\partial\Omega}(\bm{\varphi}(t,\cdot))=0$ $\mathscr{H}^{1}$-a.e. in $\{x\in\partial\Omega\colon\;\mathrm{tr}_{\partial\Omega}(\uu(t,x))=0\}$. 
\end{enumerate}
Then $\uu$ satisfies the \emph{relaxed evolution equation} 
\begin{align}
\int_{\Omega_{T}}(\partial_{t}\uu)\cdot\bm{\varphi}\dif\,(t,x) & + \int_{\Omega_{T}}F'(\mathscr{T}\uu)\cdot\mathbb{T}\bm{\varphi}\dif\,(t,x) \notag \\ & + \int_{0}^{T}\int_{\Omega}(F^{\infty})'\Big(\frac{\dif\mathbb{T}^{s}\uu(t,\cdot)}{\dif|\mathbb{T}^{s}\uu(t,\cdot)|}\Big)\cdot\frac{\dif\mathbb{T}^{s}\bm{\varphi}(t,\cdot)}{\dif|\mathbb{T}^{s}\bm{\varphi}(t,\cdot)|}\dif|\mathbb{T}^{s}\bm{\varphi}(t,\cdot)|\dif t \label{eq:hiblerPDE}\\ 
& - \int_{0}^{T}\int_{\partial\Omega}(F^{\infty})'(-\mathrm{tr}_{\partial\Omega}(\uu(t,\cdot)\otimes_{\mathbb{T}}\nu_{\partial\Omega}))\cdot\mathrm{tr}_{\partial\Omega}(\bm{\varphi}(t,\cdot))\dif\mathscr{H}^{1}\dif t \notag \\ 
& = \int_{\Omega_{T}}\tocean(\uu)\cdot\bm{\varphi}\dif\,(t,x) + \int_{\Omega_{T}}\mathbf{f}\cdot\bm{\varphi}\dif\,(t,x)\;\;\;\text{for all}\;\bm{\varphi}\in\mathrm{Adm}(\mathbf{u}).\notag
\end{align}
\end{theorem}
\begin{proof}
Let $\bm{\varphi}\in\mathrm{Adm}(\uu)$ and let $\theta>0$ be arbitrary. Then the functions $\vv_{\theta}^{\pm}\coloneqq \uu \pm\theta\bm{\varphi}$ are admissible in \eqref{eq:varsolmain}. Dividing the resulting inequality by $\theta$, we obtain 
\begin{align}
\pm\int_{0}^{s}\int_{\Omega}\partial_{t}(\mathbf{u}\pm\theta\bm{\varphi})\cdot\bm{\varphi}\dif x \dif t & + \int_{0}^{s}\frac{1}{\theta}(\mathscr{F}_{0}^{*}[(\uu\pm\theta\bm{\varphi})(t,\cdot);\Omega]-\mathscr{F}_{0}^{*}[\uu(t,\cdot);\Omega])\dif t \notag \\ & \mp   \int_{0}^{s}\int_{\Omega}\mathbf{f}\cdot\bm{\varphi}\dif x \dif t \notag \\ 
& \geq \pm \int_{0}^{s}\int_{\Omega}\tocean(\uu)\cdot\bm{\varphi}\dif x \dif t + \frac{\theta}{2}\int_{\Omega}\Big(|\bm{\varphi}(s)|^{2} - |\bm{\varphi}(0)|^{2}\Big)\dif x \label{eq:randalldisher}\\ 
& \Longleftrightarrow: \mathrm{I}_{38} + \mathrm{I}_{39} + \mathrm{I}_{40} \geq \mathrm{I}_{41}+\mathrm{I}_{42}\notag
\end{align}
for $\mathscr{L}^{1}$-a.e. $0<s<T$. By our assumptions, 
\begin{align}\label{eq:stottlemeyer}
\mathrm{I}_{38} \stackrel{\theta\searrow 0}{\longrightarrow} \int_{0}^{s}\int_{\Omega}(\partial_{t}\uu)\cdot\bm{\varphi}\dif x \dif t\;\;\;\text{and}\;\;\; \mathrm{I}_{42}\stackrel{\theta\searrow 0}{\longrightarrow} 0.
\end{align}
In order to interchange limits and time integrals, we first note that $F$ being Lipschitz yields 
\begin{align}\label{eq:timederivativemajorant}
\left\vert \underbrace{\frac{1}{\theta}(\mathscr{F}_{0}^{*}[(\uu\pm\theta\bm{\varphi})(t,\cdot);\Omega]-\mathscr{F}_{0}^{*}[\uu(t,\cdot);\Omega])}_{\eqqcolon \mathrm{II}(t)}\right\vert \leq L\, \|\bm{\varphi}(t,\cdot)\|_{\bd(\Omega)} 
\end{align}
for $\mathscr{L}^{1}$-a.e. $0<t<T$. For such $t$, we split 
\begin{align*}
\mathrm{II}(t) & = \int_{\Omega}\frac{1}{\theta}\big(F(\mathscr{T}\uu(t,x)\pm\theta\mathscr{T}\bm{\varphi}(t,x))-F(\mathscr{T}\uu(t,x))\big)\dif x \\ 
& + \frac{1}{\theta}\Big(\int_{\Omega}F^{\infty}\Big(\frac{\dif\,(\mathbb{T}^{s}\uu(t,\cdot)\pm\theta\mathbb{T}^{s}\bm{\varphi}(t,\cdot))}{\dif|\mathbb{T}^{s}(\uu(t,\cdot)\pm\theta\bm{\varphi}(t,\cdot))|} \Big)\dif |\mathbb{T}^{s}\uu(t,\cdot)\pm\theta\mathbb{T}^{s}\bm{\varphi}(t,\cdot)| \Big. \\ & \Big. \;\;\;\;\;\;\;\;\;\;\;\;\;\;\;\;\;\;\;\;\;\;\;\;\;\;\;\;\;\;\;\;\;\;\;\;\;\;\;\;\;\;\;\;\;\;\;\;- \int_{\Omega}F^{\infty}\Big(\frac{\dif\mathbb{T}^{s}\uu(t,\cdot)}{\dif|\mathbb{T}^{s}\uu(t,\cdot)|}\Big)\dif|\mathbb{T}^{s}\uu(t,\cdot)|\Big) \\ 
& + \frac{1}{\theta}\Big(\int_{\partial\Omega}F^{\infty}(-\mathrm{tr}_{\partial\Omega}(\uu(t,\cdot)\pm\theta\bm{\varphi}(t,\cdot))\otimes_{\mathbb{T}}\nu_{\partial\Omega})-F^{\infty}(-\mathrm{tr}_{\partial\Omega}(\uu(t,\cdot))\otimes_{\mathbb{T}}\nu_{\partial\Omega})\dif\mathscr{H}^{1} \Big) \\ 
& \eqqcolon \mathrm{II}_{1}(t) + \mathrm{II}_{2}(t) + \mathrm{II}_{3}(t). 
\end{align*}
For term $\mathrm{II}_{1}(t)$, we employ \ref{item:EL2}. Since $F$ is Lipschitz, $L|\mathscr{T}\bm{\varphi}(t,\cdot)|$ is an integrable majorant of the underlying integrand. We may thus interchange the limit as $\theta\searrow 0$ and the spatial integral. By \ref{item:EL2}, we may moreover focus on those point $x\in\Omega$ where $F$ is differentiable at $\mathscr{T}\uu(t,\cdot)$. Therefore, 
\begin{align*}
\lim_{\theta\searrow 0} \mathrm{II}_{1}(t) = \pm\int_{\Omega}F'(\mathscr{T}\uu(t,x))\cdot\mathscr{T}\bm{\varphi}(t,x)\dif x. 
\end{align*}
For term $\mathrm{II}_{2}(t)$, we note that 
\begin{align*}
|\mathbb{T}^{s}(\uu(t,\cdot)\pm\theta\bm{\varphi}(t,\cdot))| \leq |\mathbb{T}^{s}\uu(t,\cdot)|+|\mathbb{T}^{s}\bm{\varphi}(t,\cdot)|\stackrel{\ref{item:EL1}}{\ll}  |\mathbb{T}^{s}\uu(t,\cdot)|,
\end{align*}
and all of the Radon measures appearing in the preceding line are finite. Hence, in particular, they lead to $\sigma$-finite measure spaces. Thus, by the Lebesgue differentiation theorem for Radon measures, we find by the positive $1$-homogeneity of $F^{\infty}$ that 
\begin{align*}
\int_{\Omega}&F^{\infty}\Big(\frac{\dif\,(\mathbb{T}^{s}\uu(t,\cdot)\pm\theta\mathbb{T}^{s}\bm{\varphi}(t,\cdot))}{\dif|\mathbb{T}^{s}(\uu(t,\cdot)\pm\theta\bm{\varphi}(t,\cdot))|}\Big)\dif|\mathbb{T}^{s}\uu(t,\cdot)\pm\theta\mathbb{T}^{s}\bm{\varphi}(t,\cdot)|  \\ 
& = \int_{\Omega}F^{\infty}\Big(\frac{\dif\,(\mathbb{T}^{s}\uu(t,\cdot)\pm\theta\mathbb{T}^{s}\bm{\varphi}(t,\cdot))}{\dif|\mathbb{T}^{s}(\uu(t,\cdot)\pm\theta\bm{\varphi}(t,\cdot))|}\Big) \frac{\dif|\mathbb{T}^{s}\uu(t,\cdot)\pm\theta\mathbb{T}^{s}\bm{\varphi}(t,\cdot)|}{\dif|\mathbb{T}^{s}\uu(t,\cdot)|}\dif|\mathbb{T}^{s}\uu(t,\cdot)| \\ 
& = \int_{\Omega}F^{\infty}\Big(\frac{\dif\,(\mathbb{T}^{s}\uu(t,\cdot)\pm\theta\mathbb{T}^{s}\bm{\varphi}(t,\cdot))}{\dif|\mathbb{T}^{s}\uu(t,\cdot)|}\Big) \dif|\mathbb{T}^{s}\uu(t,\cdot)| \\ 
& = \int_{\Omega}F^{\infty}\Big(\frac{\dif\mathbb{T}^{s}\uu(t,\cdot)}{\dif|\mathbb{T}^{s}\uu(t,\cdot)|}\pm\theta\frac{\dif\mathbb{T}^{s}\bm{\varphi}(t,\cdot))}{\dif|\mathbb{T}^{s}\uu(t,\cdot)|}\Big) \dif|\mathbb{T}^{s}\uu(t,\cdot)|. 
\end{align*}
Next note that, again because $F^{\infty}$ is Lipschitz (see Lemma \ref{lem:subdif}\ref{item:RobertDenk}), we have  
\begin{align}\label{eq:ruskin}
\left\vert\frac{1}{\theta}\Big( F^{\infty}\Big(\frac{\dif\mathbb{T}^{s}\uu(t,\cdot)}{\dif|\mathbb{T}^{s}\uu(t,\cdot)|}\pm\theta\frac{\dif\mathbb{T}^{s}\bm{\varphi}(t,\cdot))}{\dif|\mathbb{T}^{s}\uu(t,\cdot)|}\Big)- F^{\infty}\Big(\frac{\dif\mathbb{T}^{s}\uu(t,\cdot)}{\dif|\mathbb{T}^{s}\uu(t,\cdot)|}\Big)\Big)\right\vert \leq L \frac{\dif|\mathbb{T}^{s}\bm{\varphi}(t,\cdot)|}{\dif|\mathbb{T}^{s}\uu(t,\cdot)|}
\end{align}
$|\mathbb{T}^{s}\uu(t,\cdot)|$-a.e. in $\Omega$, and since 
\begin{align*}
\int_{\Omega}\frac{\dif|\mathbb{T}^{s}\bm{\varphi}(t,\cdot)|}{\dif|\mathbb{T}^{s}\uu(t,\cdot)|}\dif|\mathbb{T}^{s}\uu(t,\cdot)| = |\mathbb{T}^{s}\bm{\varphi}(t,\cdot)|(\Omega), 
\end{align*}
the right-hand side of \eqref{eq:ruskin} is an integrable majorant of its left-hand side. By the polar decomposition theorem for finite Radon measures, we have 
\begin{align*}
\left\vert \frac{\dif\mathbb{T}^{s}\uu(t,\cdot)}{\dif|\mathbb{T}^{s}\uu(t,\cdot)|}\right\vert = 1\;\;\;|\mathbb{T}^{s}\uu(t,\cdot)|\text{-a.e. in $\Omega$}.
\end{align*}
By our assumptions, $F^{\infty}$ is differentiable at $\frac{\dif\mathbb{T}^{s}\uu(t,\cdot)}{\dif|\mathbb{T}^{s}\uu(t,\cdot)|}(x)$ at $|\mathbb{T}^{s}\uu(t,\cdot)|$-a.e. $x\in\Omega$. In conclusion, we arrive at 
\begin{align}\label{eq:ruskin1}
\begin{split}
\lim_{\theta\searrow 0} \mathrm{II}_{2}(t) & = \pm \int_{\Omega}(F^{\infty})'\Big(\frac{\dif\mathbb{T}^{s}\uu(t,\cdot)}{\dif|\mathbb{T}^{s}\uu(t,\cdot)|}\Big)\cdot\frac{\dif\mathbb{T}^{s}\bm{\varphi}(t,\cdot)}{\dif|\mathbb{T}^{s}\uu(t,\cdot)|}\dif|\mathbb{T}^{s}\uu(t,\cdot)| \\ 
& = \pm \int_{\Omega}(F^{\infty})'\Big(\frac{\dif\mathbb{T}^{s}\uu(t,\cdot)}{\dif|\mathbb{T}^{s}\uu(t,\cdot)|}\Big)\cdot\frac{\dif\mathbb{T}^{s}\bm{\varphi}(t,\cdot)}{\dif|\mathbb{T}^{s}\bm{\varphi}(t,\cdot)|}\dif|\mathbb{T}^{s}\bm{\varphi}(t,\cdot)|. 
\end{split}
\end{align}
By analogous means, we may employ \ref{item:EL3} to find that 
\begin{align}\label{eq:ruskin2}
\lim_{\theta\searrow 0}\mathrm{II}_{3} = \mp \int_{\partial\Omega}(F^{\infty})'(-\mathrm{tr}_{\partial\Omega}(\uu(t,\cdot))\otimes_{\mathbb{T}}\nu_{\partial\Omega}))\cdot\mathrm{tr}_{\partial\Omega}(\bm{\varphi}(t,\cdot))\dif\mathscr{H}^{1}.
\end{align}
Based on \eqref{eq:stottlemeyer}--\eqref{eq:ruskin2}, we may pass to the limit $\theta\searrow 0$ in \eqref{eq:randalldisher}. This gives us a variational inequality with the corresponding '$\pm$' signs. Keeping the inequality corresponding to '$+$' and multiplying the inequality corresponding to '$-$' by $(-1)$, we arrive at \eqref{eq:hiblerPDE}. The proof is complete. 
\end{proof}
We conclude the present subsection with two remarks.
\begin{remark}
If $F\colon\RR\to\R$ is of class $\hold^{1}(\RR)$ and the solution moreover satisfies $\uu\in\lebe^{1}(0,T;\ld_{0}(\Omega))\cap\sobo^{1,2}(0,T;\lebe^{2}(\Omega;\R^{2})$, conditions \ref{item:EL1}--\ref{item:EL3} are not required. In this case, \eqref{eq:hiblerPDE} becomes 
\begin{align*}
\int_{\Omega_{T}}(\partial_{t}\uu)\cdot\bm{\varphi}\dif\,(t,x) + \int_{\Omega_{T}}F'(\mathbb{T}\uu)\cdot\mathbb{T}\bm{\varphi}\dif\,(t,x) = \int_{\Omega_{T}}\tocean(\uu)\cdot\bm{\varphi}\,\dif\,(t,x) + \int_{\Omega_{T}}\mathbf{f}\cdot\bm{\varphi}\,\dif\,(t,x)
\end{align*}
for all $\bm{\varphi}\in\lebe^{1}(0,T;\ld_{0}(\Omega))\cap\sobo^{1,2}(0,T;\lebe^{2}(\Omega;\R^{2}))$. The latter is precisely the \emph{natural} weak formulation of \eqref{eq:motiv} subject to the higher $\ld_{0}$-regularity assumption. In the present context, this however \emph{cannot} be ensured. It is in this way that the formulation from Theorem \ref{thm:relaxedevoleq} captures the potential spatial singularity of the horizontal velocity fields. In particular, \eqref{eq:hiblerPDE} can be understood as the natural relaxed version of a weak formulation, for which Theorem \ref{thm:main} establishes the existence of a solution. 
\end{remark}
\begin{remark}
If we only test \eqref{eq:varfinal} with $\bm{\varphi}\in\hold_{c}^{\infty}(\Omega_{T};\R^{2})$ towards a distributional formulation, then the Hibler singular part $\mathbb{T}^{s}\bm{\varphi}(t,\cdot)$ vanishes globally and $\mathrm{tr}_{\partial\Omega}(\uu\pm\theta\bm{\varphi})=\mathrm{tr}_{\partial\Omega}(\uu)|$. In particular, the terms from \eqref{eq:ruskin1} and \eqref{eq:ruskin2} vanish. If, for simplicity, $F$ is differentiable, we then may rewrite \eqref{eq:hiblerPDE} as 
\begin{align}\label{eq:evol}
\partial_{t}\uu - \mathbb{T}^{*}(F'(\mathscr{T}\uu)) = \mathbf{f} + \tocean(\uu)\qquad\text{in}\;\mathscr{D}'(\Omega_{T};\R^{2}). 
\end{align}
However, \eqref{eq:evol} does not really grasp the key properties of the Hibler system. This is due to the fact that it might happen that $\mathscr{T}\uu(t,\cdot)\neq \mathbb{T}\vv(t,\cdot)$ for any $0<t<T$ and all $\vv(t,\cdot)\in\bd(\Omega)$; this is analogous to Alberti's result for $\bv$, see \cite{Alberti}. In consequence, the term $F'(\mathscr{T}\uu)$ does not represent a proper stress term and so \eqref{eq:evol} \emph{cannot} be understood as a gradient flow equation for the Hibler-type energies. This is due to the fact that \eqref{eq:evol} ignores the singular parts. The underlying deficit is overcome by \eqref{eq:hiblerPDE}, which indeed admits a gradient flow interpretation. 
\end{remark}

\subsection{Generalizations to $\mathbb{C}$-elliptic operators.}
Several of the arguments as given above do not require the specific structure of the Hibler deformations or the specific dimension $n=2$. We thus conclude this section by giving an associated result for general $\mathbb{C}$-elliptic operators $\mathbb{A}$; here, we focus on scenarios where the underlying dimension does not play an important role.
\begin{proposition}\label{prop:Ageneral}
Let $\Omega\subset\R^{n}$ be open and bounded with Lipschitz boundary oriented by the outer unit normal $\nu_{\partial\Omega}\colon\partial\Omega\to\mathbb{S}^{n-1}$, and let $\mathbb{A}$ be a $\mathbb{C}$-elliptic differential operator of the form \eqref{eq:diffopform}. Moreover, let $F\colon W\to\R$ satisfy \emph{\ref{item:Fprop1}--\ref{item:Fprop3}} with the natural modifications and put, for $\vv\in\bv^{\mathbb{A}}(\Omega)$, 
\begin{align}\label{eq:Hieber1}
\begin{split}
\mathscr{F}_{0}^{*}[\vv;\Omega]  \coloneqq \int_{\Omega}F\Big(\frac{\dif\mathbb{A}^{a}\vv}{\dif\mathscr{L}^{n}}\Big)\dif x & + \int_{\Omega}F^{\infty}\Big(\frac{\dif\mathbb{A}^{s}\vv}{\dif|\mathbb{A}^{s}\vv|}\Big)\dif|\mathbb{A}^{s}\vv| \\ &  + \int_{\partial\Omega}F^{\infty}(-\mathrm{tr}_{\partial\Omega}(\vv)\otimes_{\mathbb{A}}\nu_{\partial\Omega})\dif\mathscr{H}^{n-1}, 
\end{split}
\end{align}
where $\mathbb{A}\vv = \mathbb{A}^{a}\vv + \mathbb{A}^{s}\vv$ is the Lebesgue-Radon-Nikod\'{y}m decomposition of $\mathbb{A}\vv$ with respect to $\mathscr{L}^{n}$. Then, for any $T>0$, any $\uu_{0}\in\bv^{\mathbb{A}}(\Omega)$ with compact support in $\Omega$ and any $\mathbf{f}\in\lebe^{2}(\Omega_{T};V)$, there exists a variational solution $\uu\in\lebe_{\mathrm{w}^{*}}^{1}(0,T;\bv^{\mathbb{A}}(\Omega))\cap\sobo^{1,2}(0,T;\lebe^{2}(\Omega;V))$ of the equation 
\begin{align*}
\begin{cases}
\partial_{t}\uu - \mathbb{A}^{*}(\partial F(\mathbb{A}\uu))=\mathbf{f}&\;\text{in}\;\Omega_{T}, \\ 
\uu(0,\cdot) = \uu_{0} &\;\text{in}\;\Omega, \\ 
\uu|_{\partial\Omega}=0&\;\text{on}\;(0,T)\times\partial\Omega. 
\end{cases}
\end{align*}
More precisely, we have $\uu(0,\cdot)=\uu_{0}$ $\mathscr{L}^{n}$-a.e. in $\Omega$ and, for $\mathscr{L}^{1}$-a.e. $0<s<T$, the \emph{evolutionary variational inequality}
\begin{align}\label{eq:varsolmainLAST}
\begin{split}
\int_{0}^{s}\int_{\Omega}(\partial_{t}\mathbf{v})(\mathbf{v}-\mathbf{u})\dif t\dif x & + \int_{0}^{s}\mathscr{F}_{0}^{*}[\mathbf{v}(t,\cdot);\Omega]\dif t  \geq \int_{0}^{s}\mathscr{F}_{0}^{*}[\mathbf{u}(t,\cdot);\Omega]\dif t \\ 
& + \int_{0}^{s}\int_{\Omega}\mathbf{f}\cdot(\vv-\uu)\dif x\dif t \\ 
& + \frac{1}{2}\int_{\Omega}\Big(|\mathbf{u}(s,\cdot)-\mathbf{v}(s,\cdot)|^{2} - |\mathbf{u}_{0}-\mathbf{v}(0,\cdot)|^{2}\Big) \dif x 
\end{split}
\end{align}
holds for every $\vv\in\lebe_{\mathrm{w}^{*}}^{1}(0,T;\bv^{\mathbb{A}}(\Omega))\cap\sobo^{1,2}(0,T;\lebe^{2}(\Omega;V))$.
\end{proposition}
\begin{proof}
We confine ourselves to a sketch. As in Section \ref{sec:existence}, we may perform the viscosity stabilization as in \eqref{eq:motiv1}. The corresponding compactness assertions required for the singular limit $\varepsilon\searrow 0$ then follow from \cite{GmRa}. For the critical limit passage, see Remark \ref{rem:bulk3}, we use Corollary \ref{cor:Celliptic}. Then we may argue and conclude as above. 
\end{proof}

\section{Hibler-Anzellotti pairings and weak-variational solutions}\label{sec:weakvar}
Let $\Omega\subset\R^{2}$ be open and bounded with Lipschitz boundary. In the preceding sections, we worked subject to a constant mass hypothesis. This leads to the notion of solutions as displayed in Definition \ref{def:varsol1} and admits to study a perturbed gradient flow for the relaxed Hibler energy; see Section \ref{sec:main}. The constant mass assumption can be justified in various scenarios, e.g., when considering small time intervals. If we however aim to include variable masses, the concept of solution from Definition \ref{def:varsol1} has to be modified. The chief reason for this is that the variational structure of the leading gradient terms then is lost. 

To explain this point in detail, suppose that a mass function $m\colon \Omega_{T}\to [a,b]$ with some $0<a<b<\infty$ is given. Subject to the initial condition $\uu(0,\cdot)=\uu_{0}$ for some $\uu_{0}\in\bd_{c}(\Omega)$ and boundary conditions $\uu|_{\partial\Omega}=0$, we are  interested in the system 
\begin{align}\label{eq:hiblervariablemass}
m\partial_{t}\uu - \mathbb{T}^{*}\Big(\frac{\mathbb{T}\uu}{|\mathbb{T}\uu|} \Big) = \mathbf{f} + \tocean(\uu)\qquad\text{in}\;\Omega_{T}. 
\end{align}
Imposing a suitable smoothness assumption on $m$, \eqref{eq:hiblervariablemass} by $m$ leads to 
\begin{align}\label{eq:hiblervariablemassrewrite}
\partial_{t}\uu - \mathbb{T}^{*}\Big(\frac{1}{m}\frac{\mathbb{T}\uu}{|\mathbb{T}\uu|} \Big) & + \frac{\mathbb{T}\uu}{|\mathbb{T}\uu|}\otimes_{\mathbb{T}^{*}}\nabla\frac{1}{m}= \frac{1}{m}\mathbf{f} + \frac{1}{m}\tocean(\uu)\qquad\text{in}\;\Omega_{T}. 
\end{align}
We write \eqref{eq:hiblervariablemassrewrite} in this particular form since it then becomes obvious that the second term on the left-hand side of \eqref{eq:hiblervariablemassrewrite} is of variational structure; more precisely, this terms corresponds to the (non-relaxed) \emph{variable mass} Hibler energy 
\begin{align}\label{eq:variablemassHiblerenergy}
\mathscr{E}[\uu(t,\cdot);\Omega] \coloneqq \int_{\Omega}\frac{1}{m(t,\cdot)}\dif|\mathbb{T}\uu(t,\cdot)|.
\end{align}
The third term on the left-hand side of \eqref{eq:hiblervariablemassrewrite}, however, is not of variational form. Indeed, this is a pairing involving the Hibler stress and thus a nonlinear first order expression. Whereas the existence theory from Section \ref{sec:main} circumvents the issue of giving sense to the stress term  $\bm{\sigma}=\frac{\mathbb{T}\uu}{|\mathbb{T}\uu|}$ by proposing an evolutionary  variational reformulation amenable to relaxed energies,
the variable mass scenario \eqref{eq:hiblervariablemass}, \eqref{eq:hiblervariablemassrewrite} does not immediately allow for an analogous procedure. In the remainder of this subsection, we introduce a notion of solution that allows to give \eqref{eq:hiblervariablemass}, \eqref{eq:hiblervariablemassrewrite} a precise meaning. 

\emph{A priori} and conceptually slightly different from Section \ref{sec:main}, we now interpret the stress $\bsigma$ spatially as an $\lebe^{\infty}(\Omega;\RR)$-function; this viewpoint stems from convex analysis and even applies to relaxed formulations, see, e.g., \cite{BeckSchmidtConvex,Bildhauer}. Based on this convention, \eqref{eq:hiblervariablemassrewrite} becomes 
\begin{align}\label{eq:hiblervariablemassrewrite1}
\partial_{t}\uu - \mathbb{T}^{*}\Big(\frac{1}{m}\bsigma\Big) + \bsigma\otimes_{\mathbb{T}^{*}}\nabla\frac{1}{m} = \frac{1}{m}\mathbf{f} + \frac{1}{m}\tocean(\uu)\qquad\text{in}\;\Omega_{T}. 
\end{align}
In this formulation, we have circumvented the critical nonlinear expression $\frac{\mathbb{T}\uu}{|\mathbb{T}\uu|}$. For sufficiently regular maps $\uu$ and $\bsigma$, we may thus consider \eqref{eq:hiblervariablemassrewrite1} in the sense of $\R^{2}$-valued distributions on $\Omega_{T}$. This in itself, however, ignores the constitutive law linking the Hibler operator of the horizontal ice velocity to the stress. Hence, in view of \eqref{eq:hiblervariablemass}--\eqref{eq:hiblervariablemassrewrite}, we must relate $\bsigma$ and $\mathbb{T}\uu$. If $\uu$ were smooth,  $|\mathbb{T}\uu|>0$ and $\bsigma=\frac{\mathbb{T}\uu}{|\mathbb{T}\uu|}$, then \eqref{eq:variablemassHiblerenergy} suggests 
\begin{align}\label{eq:hiblerrelate}
\mathscr{E}[\uu(t,\cdot);\Omega] = \int_{\Omega}\frac{1}{m(t,\cdot)}\bm{\sigma}(t,\cdot)\cdot\mathbb{T}\uu(t,\cdot)\dif x. 
\end{align}
Similar to the discussion in Section \ref{sec:modified}, $\uu$ must be assumed to be of spatial $\bd$-regularity. Therefore, a relaxed variant of the identity \eqref{eq:hiblerrelate} which also takes into account boundary values, needs to firstly declare the relaxed version of the left-hand side of \eqref{eq:hiblerrelate}. The latter is non-trivial, since $\bsigma(t,\cdot)\in\lebe^{\infty}(\Omega;\RR)$ (with respect to $\mathscr{L}^{2})$ and $\mathbb{T}\uu(t,\cdot)$ is, in general, a Radon measure which potentially satisfies $\mathbb{T}\uu\bot\mathscr{L}^{2}$. In particular, it is a priori not clear how to understand the pairing of $\bm{\sigma}$ against $\mathbb{T}\uu$. Inspired by \cite{Anz,KinnunenScheven}, we thus pause to introduce what shall be referred to as \emph{Hibler-Anzellotti pairing}. To this end, we recall that $\mathbb{T}^{*}$ denotes the formal $\lebe^{2}$-adjoint of the Hibler deformation tensor $\mathbb{T}$, see \eqref{eq:formaladjoint}. 
\begin{definition}[Hibler-Anzellotti pairing]\label{def:hibleranzellotti} Let $\Omega\subset\R^{2}$ be open and bounded with Lipschitz boundary, and let $\widetilde{\Omega}\subset\R^{2}$ be open and bounded with $\Omega\Subset\widetilde{\Omega}$. Moreover, let $\bsigma\in\lebe^{\infty}(\Omega;\RR)$ be such that $\mathbb{T}^{*}\bsigma\in\lebe^{p}(\Omega;\R^{2})$ with $p\geq 2$. For $\uu\in\bd(\Omega)$, we introduce the \emph{Hibler-Anzellotti pairing} $[\bsigma\cdot\mathbb{T}\uu]_{0}$ as the distribution defined by 
\begin{align}\label{eq:hibleranzellottipairing}
\langle [\bsigma\cdot\mathbb{T}\uu]_{0},\varphi\rangle & \coloneqq - \int_{\Omega}\varphi\mathbf{u}\cdot\mathbb{T}^{*}\bsigma\dif x - \int_{\Omega}\uu\cdot (\bsigma\otimes_{\mathbb{T}^{*}}\nabla\varphi)\dif x,\qquad\varphi\in\hold_{c}^{\infty}(\widetilde{\Omega}). 
\end{align}
\end{definition}
Since $\bd(\Omega)\hookrightarrow\lebe^{2}(\Omega;\R^{2})$ by Lemma \ref{lem:poincaresobolev} and $\mathbb{T}^{*}\bm{\sigma}\in\lebe^{p}(\Omega;\R^{2})$ with $p\geq 2$, the right-hand side of \eqref{eq:hibleranzellottipairing} is well-defined. Let us note that, in order to grasp boundary effects, \eqref{eq:hibleranzellottipairing} requires test functions on $\widetilde{\Omega}$ (and not $\Omega$) indeed. The key idea behind \eqref{eq:hibleranzellottipairing} is the smooth context, where \eqref{eq:hibleranzellottipairing} can be understood as 
\begin{align}\label{eq:Fever105}
\begin{split}
\int_{\widetilde{\Omega}}(\bsigma\cdot\mathbb{T}\uu)\varphi\dif x & = \int_{\widetilde{\Omega}}(\varphi\bsigma)\cdot\mathbb{T}\uu\dif x = - \int_{\widetilde{\Omega}}\uu\cdot\mathbb{T}^{*}(\varphi\bsigma)\dif x \\ & \!\!\!\! \stackrel{\eqref{eq:prodruleweak}}{=} - \int_{\widetilde{\Omega}}\uu\cdot\varphi\mathbb{T}^{*}\bsigma \dif x - \int_{\widetilde{\Omega}}\uu\cdot(\bsigma\otimes_{\mathbb{T}^{*}}\nabla\varphi)\dif x\qquad\text{for}\;\varphi\in\hold_{c}^{\infty}(\widetilde{\Omega}).
\end{split}
\end{align}
Definition \ref{def:hibleranzellotti} thus allows us to introduce the product $\bsigma\cdot\mathbb{T}\uu$ via a weak formulation motivated by integration by parts. This idea dates back to Anzellotti \cite{Anz}, and variants thereof in the full gradient case are, e.g., due to Beck \& Schmidt \cite{BeckSchmidtConvex} or Kinnunen \& Scheven \cite{KinnunenScheven}. To incorporate the coupling between $\mathbb{T}\uu$ and $\bsigma$ into \eqref{eq:hiblervariablemassrewrite} while keeping in mind the relaxed energies, it is required to identify $[\bsigma\cdot\mathbb{T}\uu]_{0}$ as a measure on $\overline{\Omega}$. This is made precise in: 
\begin{lemma}\label{lem:Radon}
In the situation of Definition \ref{def:hibleranzellotti}, $[\bsigma\cdot\mathbb{T}\uu]_{0}$ can be represented by a \emph{well-defined, signed and finite Radon measure} on $\overline{\Omega}$. For brevity, we simply say that $[\bsigma\cdot\mathbb{T}\uu]_{0}$ \emph{is} a well-defined, signed and finite Radon measure on $\overline{\Omega}$.
\end{lemma}
\begin{proof}
The proof is an adaptation of the arguments given in \cite[Lemma 3.2]{KinnunenScheven} for the full gradient case. We directly infer from \eqref{eq:hibleranzellottipairing} that $[\bm{\sigma}\cdot\mathbb{T}\uu]_{0}$ is a distribution supported in $\overline{\Omega}$. Now let  $\varphi\in\hold_{c}^{\infty}(\widetilde{\Omega})$ be arbitrary  and choose two sequences $(\uu_{i})\subset\hold^{\infty}(\Omega;\R^{2})\cap\bd(\Omega)$ and $(\bm{\sigma}_{j})\subset\hold_{c}^{\infty}(\Omega;\R^{2})$ with the following properties: 
\begin{align}\label{eq:wheneveryouarearound}
\begin{split}
&\uu_{i}\to\uu\;\;\;\text{strictly in $\bd(\Omega)$},\\ & \bm{\sigma}_{j}\stackrel{*}{\rightharpoonup}\bm{\sigma}\;\text{weakly* in $\lebe^{\infty}(\Omega;\R^{2})$}\;\text{and}\\ & \mathbb{T}^{*}\bm{\sigma}_{j}\rightharpoonup\mathbb{T}^{*}\bm{\sigma}\;\text{weakly in $\lebe^{p}(\Omega;\R^{2})$ (weakly* if $p=\infty$)}. 
\end{split}
\end{align}
Here, $\eqref{eq:wheneveryouarearound}_{1}$ is possible due to Lemma \ref{lem:auxBD}\ref{item:aux2}, whereas $\eqref{eq:wheneveryouarearound}_{2},\eqref{eq:wheneveryouarearound}_{3}$ can be established by cut-offs. More precisely, let $(\eta_{j})\subset\hold_{c}^{\infty}(\Omega;[0,1])$ be such that $\eta_{j}\to 1$ pointwise everywhere in $\Omega$. We put $\bm{\sigma}_{j}\coloneqq\eta_{j}\bm{\sigma}$, so that 
\begin{align}\label{eq:linftyabs}
\|\bsigma_{j}\|_{\lebe^{\infty}(\Omega)} \leq \|\bsigma\|_{\lebe^{\infty}(\Omega)}\qquad\text{for all}\;j\in\mathbb{N}. 
\end{align}
For any $\bm{\varphi}\in\lebe^{1}(\Omega;\R^{2})$, we then have $\eta_{j}\bm{\varphi}\to\bm{\varphi}$ strongly in $\lebe^{1}(\Omega;\R^{2})$ by dominated convergence. Hence, $\eqref{eq:wheneveryouarearound}_{2}$ follows from 
\begin{align*}
\int_{\Omega}\bm{\sigma}_{j}\cdot\bm{\varphi}\dif x = \int_{\Omega}\bm{\sigma}\cdot(\eta_{j}\bm{\varphi})\dif x \to \int_{\Omega}\bm{\sigma}\cdot\bm{\varphi}\dif x.
\end{align*}
For $\bm{\varphi}\in\hold_{c}^{\infty}(\Omega;\R^{2})$, we have $\bm{\sigma}\cdot\mathbb{T}\bm{\varphi}\in\lebe^{1}(\Omega)$. Hence, dominated convergence gives us
\begin{align*}
\int_{\Omega}(\mathbb{T}^{*}\bm{\sigma}_{j})\cdot\bm{\varphi}\dif x = - \int_{\Omega}\eta_{j}(\bm{\sigma}\cdot\mathbb{T}\bm{\varphi})\dif x \to - \int_{\Omega}\bm{\sigma}\cdot\mathbb{T}\bm{\varphi}\dif x = \int_{\Omega}(\mathbb{T}^{*}\bm{\sigma})\cdot\bm{\varphi}\dif x. 
\end{align*}
Since $\hold_{c}^{\infty}(\Omega;\R^{2})$ is dense in $\lebe^{p'}(\Omega;\R^{2})$, this implies $\eqref{eq:wheneveryouarearound}_{3}$. In consequence, we find that 
\begin{align*}
|\langle[\bsigma\cdot\mathbb{T}\uu]_{0},\varphi\rangle| & \stackrel{\eqref{eq:hibleranzellottipairing}}{=} \left\vert - \int_{\Omega}\varphi\mathbf{u}\cdot\mathbb{T}^{*}\bsigma\dif x - \int_{\Omega}\uu\cdot (\bsigma\otimes_{\mathbb{T}^{*}}\nabla\varphi)\dif x\right\vert \\ 
& \!\!\!\stackrel{\eqref{eq:wheneveryouarearound}_{2,3}}{=} \lim_{j\to\infty}\left\vert - \int_{\Omega}\varphi\mathbf{u}\cdot\mathbb{T}^{*}\bsigma_{j}\dif x - \int_{\Omega}\uu\cdot (\bsigma_{j}\otimes_{\mathbb{T}^{*}}\nabla\varphi)\dif x\right\vert\;\text{(since $\uu\in\lebe^{2}(\Omega;\R^{2})$)} \\ 
& \!\!\stackrel{\eqref{eq:wheneveryouarearound}_{1}}{=}  \lim_{j\to\infty}\lim_{i\to\infty}\left\vert - \int_{\Omega}\varphi\mathbf{u}_{i}\cdot\mathbb{T}^{*}\bsigma_{j}\dif x - \int_{\Omega}\uu_{i}\cdot (\bsigma_{j}\otimes_{\mathbb{T}^{*}}\nabla\varphi)\dif x\right\vert \\ 
& \stackrel{\eqref{eq:Fever105}}{=} \lim_{j\to\infty}\lim_{i\to\infty} \left\vert \int_{\widetilde{\Omega}}(\bsigma_{j}\cdot\mathbb{T}\uu_{i})\varphi\dif x \right\vert \\ 
& \;\leq \limsup_{j\to\infty}\limsup_{i\to\infty} \|\bm{\sigma}_{j}\|_{\lebe^{\infty}(\Omega)}\|\mathbb{T}\uu_{i}\|_{\lebe^{1}(\Omega)}\|\varphi\|_{\sup} \\ 
& \!\!\!\!\!\!\!\!  \stackrel{\eqref{eq:wheneveryouarearound}_{1},\eqref{eq:linftyabs}}{\leq} \|\bm{\sigma}\|_{\lebe^{\infty}(\Omega)}|\mathbb{T}\uu|(\Omega)\|\varphi\|_{\sup}. 
\end{align*}
In consequence, $[\bm{\sigma},\mathbb{T}\uu]_{0}$ is a distribution of order zero, and so the preceding estimate entails that it can be represented by a signed, finite Radon measure supported in $\overline{\Omega}$. This completes the proof. 
\end{proof}
We now derive the notion of solutions in the variable mass case. To this, we firstly suppose that the underlying maps $\uu,\vv$ are smooth. For $0<t<T$, we multiply  \eqref{eq:hiblervariablemassrewrite1} by $(\uu(t,\cdot)-\vv(t,\cdot))$ and integrate over $\{t\}\times\Omega$. This gives us 
\begin{align}
\int_{\{t\}\times\Omega}\partial_{t}\uu\cdot(\uu-\vv)\dif x & - \int_{\{t\}\times\Omega}\mathbb{T}^{*}\Big(\frac{1}{m}\bm{\sigma}\Big)\cdot(\uu-\vv)\dif x \notag  \\ & +  \int_{\{t\}\times\Omega}\Big(\bm{\sigma}\otimes_{\mathbb{T}^{*}}\nabla\frac{1}{m}\Big)\cdot(\uu-\vv)\dif x \notag \\ & = \int_{\{t\}\times\Omega}\frac{1}{m}\mathbf{f}\cdot(\uu-\vv)\dif x + \int_{\{t\}\times\Omega}\frac{1}{m}\tocean(\uu)\cdot(\uu-\vv)\dif x\label{eq:itsmylife}\\ 
& \Longleftrightarrow: \mathrm{II}_{1} - \mathrm{II}_{2} + \mathrm{II}_{3} = \mathrm{II}_{4} + \mathrm{II}_{5}. \notag
\end{align}
To deal with the term $\mathrm{II}_{2}$ in the rough context, we employ the following lemma. 
\begin{lemma}\label{lem:rewrite}
In the situation of Definition \ref{def:hibleranzellotti}, let $\vv\in\bd(\Omega)$. Then the Radon measure $[\bm{\sigma}\cdot\mathbb{T}\vv]_{0}$ satisfies 
\begin{align}\label{eq:resurrection}
[\bm{\sigma}\cdot\mathbb{T}\vv]_{0}(\overline{\Omega}) = - \int_{\Omega}(\mathbb{T}^{*}\bm{\sigma})\cdot\vv\dif x. 
\end{align}
\end{lemma}
\begin{proof}
By Lemma \ref{lem:Radon}, $[\bm{\sigma}\cdot\mathbb{T}\vv]_{0}$ is a Radon measure on $\widetilde{\Omega}$ supported in $\overline{\Omega}$. Now let $(\varphi_{j})\subset\hold_{c}^{\infty}(\widetilde{\Omega};[0,1])$ be such that $\varphi_{j}=1$ on $\overline{\Omega}$ for all $j\in\mathbb{N}$ and $\varphi_{j}(x)\to 0$ for all $x\in\widetilde{\Omega}\setminus\overline{\Omega}$. In consequence, dominated convergence yields 
\begin{align*}
[\bm{\sigma}\cdot\mathbb{T}\vv]_{0}(\overline{\Omega})= \lim_{j\to\infty} \langle[\bm{\sigma}\cdot\mathbb{T}\vv]_{0},\varphi_{j}\rangle \stackrel{\eqref{eq:hibleranzellottipairing}}{=} - \lim_{j\to\infty}\int_{\Omega}\varphi_{j}\vv\cdot\mathbb{T}^{*}\bm{\sigma}\dif x = - \int_{\Omega}\mathbb{T}^{*}\bm{\sigma}\cdot\vv\dif x, 
\end{align*}
since $\nabla\varphi_{j}=0$ in $\Omega$ for all $j\in\mathbb{N}$. This is \eqref{eq:resurrection}, and the proof is complete. 
\end{proof}
We now consider the critical term $\mathrm{II}_{2}$. In view of \eqref{eq:hiblerrelate} and in order to link $\mathbb{T}\uu$ to the stress $\bsigma$, we now postulate 
\begin{align}\label{eq:postulate}
[\widetilde{\bsigma}(t,\cdot)\cdot\mathbb{T}\uu(t,\cdot)]_{0}(\overline{\Omega}) \stackrel{!}{=} \int_{\Omega}\frac{1}{m(t,\cdot)}\dif|\mathbb{T}\uu(t,\cdot)|.
\end{align}
In a next step, we use Lemma \ref{lem:rewrite} to write with $\widetilde{\bsigma}\coloneqq \frac{1}{m}\bm{\sigma}$:
\begin{align*}
\mathrm{II}_{2} = \int_{\{t\}\times\Omega} \mathbb{T}^{*}\widetilde{\bsigma}\cdot(\uu-\vv) \dif x & \stackrel{\eqref{eq:resurrection}}{=} [\widetilde{\bsigma}(t,\cdot)\cdot\mathbb{T}\vv(t,\cdot)]_{0}(\overline{\Omega}) - [\widetilde{\bsigma}(t,\cdot)\cdot\mathbb{T}\uu(t,\cdot)]_{0}(\overline{\Omega}) \\ 
& \stackrel{\eqref{eq:postulate}}{=} [\widetilde{\bsigma}(t,\cdot)\cdot\mathbb{T}\vv(t,\cdot)]_{0}(\overline{\Omega}) - \int_{\Omega}\frac{1}{m(t,\cdot)}\dif|\mathbb{T}\uu(t,\cdot)|. 
\end{align*}
In view of \eqref{eq:itsmylife}, the ultimate identity motivates the following notion of solutions: 
\begin{definition}[Weak-variational solutions]\label{def:weakvarsol}
Let $T>0$, let $\Omega\subset\R^{2}$ be open and bounded with Lipschitz boundary, and let $\uu_{0}\in\bd_{c}(\Omega)$. Moreover, let $m\in \lebe^{2}(0,T;\sobo^{1,2}(\Omega))$ such that 
\begin{align}\label{eq:massasses}
m\in\hold(\overline{\Omega_{T}})\;\;\;\text{and}\;\;\;\min_{\overline{\Omega}}m>0\;\;\text{for all $0<t<T$}.
\end{align}
and suppose that the data are as in Definition \ref{def:varsol1}. We say that $\uu\in\lebe_{\mathrm{w}^{*}}^{1}(0,T;\bd(\Omega))\cap\sobo^{1,2}(0,T;\lebe^{2}(\Omega;\R^{2}))$ is a \emph{weak-variational solution} of the variable mass momentum balance equation \eqref{eq:hiblervariablemass} if there exists $\bsigma\in\lebe^{\infty}(\Omega_{T};\RR)$ with $\|\bsigma(t,\cdot)\|_{\lebe^{\infty}(\Omega)}\leq 1$ for $\mathscr{L}^{1}$-a.e. $0<t<T$ such that the following hold:
\begin{enumerate}
    \item\label{item:varmass1} \emph{Initial values:} $\uu(0,\cdot)=\uu_{0}$ holds  $\mathscr{L}^{2}$-a.e. in $\Omega$, 
    \item\label{item:varmass2} \emph{Distributional velocity-stress evolution:} The equation  
\begin{align}\label{eq:hiblervariablemassrewrite2}
\partial_{t}\uu - \mathbb{T}^{*}\Big(\frac{1}{m}\bsigma\Big) + \bsigma\otimes_{\mathbb{T}^{*}}\nabla\frac{1}{m} = \frac{1}{m}\mathbf{f} + \frac{1}{m}\tocean(\uu)\qquad\text{holds in}\;\mathscr{D}'(\Omega_{T};\R^{2}). 
\end{align}
    \item\label{item:varmass3} \emph{Energy-stress coupling:} For $\mathscr{L}^{1}$-a.e. $0<t<T$ and all $\vv\in\mathrm{L}_{\mathrm{w}^{*}}^{1}(0,T;\bd(\Omega))$, we have  
    \begin{align}
\int_{\Omega\times\{t\}}\partial_{t}\uu\cdot(\uu-\vv)\dif x & + \int_{\{t\}\times\Omega}\frac{1}{m}\dif|\mathbb{T}\uu(t,\cdot)| + \int_{\Omega\times\{t\}}\Big(\bsigma\otimes_{\mathbb{T}^{*}}\nabla\frac{1}{m}\Big)\cdot(\uu-\vv)\dif x \notag \\ & = \int_{\Omega\times\{t\}}\frac{1}{m(t,\cdot)}\mathbf{f}\cdot(\uu-\vv)\dif x  + \int_{\Omega\times\{t\}}\frac{1}{m}\tocean(\uu)\cdot(\uu-\vv)\dif x \label{eq:hiblerveryweak}\\ 
& + \left[\frac{1}{m}\bsigma(t,\cdot)\cdot\mathbb{T}\vv(t,\cdot)\right]_{0}(\overline{\Omega})\notag
\end{align}
for all $\vv\in\lebe_{\mathrm{w}^{*}}(0,T;\bd(\Omega))$ and $\mathscr{L}^{1}$-a.e. $0<t<T$.
\end{enumerate}  
\end{definition}
We call this formulation \emph{weak-variational} for the following reason: Property \ref{item:varmass3} re-introduces the variable mass \emph{energy} term, but this happens at the cost of additional lower order terms that can only be dealt with \emph{weakly}.

Note that the energy-stress coupling from \eqref{eq:hiblerveryweak} is well-defined. For the first term, this follows from $\bd(\Omega)\hookrightarrow\lebe^{2}(\Omega;\R^{2})$ in $n=2$ dimensions. For the second term, we note that $\frac{1}{m}$ is continuous up to the boundary and  bounded away from zero by \eqref{eq:massasses}; in $n=2$, $\sobo^{1,2}(\Omega)\not\hookrightarrow\hold(\overline{\Omega})$, and so the continuity has to be required separately. In particular, $\frac{1}{m}\in\hold(\overline{\Omega})$, whereby the second term is well-defined too. Our assumption \eqref{eq:massasses} implies that $\nabla\frac{1}{m}=-\frac{1}{m^{2}}\nabla m\in\lebe^{2}(\Omega_{T};\R^{2})$, and since $\bm{\sigma}(t,\cdot)\in\lebe^{\infty}(\Omega;\RR)$ together with $(\uu-\vv)(t,\cdot)\in\lebe^{2}(\Omega;\R^{2})$, the third term is equally well-defined. The well-definedness of the fourth and fifth terms follow from the assumption on our data. For the sixth term, we note that $\frac{1}{m(t,\cdot)}\bm{\sigma}(t,\cdot)\in\lebe^{\infty}(\Omega;\RR)$. 
By our assumptions on $m$, see \eqref{eq:massasses}, and the underlying data, we find because of the requirement $\partial_{t}\uu\in\lebe^{2}(\Omega_{T};\R^{2})$:  
\begin{align*}
- \mathbb{T}^{*}\Big(\frac{1}{m}\bsigma\Big)= - \partial_{t}\uu - \bsigma\otimes_{\mathbb{T}^{*}}\nabla\frac{1}{m}-\frac{1}{m}\mathbf{f} - \frac{1}{m}\tocean(\uu)\in\lebe^{2}(\Omega_{T};\R^{2}), 
\end{align*}
and so, in particular, $\mathbb{T}^{*}(\frac{1}{m(t,\cdot)}\bsigma(t,\cdot))\in\lebe^{2}(\Omega;\R^{2})$ for $\mathscr{L}^{1}$-a.e. $0<t<T$. In consequence, $\frac{1}{m(t,\cdot)}\bsigma(t,\cdot)$ satisfies the assumptions from Definition \ref{def:hibleranzellotti}, and so the sixth term is well-defined too. We now briefly address the required regularity on the mass functions $m$:
\begin{remark}[Regularity of $m$]
Definition \ref{def:hibleranzellotti} offers a pairing between $\lebe^{\infty}$-maps and special Radon measures $\bm{\mu}=\mathbb{T}\uu$. Here, the differential structure of $\bm{\mu}$ is essential since it gives us access to an integration by parts-type formula. In the second term of \eqref{eq:hiblerveryweak}, this is not the case; the total Hibler deformation measure $|\mathbb{T}\uu(t,\cdot)|$ is a non-negative measure without differential structure. A priori, the spatial continuity assumption  on $m$ can only be relaxed towards requiring $m(t,\cdot)$ to be $|\mathbb{T}\uu(t,\cdot)|$-measurable for $\mathscr{L}^{1}$-a.e. $0<t<T$ and, moreover, belonging to $\lebe_{|\mathbb{T}\uu(t,\cdot)|}^{\infty}(\overline{\Omega})$ for $\mathscr{L}^{1}$-a.e. $0<t<T$. 
\end{remark}
Definition \ref{def:weakvarsol} assumes the mass function $m$ to be given. Whereas its solvability shall be the subject of future work, we now give the link to the constant mass case from Definition \ref{def:varsol1} in a simplified scenario.  
\begin{proposition}\label{prop:weakvarcomm}
In the situation of Definition \ref{def:weakvarsol}, suppose that the mass function is constant ($m=c>0$). Moreover, assume that $\mathbf{f}\equiv 0$ and that the function $\widetilde{\eta}$ from \eqref{eq:oceanicHoelder} satisfies $\widetilde{\eta}\equiv 0$; in particular, the oceanic force terms are absent. Then the variational solution from Theorem \ref{thm:main} with $F=|\cdot|$ is a weak-variational solution in the sense of Definition \ref{def:weakvarsol}. 
\end{proposition}
Up to minor modifications, the proof of Proposition \ref{prop:weakvarcomm} is fully analogous to the corresponding statement for the total variation flow as given by Kinnunen \& Scheven in \cite[Theorem 5.1]{KinnunenScheven}. We therefore omit it and refer the reader to future work, where this implication shall be addressed for non-trivial forces $\mathbf{f}$ and $\tocean$-terms, finally dealing with variable masses. 
\section{Appendix}\label{sec:appendix}
In this appendix, we establish two auxiliary results which are of importance of the main part, yet are difficult to be traced back to a specific reference.
\subsection{Proof of Lemma \ref{lem:jensen}}\label{sec:jensen}
In this section, we give the proof of Jensen's inequality \eqref{eq:Jensen}. 
\begin{proof}[Proof of Lemma \ref{lem:jensen}] Let $x\in U$. 
We define a positive measure $\nu\coloneqq \mathscr{L}^{n}+|\bm{\mu}^{s}|$, which we think to be extended from $\Omega$ to $\R^{n}$ by zero. On $\Omega$, we then have both $\mathscr{L}^{n}\ll\nu$ and $|\bm{\mu}^{s}|\ll\nu$. Hence, in particular, 
\begin{align}\label{eq:sharonafleming}
\mathscr{L}^{n}=\frac{\dif\mathscr{L}^{n}}{\dif\nu}\nu\;\;\;\text{and}\;\;\;\bm{\mu}^{s}=\frac{\dif\bm{\mu}^{s}}{\dif\nu}\nu\qquad\text{as measures on $\Omega$}. 
\end{align}
Moreover, by the Radon-Nikod\'{y}m theorem, there exists a $\nu$-measurable set $A\subset\Omega$ such that $|\bm{\mu}^{s}|(A)=\mathscr{L}^{n}(\Omega\setminus A)=0$. In consequence, we have  
\begin{align}\label{eq:lelandstottle}
\nu=\mathscr{L}^{n}\;\;\text{on}\;A\;\;\;\text{and}\;\;\;\nu=|\bm{\mu}^{s}|\;\;\text{on}\;\Omega\setminus A. 
\end{align}
We now define a probability measure via 
\begin{align}\label{eq:sharonafleming1}
\widetilde{\nu}_{x} \coloneqq \frac{\nu}{\nu(\ball_{\varepsilon}(x))} 
\end{align}
and recall that the linear perspective integrand
\begin{align}\label{eq:sharonafleming2}
\text{$F^{\#}$ from \eqref{eq:linearperspective} is  convex and positively $1$-homogeneous}. 
\end{align}
We proceed to estimate as follows:
\begin{align*}
F(\rho_{\varepsilon}*\bm{\mu}(x)) & = F\Big(\int_{\R^{n}}\rho_{\varepsilon}(x-y)\dif\bm{\mu}(y)\Big) \\ 
& = F^{\#}\Big(\underbrace{\int_{\R^{n}}\rho_{\varepsilon}(x-y)\dif\mathscr{L}^{n}(y)}_{=1},\int_{\R^{n}}\rho_{\varepsilon}(x-y)\dif\bm{\mu}(y)\Big) \\ 
& \!\!\stackrel{\eqref{eq:sharonafleming}}{=} F^{\#}\Big(\int_{\ball_{\varepsilon}(x)}\rho_{\varepsilon}(x-y)\frac{\dif\mathscr{L}^{n}}{\dif\nu}(y)\dif\nu(y),\int_{\ball_{\varepsilon}(x)}\rho_{\varepsilon}(x-y)\frac{\dif\bm{\mu}}{\dif\nu}(y)\dif\nu(y) \Big) \\ 
& \!\!\!\!\!\!\!\! \stackrel{\eqref{eq:sharonafleming1},\eqref{eq:sharonafleming2}}{=} F^{\#}\Big(\int_{\ball_{\varepsilon}(x)}\Big(\rho_{\varepsilon}(x-y)\frac{\dif\mathscr{L}^{n}}{\dif\nu}(y),\rho_{\varepsilon}(x-y)\frac{\dif\bm{\mu}}{\dif\nu}(y)\Big)\dif\widetilde{\nu}_{x}(y)\Big)\nu(\ball_{\varepsilon}(x)) \\ &  \eqqcolon \mathrm{I}. 
\end{align*}
At this stage, we employ the classical Jensen inequality for functions and probability measures. In consequence, we arrive at 
\begin{align*}
\mathrm{I}\,\,\,\,\,\,\,\,& \!\!\!\!\!  \stackrel{\text{Jensen}}{\leq} \int_{\ball_{\varepsilon}(x)}F^{\#}\Big(\rho_{\varepsilon}(x-y)\frac{\dif\mathscr{L}^{n}}{\dif\nu}(y),\rho_{\varepsilon}(x-y)\frac{\dif\bm{\mu}}{\dif\nu}(y)\Big)\dif\nu(y)\\ 
& \!\!\! \stackrel{\eqref{eq:sharonafleming2}}{=} \int_{\ball_{\varepsilon}(x)}\rho_{\varepsilon}(x-y)F^{\#}\Big(\frac{\dif\mathscr{L}^{n}}{\dif\nu}(y),\frac{\dif\bm{\mu}}{\dif\nu}(y)\Big)\dif\nu(y)  \\ 
& \!\!\!\!\!\!\!\!\! \stackrel{\eqref{eq:sharonafleming}, \eqref{eq:lelandstottle}}{=} \int_{\ball_{\varepsilon}(x)\cap A}\rho_{\varepsilon}(x-y)F^{\#}\Big(\frac{\dif\mathscr{L}^{n}}{\dif\nu}(y),\frac{\dif\bm{\mu}}{\dif\nu}\frac{\dif\nu}{\dif\mathscr{L}^{n}}\frac{\dif\mathscr{L}^{n}}{\dif\nu}(y)\Big)\dif\nu(y) \\ & + \int_{\ball_{\varepsilon}(x)\cap(\Omega\setminus A)}\rho_{\varepsilon}(x-y)F^{\#}\Big(0,\frac{\dif\bm{\mu}}{\dif\nu}\frac{\dif\nu}{\dif|\bm{\mu}^{s}|}\frac{\dif|\bm{\mu}^{s}|}{\dif\nu}(y)\Big)\dif\nu(y) \\ 
& \!\!\!\! \stackrel{\eqref{eq:sharonafleming2}}{=} \int_{\ball_{\varepsilon}(x)\cap A}\rho_{\varepsilon}(x-y)F^{\#}\Big(1,\frac{\dif\bm{\mu}^{a}}{\dif\mathscr{L}^{n}}\Big)\dif\mathscr{L}^{n}(y) \\ & + \int_{\ball_{\varepsilon}(x)\cap(\Omega\setminus A)}\rho_{\varepsilon}(x-y)F^{\infty}\Big(\frac{\dif\bm{\mu}^{s}}{\dif|\bm{\mu}^{s}|}\Big)\dif|\bm{\mu}^{s}|(y) \\ & = \int_{\ball_{\varepsilon}(x)}\rho_{\varepsilon}(x-y)\dif F(\bm{\mu})(y) = \int_{\R^{n}}\rho_{\varepsilon}(x-y)\dif F(\bm{\mu})(y). 
\end{align*}
Now we use Fubini's theorem to conclude 
\begin{align*}
\int_{U}F(\rho_{\varepsilon}*\bm{\mu}(x))\dif x & \leq  \int_{U}\int_{\R^{n}}\rho_{\varepsilon}(x-y)\dif F(\bm{\mu})(y)\dif x = \int_{\R^{n}}\int_{U}\rho_{\varepsilon}(x-y)\dif x \dif F(\bm{\mu})(y) \\ 
& = \int_{U_{\varepsilon}}\int_{\R^{n}}\rho_{\varepsilon}(x-y)\dif y \dif F(\bm{\mu})(y) = \int_{U_{\varepsilon}} F(\bm{\mu}). 
\end{align*}
This concludes the proof. 
\end{proof}
\subsection{The predual of $\bv^{\mathbb{A}}(\Omega)$} For the space $\lebe_{\mathrm{w}^{*}}^{1}(0,T;\bd(\Omega))$ to be well-defined, it is clear that $\bd(\Omega)$ must have a (separable) predual. In view of Proposition \ref{prop:Ageneral}, we directly state and prove the following result for $\mathbb{C}$-elliptic operators $\mathbb{A}$. 
\begin{proposition}\label{prop:predual}
Let $\mathbb{A}$ be a first order, $\mathbb{C}$-elliptic differential operator of the form \eqref{eq:diffopform}, and let $\Omega\subset\R^{n}$ be open and bounded with Lipschitz boundary. Then there exists a separable space $(X,\|\cdot\|_{X})$ of continuous maps on $\Omega$ such that 
\begin{align}\label{eq:keyduality}
(X,\|\cdot\|_{X})'\cong (\bv^{\mathbb{A}}(\Omega),\|\cdot\|_{\bv^{\mathbb{A}}(\Omega)}), 
\end{align}
where $\|\uu\|_{\bv^{\mathbb{A}}(\Omega)}\coloneqq \|\uu\|_{\lebe^{1}(\Omega)}+|\mathbb{A}\uu|(\Omega)$. 
\end{proposition}
Note that it is natural to require some sort of ellipticity on $\mathbb{A}$, as otherwise Proposition \ref{prop:predual} is false. Namely, taking $\mathbb{A}\equiv 0$, $\bv^{\mathbb{A}}(\Omega)=\lebe^{1}(\Omega;V)$, and $\lebe^{1}(\Omega;V)$ does not have a predual. For the proof of Proposition \ref{prop:predual}, we require the following classical auxiliary result.
\begin{lemma}\label{lem:abstisom}
    Let $(\mathcal{X},\|\cdot\|_{\mathcal{X}})$ be a normed space and let $\mathcal{Y}\subset \mathcal{X}$ be a closed subspace. We define the annihilator of $\mathcal{Y}$ by $\mathcal{Y}^{\bot}\coloneqq \{x'\in \mathcal{X}'\colon\;x'(x)=0\;\text{for all}\;x\in \mathcal{Y}\}$. Then there holds 
    \begin{align}\label{eq:abstractisomorphy}
    (\mathcal{X}/\mathcal{Y})'\cong \mathcal{Y}^{\bot}. 
    \end{align}
More precisely, the operator $\Phi\colon (\mathcal{X}/\mathcal{Y})'\ni f \mapsto \Phi_{f}\in \mathcal{Y}^{\bot}$, where $\Phi_{f}(x)\coloneqq f(x+\mathcal{Y})$, is an  \emph{isometric isomorphism}.
\end{lemma}
We are now ready to give the
\begin{proof}[Proof of Proposition \ref{prop:predual}]
We aim to apply Lemma~\ref{lem:abstisom}. More precisely, we let $(\mathcal{X},\|\cdot\|_\mathcal{X})=(\hold_{0}(\Omega;V\times W),\|\cdot\|_{\hold(\Omega)})$ and construct a closed subspace $\mathcal{Y}\subset\mathcal{X}$ such that the annihilator $\mathcal{Y}^{\bot}$ is isomorphic to $\bv^{\mathbb{A}}(\Omega)$. By Lemma \ref{lem:abstisom},  $(\mathcal{X}/\mathcal{Y})'\cong\mathcal{Y}^{\bot}$, and so $\bv^{\mathbb{A}}(\Omega)$ will be shown to be isomorphic to $(\mathcal{X}/\mathcal{Y})'$. Then it suffices to put $X\coloneqq \mathcal{X}/\mathcal{Y}$ to conclude \eqref{eq:keyduality}, as $(X,\|\cdot\|_{X})$ is a separable Banach space; $X=\mathcal{X}/\mathcal{Y}$ is a separable Banach space since $\mathcal{X}$ is a separable Banach space and $\mathcal{Y}$ is closed. Throughout, we tacitly identify $V\cong V'$ and $W\cong W'$. 

By the Riesz representation theorem for vectorial Radon measures, we have 
\begin{align*}
\mathcal{X}'\cong\mathrm{RM}_{\mathrm{fin}}(\Omega;V\times W),
\end{align*}
where the finite $V\times W$-valued Radon measures are equipped with the total variation norm. We denote by $\mathscr{J}_{1}\colon\mathrm{RM}_{\mathrm{fin}}(\Omega;V\times W)\to \mathcal{X}'$ the underlying isometric isomorphism. The space $\bv^{\mathbb{A}}(\Omega)$ can be embedded isometrically into $\mathrm{RM}_{\mathrm{fin}}(\Omega;V\times W)$ by means of 
\begin{align*}
\mathscr{J}_{2}\colon \bv^{\mathbb{A}}(\Omega)\ni \uu \mapsto (\uu\,\mathscr{L}^{n}\mres\Omega,\mathbb{A}\uu)\in\mathrm{RM}_{\mathrm{fin}}(\Omega;V\times W). 
\end{align*}
By definition of the total variation norm $\|\bm{\mu}\|_{\mathrm{RM}_{\mathrm{fin}}(\Omega)}\coloneqq|\bm{\mu}|(\Omega)$ on $\mathrm{RM}_{\mathrm{fin}}(\Omega;V\times  W)$, we have with $\mathscr{J}\coloneqq\mathscr{J}_{1}\mathscr{J}_{2}$ that 
\begin{align}\label{eq:isomorphyBVdual}
\|\uu\|_{\bv^{\mathbb{A}}(\Omega)} =  \|\mathscr{J}_{2}\uu\|_{\mathrm{RM}_{\mathrm{fin}}(\Omega)} =\|\mathscr{J}\uu\|_{\mathcal{X}'}\qquad\text{for all}\;\uu\in\bv^{\mathbb{A}}(\Omega). 
\end{align}
Now define $\mathcal{Y}\subset\mathcal{X}$ to be the closure of 
\begin{align*}
Y\coloneqq \{\bm{\psi}=(\bm{\psi}_{0},\bm{\psi}_{1})\in\hold_{c}^{\infty}(\Omega;V\times W)\colon\;\bm{\psi}_{0}=\mathbb{A}^{*}\bm{\psi}_{1}\}
\end{align*}
in $\hold_{0}(\Omega;V\times W)$ with respect to $\|\cdot\|_{\hold(\Omega)}$. Here, 
\begin{align*}
\mathbb{A}^{*} \coloneqq \sum_{k=1}^{n}\mathbb{A}_{k}^{\top}\partial_{k}
\end{align*}
is the formal $\lebe^{2}$-adjoint of $\mathbb{A}$, see \eqref{eq:diffopform}. 
We now establish that $\mathcal{Y}^{\bot}=\mathscr{J}(\bv^{\mathbb{A}}(\Omega))$, and this will conclude the proof of the proposition. 

We begin with the inclusion '$\supset$'. Let $\bm{\psi}=(\bm{\psi}_{0},\bm{\psi}_{1})\in\mathcal{Y}$ and choose a sequence $(\bm{\psi}^{i})\subset\hold_{c}^{\infty}(\Omega;V\times W)$ satisfying $\bm{\psi}^{i}=(\bm{\psi}_{0}^{i},\bm{\psi}_{1}^{i})$ and $\bm{\psi}_{0}^{i}=\mathbb{A}^{*}\bm{\psi}_{1}^{i}$, such that $\|\bm{\psi}^{i}-\bm{\psi}\|_{\hold(\Omega)}\to 0$.  Given $\uu\in\bv^{\mathbb{A}}(\Omega)$, we compute the duality pairing as
\begin{align*}
\langle \mathscr{J}\uu,\bm{\psi}\rangle & = \int_{\Omega}\uu\cdot\bm{\psi}_{0}\dif x + \int_{\Omega}\bm{\psi}_{1}\cdot\dif\mathbb{A}\uu = \lim_{i\to\infty}\Big(\int_{\Omega}\uu\cdot\bm{\psi}_{0}^{i}\dif x + \int_{\Omega}\bm{\psi}_{1}^{i}\cdot\dif\mathbb{A}\uu \Big)\\ 
& = \lim_{i\to\infty} \Big(\int_{\Omega}\uu\cdot\bm{\psi}_{0}^{i}\dif x  - \int_{\Omega}(\mathbb{A}^{*}\bm{\psi}_{1}^{i})\cdot\uu\dif x \Big) \stackrel{\bm{\psi}^{i}\in Y}{=} 0, 
\end{align*}
so that $\mathscr{J}(\bv^{\mathbb{A}}(\Omega))\subset \mathcal{Y}^{\bot}$. In view of '$\subset$', let $f\in\mathcal{Y}^{\bot}$. By definition, we  have $f\in \mathcal{X}'$ and $\langle f,\bm{\psi}\rangle=0$ for all $\bm{\psi}\in\mathcal{Y}$. In particular, since $f\in\mathcal{X}'$, the Riesz representation theorem implies that there exists $\bm{\mu}=(\bm{\mu}_{0},\bm{\mu}_{1})\in\mathrm{RM}_{\mathrm{fin}}(\Omega;V\times W)$ such that 
\begin{align}\label{eq:traylorhoward}
\langle f,\bm{\psi}\rangle = \sum_{j\in\{0,1\}}\int_{\Omega}\bm{\psi}_{j}\cdot\dif\bm{\mu}_{j}\qquad\text{for all}\;\bm{\psi}=(\bm{\psi}_{0},\bm{\psi}_{1})\in\hold_{0}(\Omega;V\times W).
\end{align}
We need to identify $\bm{\mu}$ as $\mathscr{J}_{2}\uu$ for some $\uu\in\bv^{\mathbb{A}}(\Omega)$. For $\bm{\psi}=(\bm{\psi}_{0},\bm{\psi}_{1})\in Y$, the left side of \eqref{eq:traylorhoward} vanishes. Since $\bm{\psi}_{0}=\mathbb{A}^{*}\bm{\psi}_{1}$ for such $\bm{\psi}$, we thus get 
\begin{align}\label{eq:daflow}
\int_{\Omega}(\mathbb{A}^{*}\bm{\psi}_{1})\cdot\dif \bm{\mu}_{0} = \int_{\Omega}\bm{\psi}_{0}\cdot\dif\bm{\mu}_{0} = - \int_{\Omega}\bm{\psi}_{1}\dif\bm{\mu}_{1},  
\end{align}
and this identity holds for \emph{all}  $\bm{\psi}_{1}\in\hold_{c}^{\infty}(\Omega;W)$.

For an arbitrary $\omega\Subset\Omega$ and $\varepsilon>0$ sufficiently small, the set $\Omega^{\varepsilon}:=\{x\in\Omega\colon\;\mathrm{dist}(x,\partial\Omega)>\varepsilon\}$ contains $\omega$. We let $\rho_{\varepsilon}$ be the $\varepsilon$-rescaled variant of a standard mollifier $\rho$, so that the measures $\bm{\mu}_{j}^{\varepsilon}:=(\rho_{\varepsilon}*\bm{\mu}_{j}\mres\Omega_{\varepsilon})\mathscr{L}^{n}\mres\Omega$ are well-defined for $j\in\{0,1\}$. For any $\bm{\psi}_{1}\in\hold_{c}^{\infty}(\omega;W)$, we have 
\begin{align*}
\int_{\omega}(\mathbb{A}^{*}\bm{\psi}_{1})\cdot(\rho_{\varepsilon}*\bm{\mu}_{0})\dif x & = \int_{\Omega}(\mathbb{A}^{*}\bm{\psi}_{1})\cdot\dif\bm{\mu}_{0}^{\varepsilon} = \int_{\Omega}(\mathbb{A}^{*}(\rho_{\varepsilon}*\bm{\psi}_{1}))\cdot\dif\bm{\mu}_{0} \\ &  \!\!\!\stackrel{\eqref{eq:daflow}}{=} - \int_{\Omega}(\rho_{\varepsilon}*\bm{\psi}_{1})\cdot\dif\bm{\mu}_{1} = - \int_{\Omega}\bm{\psi}_{1}\cdot(\rho_{\varepsilon}*\bm{\mu}_{1})\dif x.
\end{align*}
By arbitrariness of $\bm{\psi}_{1}$, we see that
\begin{align*}
(\rho_{\varepsilon}*(\bm{\mu}_{1}\mres\Omega_{\varepsilon}))\mathscr{L}^{n}\mres\Omega=\mathbb{A}(\rho_{\varepsilon}*\bm{\mu}_{0})\qquad\text{on}\;\omega. 
\end{align*}
Let $(\varepsilon_{l})\subset\R_{>0}$ with $\varepsilon_{l}\searrow 0$ as $l\to\infty$. By arbitrariness of $\omega$, we have 
\begin{align}\label{eq:convossimo}
\begin{split}
\uu_{l}\,\mathscr{L}^{n}\mres\Omega \coloneqq (\rho_{\varepsilon_{l}}*(\bm{\mu}_{0}\mres\Omega_{\varepsilon_{l}}))\mathscr{L}^{n}\mres\Omega\stackrel{*}{\rightharpoonup}\bm{\mu}_{0}&\;\text{in}\;\mathrm{RM}(\Omega;V),\\ 
\mathbf{v}_{l}\mathscr{L}^{n}\mres\Omega \coloneqq(\rho_{\varepsilon_{l}}*(\bm{\mu}_{1}\mres\Omega_{\varepsilon_{l}}))\mathscr{L}^{n}\mres\Omega \stackrel{*}{\rightharpoonup}\bm{\mu}_{1}&\;\text{in}\;\mathrm{RM}(\Omega;W)
\end{split}
\end{align}
as $l\to\infty$. Moreover, since $\bm{\mu}_{0}\in\mathrm{RM}_{\mathrm{fin}}(\Omega;V)$ and $\bm{\mu}_{1}\in\mathrm{RM}_{\mathrm{fin}}(\Omega;W)$, we have 
\begin{align}\label{eq:boundossimo}
\sup_{l\in\mathbb{N}}\|\uu_{l}\|_{\lebe^{1}(\Omega)}=\sup_{l\in\mathbb{N}}|\uu_{l}\mathscr{L}^{n}|(\Omega)<\infty\;\;\;\text{and}\;\;\;\sup_{l\in\mathbb{N}}\|\vv_{l}\|_{\lebe^{1}(\Omega)} = 
\sup_{l\in\mathbb{N}}|\mathbf{v}_{l}\mathscr{L}^{n}|(\Omega)<\infty.
\end{align}
For an arbitrary ball $\ball\Subset\Omega$ and $l\in\mathbb{N}$ sufficiently large, $\vv_{l}\,\mathscr{L}^{n}\mres\ball$ is the $\mathbb{A}$-gradient measure of $\uu_{l}$ on $\ball$. Because of the uniform bounds in \eqref{eq:boundossimo}, the $\mathbb{C}$-ellipticity of $\mathbb{A}$ implies the existence of a subsequence  $(\uu_{l_{m}})$ and some $\uu^{\omega}\in\bv^{\mathbb{A}}(\omega)$ such that $(\uu_{l_{m}})$ converges to $\uu^{\omega}$ in the weak*-sense on $\bv^{\mathbb{A}}(\omega)$, see \cite{GmRa}. This means that $\uu_{l_{m}}\to\uu^{\omega}$ strongly in $\lebe^{1}(\omega;V)$ and $\mathbb{A}\uu_{l_{m}}\stackrel{*}{\rightharpoonup}\mathbb{A}\uu^{\omega}$ in $\mathrm{RM}_{\mathrm{fin}}(\Omega;W)$. It is precisely at this point where $\mathbb{C}$-ellipticity enters, as otherwise this compactness result does not hold. Exhausting $\Omega$ by countably many balls and passing to the  diagonal sequence, we thus find an overall subsequence $(\uu_{l(m)})$ and some $\uu\in\bv_{\locc}^{\mathbb{A}}(\Omega)$ such that $(\uu_{l(m)})$ converges to $\uu$ in the weak*-sense on $\bv^{\mathbb{A}}(\omega)$ for every $\omega\Subset\Omega$. For an arbitrary $\bm{\varphi}\in\hold_{c}^{\infty}(\Omega;W)$, we then have with $\omega=\spt(\varphi)$:
\begin{align*}
-\int_{\Omega}\bm{\varphi}\cdot\dif\mathbb{A} u & = \int_{\omega}\uu\cdot\mathbb{A}^{*}\bm{\varphi}\dif x = \lim_{m\to\infty}\int_{\omega}\uu_{l(m)}\cdot\mathbb{A}^{*}\bm{\varphi}\dif x \\ 
& = -\lim_{m\to\infty}\int_{\omega}\bm{\varphi}\cdot\dif\A\uu_{l(m)} = -\lim_{m\to\infty}\int_{\omega}\bm{\varphi}\cdot \vv_{l(m)}\dif x \stackrel{\eqref{eq:convossimo}}{=} -\int_{\Omega}\bm{\varphi}\cdot\dif\bm{\mu}_{1}.
\end{align*}
By arbitrariness of $\omega\Subset\Omega$, it follows that $\A \uu=\bm{\mu}_{1}\in\mathrm{RM}_{\mathrm{fin}}(\Omega;W)$. Similarly, we see that $\uu\mathscr{L}^{n}\mres\Omega=\bm{\mu}_{0}$, and thus $\bm{\mu}=(\bm{\mu}_{0},\bm{\mu}_{1})=(\uu\mathscr{L}^{n}\mres\Omega,\A \uu)$. In conclusion, $(\bm{\mu}_{0},\bm{\mu}_{1})=\mathscr{J}_{2}\uu$, and so $\mathcal{Y}^{\bot}\subset\mathcal{J}(\bv^{\mathbb{A}}(\Omega))$ by \eqref{eq:traylorhoward}. The proof is complete.
\end{proof}

\bibliographystyle{alpha}
\bibliography{dk}

\end{document}